\PassOptionsToPackage{dvipsnames}{xcolor}

\documentclass[11pt]{amsart}

\usepackage{esint}
\usepackage{amssymb}
\usepackage{tikz}
\usepackage{graphicx}
\usepackage{amsrefs,enumitem}
\usepackage{xcolor}
\usepackage[margin=1.25in]{geometry}

\usepackage{hyperref}
\hypersetup{
  colorlinks=true,
  linkcolor=RoyalBlue,
  urlcolor=black,
  citecolor=black
}

\newtheorem{theorem}{Theorem}
\newtheorem{lemma}[theorem]{Lemma}
\newtheorem{corollary}[theorem]{Corollary}
\newtheorem{proposition}[theorem]{Proposition}
\theoremstyle{definition}

\newtheorem{remark}[theorem]{Remark}

\newtheorem{definition}[theorem]{Definition}

\newcommand{\eref}[1]{(\ref{e.#1})}
\newcommand{\tref}[1]{Theorem \ref{t.#1}}
\newcommand{\lref}[1]{Lemma \ref{l.#1}}
\newcommand{\pref}[1]{Proposition \ref{p.#1}}
\newcommand{\cref}[1]{Corollary \ref{c.#1}}
\newcommand{\fref}[1]{Figure \ref{f.#1}}
\newcommand{\sref}[1]{Section \ref{s.#1}}

\newcommand{\dref}[1]{Definition \ref{d.#1}}
\newcommand{\rref}[1]{Remark \ref{r.#1}}

\numberwithin{theorem}{section}
\numberwithin{equation}{section}

\newcommand{\N}{\mathbb{N}}

\newcommand{\R}{\mathbb{R}}

\newcommand{\grad}{\nabla}

\def\Xint#1{\mathchoice
{\XXint\displaystyle\textstyle{#1}}%
{\XXint\textstyle\scriptstyle{#1}}%
{\XXint\scriptstyle\scriptscriptstyle{#1}}%
{\XXint\scriptscriptstyle\scriptscriptstyle{#1}}%
\!\int}
\def\XXint#1#2#3{{\setbox0=\hbox{$#1{#2#3}{\int}$ }
\vcenter{\hbox{$#2#3$ }}\kern-.6\wd0}}

\def\dashint{\Xint-}

\newcommand{\ep}{\varepsilon}
\newcommand{\eps}{\ep}

\newcommand{\dist}{\operatorname{dist}}

\newcommand{\one}{\mathbf{1}}

\begin{document}

\title[Variational properties of Perron's extremal solutions]{Variational properties of Perron's extremal solutions in the Bernoulli one-phase problem}
\author{Farhan Abedin}
\address{Department of Mathematical Sciences, Lafayette College, Easton, PA 18042}
\email{abedinf@lafayette.edu}
\author{William M Feldman}
\address{Department of Mathematics, University of Utah, Salt Lake City, UT 84112}
\email{feldman@math.utah.edu}
\author{Kerrek Stinson}
\address{Department of Mathematics, University of Utah, Salt Lake City, UT 84112}
\email{kstinson@math.utah.edu}

\begin{abstract}
    We study the Perron extremal solutions of the stationary one-phase Bernoulli problem.  These solutions are important in several applications, but not much is known about the regularity of their free boundaries due to their non-variational construction.  We show that, in fact, Perron extremal solutions retain some variational features; specifically, they are  solutions in the sense of inner variations. We establish this property by showing that extremal solutions arise as infinite-time limits of monotone solutions of the parabolic Bernoulli problem. As a consequence, we are able to describe the structure of singularities for Perron extremal solutions in two dimensions: largest subsolutions have smooth free boundaries, while smallest supersolutions have free boundary points that are either smooth, or whose blow-up limits are two-plane wedges with equal slope.
\end{abstract}
\maketitle

\tableofcontents

\section{Introduction}\label{s.intro}

Let $U \subset \R^d, \, d\geq2,$ be a bounded domain and let $Q: \overline{U} \to (0,\infty)$ be a Lipschitz continuous coefficient field. We study non-negative solutions of the Bernoulli one-phase free boundary problem
\begin{equation}\label{e.bernoulli-basic}
     \begin{cases}
         \Delta u = 0 & \hbox{in } \{u>0\} \cap U, \\
         |\grad u| = Q(x) &\hbox{on } \partial \{u>0\} \cap U,
     \end{cases}
\end{equation}
which arises formally as the Euler--Lagrange condition for critical points of the Alt--Caffarelli energy
\begin{equation}\label{e.energy}
J_Q(u;U) := \int_U |\grad u|^2 + Q(x)^2 {\bf 1}_{\{u>0\}}\ dx.
\end{equation}

Several notions of weak solution to \eref{bernoulli-basic} have been introduced and studied in the literature over the past several decades.  Energy minimizers were first considered in the pioneering work of Alt and Caffarelli \cite{AltCaffarelli}. More recently, there has also been interest in understanding the behavior of non-minimizing variational solutions of \eref{bernoulli-basic}: we mention work on almost minimizers \cites{David-Toro-2015, DeSilvaSavin}, stable solutions \cite{Kamburov-Wang} and weaker notions like inner variational solutions \cites{Weiss, Jerison-Kamburov-I, KriventsovWeiss}. On the non-variational side, Caffarelli developed the viscosity solution framework for \eref{bernoulli-basic} and other free boundary problems in \cites{Caffarelli-Harnack-I, Caffarelli-Harnack-II, Caffarelli-Harnack-III}. In general, there is no simple hierarchy among these solution notions. We refer the reader to the  discussion in \cite{KriventsovWeiss}*{Subsection 3.1} for illuminating examples that illustrate the similarities and differences between various types of solution of \eref{bernoulli-basic}.

In several applications, certain special viscosity solutions are naturally singled out by the underlying physical model. These ``extremal'' solutions are obtained via Perron's method: they are, respectively, the smallest supersolution above a smooth strict subsolution, and the largest subsolution below a smooth strict supersolution. When the PDE problem enjoys uniqueness via comparison principle, the largest subsolution and smallest supersolution  coincide; this is not the case for \eqref{e.bernoulli-basic}, as illustrated by the numerical example in Figure~\ref{f.necklace}. Indeed, each extremal solution exhibits distinct free boundary behavior. Note that the extremal solutions trap all other viscosity solutions with the same boundary data, including the energy minimizer, which is also a viscosity solution (see \cite{VelichkovBook}*{Proposition 7.1}).

Perron's extremal solutions arise, for instance, as long-time limits of monotone free boundary evolutions, in quasi-static evolution problems \cite{FeldmanKimPozar2}, or in the homogenization theory of \eref{bernoulli-basic} where they play an essential role in understanding contact angle hysteresis in capillarity theory \cites{Caffarelli-Lee, Caffarelli-Mellet, Feldman-2021}.
 In some sense, these extremal solutions are more natural than energy minimizers, because they can be reliably accessed by gradient flows while the energy minimizer may sit in a complicated landscape of local minimizers and it is difficult to guarantee that a particular flow will reach the global minimum. 
 
 Despite their natural appearance in physical applications, Perron extremal solutions have been less studied in the context of regularity theory due to the lack of a suitable variational structure that allows tools such as monotonicity formulae to be used for understanding fine properties of the free boundary. The only regularity result for extremal solutions we are aware of is that of Orcan-Ekmekci \cite{OrcanEkmekci}, who establishes a non-degeneracy property of the largest subsolution in $d=2$. See also \rref{caffarelli-admissible} below for discussion of a slightly different notion of smallest supersolution studied by Caffarelli in \cite{Caffarelli-Harnack-III}. Of course, the regularity theory for general viscosity solutions \cites{CaffarelliSalsa, DeSilva} also applies to the Perron extremal solutions, but this theory mainly deals with a priori flat or cone-monotone (Lipschitz level sets) solutions. 

The purpose of this paper is to show that Perron's extremal solutions do carry variational structure: they are \emph{inner variational solutions} of \eref{bernoulli-basic} in the sense introduced by Weiss \cite{Weiss}, and they are minimal with respect to one-sided compact perturbations of the energy. In particular Weiss' monotonicity formula and resultant homogeneity of blow-ups for inner variational solutions apply to them. In two dimensions, this yields the first structure theorem for the free boundary of smallest supersolutions, and classical regularity for largest subsolutions, making these extremal solutions better behaved than an arbitrary viscosity solution of \eref{bernoulli-basic}.

Our main result is as follows. See Section \ref{s.concepts} for precise definitions.

\begin{theorem}\label{t.main}
    If $u$ is the smallest supersolution (resp.\ largest subsolution) of \eqref{e.bernoulli-basic} in $U$ above a smooth strict subsolution $g$ with $g|_{\partial U} >0$ (resp.\ below a smooth strict supersolution $g$, strict boundary positivity not needed) then
    \begin{enumerate}[label = (\roman*)]
        \item\label{part.visc} $u$ is a viscosity solution of \eqref{e.bernoulli-basic} in $U$,
        \item\label{part.inner} $u$ is an inner variational solution of \eqref{e.bernoulli-basic} in $U$, and
        \item\label{part.downward} $u$ is a downward (resp.\ upward) minimizer of $J_Q(\cdot ;U)$ locally around each free boundary point.
    \end{enumerate}
\end{theorem}

Although \tref{main} may appear to be concerned with somewhat technical solution notions, it has noteworthy consequences for the regularity theory and free boundary structure of Perron solutions. Combining the information given by the inner variational notion with the extremality properties of the Perron solutions, we can establish a much finer description of the free boundary regularity.  

\begin{figure}[t]
  \begin{minipage}[c]{0.32\textwidth}
    \centering
    \includegraphics[width=\linewidth,trim=16 0 2 24,clip]{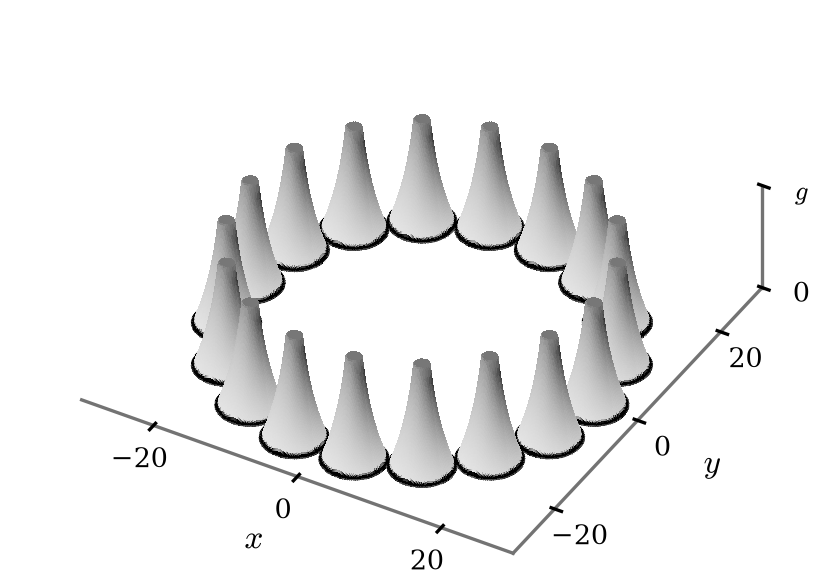}
  \end{minipage}
  \begin{minipage}[c]{0.32\textwidth}
    \centering
    \includegraphics[width=\linewidth,trim=16 0 2 24,clip]{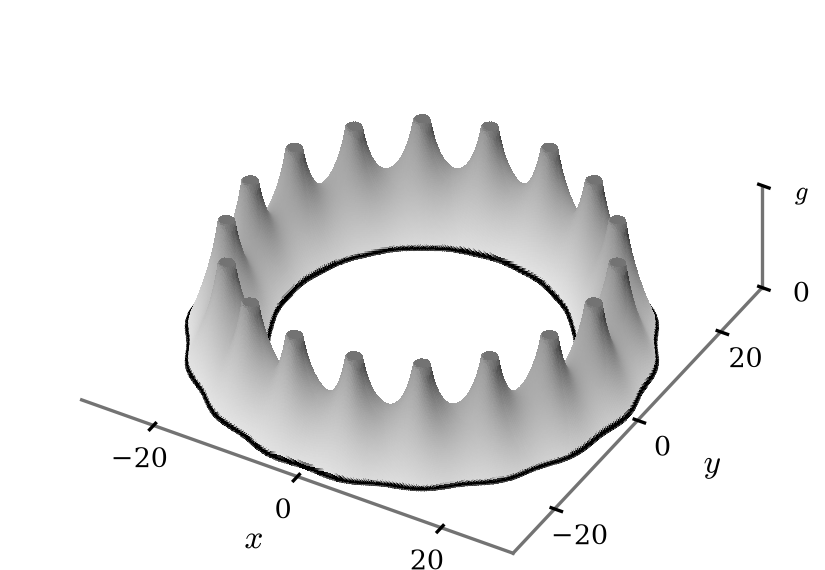}
  \end{minipage}
  \begin{minipage}[c]{0.32\textwidth}
    \centering
    \includegraphics[width=\linewidth,trim=16 0 2 24,clip]{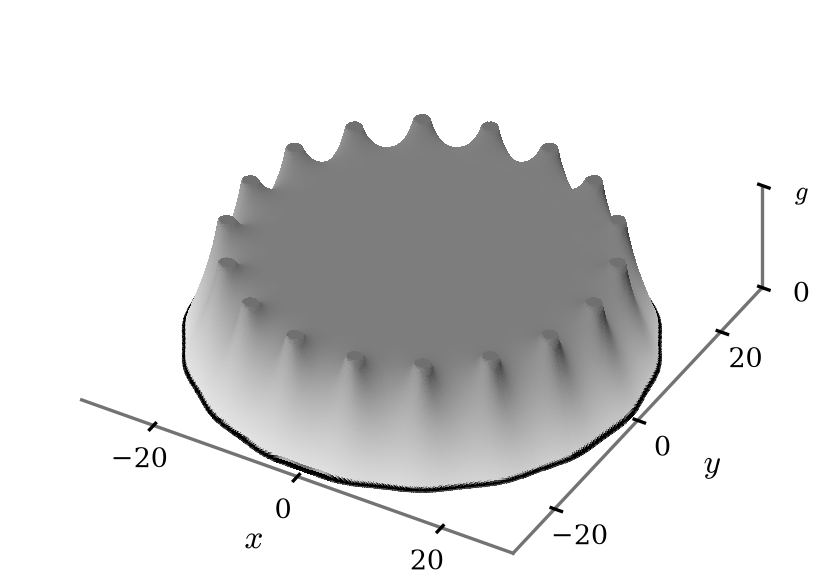}
  \end{minipage}
  \caption{A configuration where minimal supersolution (left), energy minimizer (middle), and maximal subsolution (right) are distinct. The domain is the exterior of a union of $N=18$ unit disks centered and evenly spaced around a circle of radius $r=24$, and the data $g$ is chosen so that the radial solutions centered at each disk just touch their two neighbors.  The salient features of this example are Lean verified with $N=100$ at \cite{GitHubRepo}.}
  \label{f.necklace}
\end{figure}

\begin{corollary}\label{c.2d-structure}
Extremal solutions in $\R^2$ satisfy the following improved regularity properties:
    \begin{enumerate}[label = (\roman*)]
         \item\label{part.smallest-reg} If $u$ is a smallest supersolution as in \tref{main} with $U \subset \R^2$, then the free boundary decomposes as the disjoint union
    \[\partial \{u>0\} \cap U = \textup{FB}_{\textup{reg}} \cup \textup{FB}_{\textup{TP}}.\]
    Here, $\textup{FB}_{\textup{reg}}$ is relatively open in $\partial \{u>0\}$ and, in a neighborhood of each point $x_0 \in \textup{FB}_{\textup{reg}}$, $u$ is a classical solution\footnote{i.e. the free boundary $\partial\{u>0\} \cap U$ is locally a $C^{1,\gamma}$ manifold near $x_0$ and the free boundary condition holds classically at all free boundary points near $x_0$.} of \eqref{e.bernoulli-basic}. At each point $x_0 \in \textup{FB}_{\textup{TP}}$, every subsequential blow-up limit of $u$ is of the form $Q(x_0)|x \cdot e|$ for some unit vector $e$.
    \item\label{part.largest-reg} If $u$ is a largest subsolution as in \tref{main} with $U \subset \R^2$, then $u$ is a classical solution of \eref{bernoulli-basic} in $U$.
    \end{enumerate}
\end{corollary}

We believe this is the first result describing the structure of the singular set (or lack thereof) for a sizable class of viscosity solutions of the Bernoulli free boundary problem, outside of flat or cone-monotone scenarios.

Since Corollary \ref{c.2d-structure} is a somewhat immediate consequence of \tref{main} combined with the well-known blow-up analysis for inner variational solutions via the Weiss monotonicity formula \cite{Weiss}, we provide the proof right away. 

\begin{proof}[Proof of \cref{2d-structure}]
We start with the arguments that are common to both cases. First, by \tref{main}\ref{part.inner}, $u$ is an inner variational solution with coefficient $Q \in C^{0,1} \subset C^\gamma$. Next, since $u$ is a viscosity solution by \tref{main}\ref{part.visc}, the interior Lipschitz estimates for viscosity solutions of \eqref{e.bernoulli-basic} (see \cite{CaffarelliSalsa}*{Lemma 11.19}) imply $u$ is locally Lipschitz in $U$. We can thus apply \lref{inner-compactness-stability} to find that the blow-up sequences $u_{x_0,r}(x):=r^{-1}u(x_0+rx)$ at $x_0 \in \partial \{u>0\}$ are pre-compact with respect to local uniform convergence, $H^1_{\textup{\rm loc}}$ convergence, and $L^1_{\textup{\rm loc}}$ convergence of the positivity sets. Lemma~\ref{l.super-stability} then implies all blow-up limits are viscosity supersolutions of \eref{bernoulli-basic} with constant slope $Q(x_0)$. 

The Weiss monotonicity formula \cite{Weiss} shows every blow-up limit of $u$ at $x_0$ is a $1$-homogeneous global inner variational solution with constant slope $Q(x_0)$. In the plane, the $1$-homogeneous global inner variational solutions are classified (see for instance \cite{Jerison-Kamburov-I}*{Section 5}): they are precisely the half-plane solutions $Q(x_0)(x\cdot e)_+$ with $\chi =  {\bf 1}_{\{x \cdot e >0\}}$, the two-plane solutions $\phi_\alpha:=\alpha|x\cdot e|$ with $\chi_\alpha = {\bf 1}_{\R^2}$ for $\alpha \in [0,\infty)$, and the trivial solution $\phi_0\equiv0$ with $\chi\equiv0$.
Since, for $\alpha > Q(x_0)$, the two-plane solution $\phi_\alpha$ is not a viscosity supersolution, we find that $\alpha \in [0,Q(x_0)]$. The degenerate blow-up $\phi_0$ is ruled out in each case either by the non-degeneracy properties of downward minimizers, \tref{downward-min-non-degen}, or largest subsolutions in two dimensions, \tref{largest-sub-non-degen}.

Next, we invoke minimality properties of the blow-up limits. By \tref{main} \ref{part.downward}, $u$ is energy minimizing with respect to upward or downward (in the respective cases) perturbations in a sufficiently small neighborhood of $x_0$. This property is stable with respect to the blow-up limit at $x_0$, see \cref{blow-up-directional} below, so we conclude that every blow-up limit is upward or downward minimizing (in the respective cases) for $J_{Q(x_0)}$.

Let $\textup{FB}_{\textup{reg}}$ be the set of free boundary points at which some blow-up is a half-plane solution. Using the inner variational theory of \cite{KriventsovWeiss}, we can show that the positivity sets converge in Hausdorff distance as well, see \lref{one-plane-limit} below (\cref{blow-up-directional}\ref{part.blow-up-directional-1} also works for the smallest supersolution case). Therefore, the $\ep$-regularity theory of Caffarelli \cites{Caffarelli-Harnack-I, Caffarelli-Harnack-II} (see also \cite{DeSilva}) applies at each point of $\textup{FB}_{\textup{reg}}$ and yields a neighborhood in which the free boundary is a $C^{1,\gamma}$ graph and $u$ is a classical solution. In particular, $\textup{FB}_{\textup{reg}}$ is relatively open in $\partial \{u>0\}$. Finally, we let $\textup{FB}_{\textup{TP}} = \partial \{u>0\} \setminus \textup{FB}_{\textup{reg}}$, which is a relatively closed set in $\partial \{u>0\}$ where all blow-ups of $u$ are two-plane solutions $\phi_{\alpha}$ for some $\alpha \geq 0$.

We now characterize the behavior of $\textup{FB}_{\textup{TP}}$ in each case.

Part \ref{part.smallest-reg}. For $0<\alpha < Q(x_0)$, $\phi_{\alpha}$ is not energy minimizing with respect to downward perturbations (see Proposition~\ref{p.two-plane-minimizing-props}). Since $\alpha > Q(x_0)$ was ruled out above using the viscosity supersolution property, every blow-up limit at a point $x_0 \in \textup{FB}_{\textup{TP}}$ is of the form $Q(x_0)|x \cdot e|$ for some $e \in S^{1}$.

Part \ref{part.largest-reg}. The two-plane solutions $\phi_\alpha$ for $\alpha>0$ are not upward minimizing (see Proposition~\ref{p.two-plane-minimizing-props}).  Thus $\textup{FB}_{\textup{TP}} = \emptyset$ and $u$ is a classical solution.
\end{proof}

Note that the result on minimal supersolutions only uses the viscosity supersolution property, the inner variational property, and the downward energy minimality property. The result on maximal subsolutions, on the other hand, really uses the largest subsolution property via the non-degeneracy result \tref{largest-sub-non-degen}.

\begin{remark}\label{r.caffarelli-admissible}
In \cite{Caffarelli-Harnack-III}, Caffarelli considers a notion of ``restricted'' or ``admissible'' supersolution \cite{CaffarelliSalsa} which satisfies the additional property of being outer regular at all free boundary points. The infimum of such supersolutions above an inner-regular subsolution gives rise to what Caffarelli and Salsa call a ``minimal viscosity solution'' in \cite{CaffarelliSalsa}*{Chapter 6}. These are different from the smallest supersolutions considered in this paper; in particular, they cannot have two-plane blow-up at the free boundary. The results of \cite{Caffarelli-Harnack-III} (see also the exposition in \cite{CaffarelliSalsa}*{Chapter 6}) show that the free boundary for such solutions is smooth, except on a set of zero $(d-1)$-Hausdorff measure. We mention a conjecture from \cite{Jerison-Kamburov-II} which posits that Caffarelli's minimal viscosity solutions are classical in dimension two. As illustrated by Figure~\ref{f.necklace}, the free boundary need not be classical for the minimal supersolution considered in our work.
\end{remark}

\subsection{Strategy of the proof}\label{ss.plan}

Part \ref{part.visc} of \tref{main} is a standard well-known consequence of the Perron method construction, but we also establish this independently within the proof of part \ref{part.inner} of \tref{main} for the class of sub/supersolution data $g$ considered here. For part \ref{part.downward} of \tref{main}, we take some inspiration from the work of \cite{DePLaux} and the connection between one-sided minimality and monotonicity in mean curvature flow, though our arguments are different in nature: We solve the obstacle problem for $J_Q$ with $u$ itself as an upper (resp. lower) obstacle, and invoke the extremality of $u$ to conclude that the minimizer is $u$.

The bulk of the manuscript is devoted to proving part \ref{part.inner} of \tref{main}, which establishes the inner variational property for Perron's extremal solutions. Our strategy is to realize the extremal solution via an approximation scheme which preserves both viscosity and variational structures. This is quite delicate, as naive approximation may not lead to the desired solution of \eqref{e.bernoulli-basic}. Indeed, \eqref{e.bernoulli-basic} suffers from a lack of uniqueness and well-known approximation procedures for generating inner variational solutions and viscosity solutions are not guaranteed to yield the extremal solutions. On the other hand, the extremal solutions are somewhat special compared to arbitrary viscosity solutions because they arise as the long-time limit of time-monotone solutions of the \emph{parabolic Bernoulli problem}
\[
     \begin{cases}
         \partial_t \mathfrak u =\Delta \mathfrak u & \hbox{in } \{\mathfrak u>0\} \cap U_\infty, \\
         |\grad \mathfrak u| = Q(x) &\hbox{on } \partial \{\mathfrak u>0\} \cap U_\infty, \\
         \mathfrak u = {g}_+ &\hbox{on } \partial_PU_\infty,
     \end{cases}
\]
where $U_{\infty} := U \times (0,\infty)$. This provides a clue as to what the correct approximation scheme should be. We first obtain a variational and viscosity solution of the parabolic Bernoulli problem by approximating with a well-known semilinear parabolic problem (see \eqref{e.bernoulli-parabolic-semilinear}). Then we take the $t\to \infty$ limit of the parabolic Bernoulli flow, which we show preserves both the inner variational and viscosity properties. Finally, we show via monotonicity and comparison principle that the limit of the flow must be the Perron extremal solution for the given sub/supersolution data, thus establishing the inner variational property. For both limit cases, the viscosity properties are especially subtle to verify due to the possibility of degeneracy; limits of viscosity subsolutions may fail to be viscosity subsolutions.  Several new ideas are introduced throughout the paper to deal with this issue in the various limits, see Section~\ref{subsec:paraVisc} and \sref{inner-variational-stationary} for further discussion.

\subsection{Notation and conventions}\label{ss.notation}

We write $B_r(x)$ for the open ball, ${\bf 1}_E$ for the indicator of a set $E$, $s_+ := \max(s,0)$, and $\dist(\cdot,\cdot)$ for the Euclidean distance. For space-time sets we use $U_\infty := U \times (0,\infty)$, the parabolic cylinders $\mathcal{Q}_r(x,t) := B_r(x) \times (t - r^2, t]$, and the parabolic boundary $\partial_P$.

The following conditions are in force throughout the paper.
\begin{enumerate}[label = (\roman*)]
    \item $U \subset \R^d$, $d \geq 2$, is open, bounded, and connected, and $\partial U$ is $C^2$.
    \item $Q \in C^{0,1}(\overline{U})$ and there are constants $0 < Q_{\min} \leq Q_{\max}$ with
    \[Q_{\min} \leq Q(x) \leq Q_{\max} \  \hbox{for all } x \in \overline{U}.\]
    We write $N_Q := \|\grad(Q^2)\|_{L^\infty(U)}$.
\end{enumerate}

The letters $C$ and $c$ denote positive constants which may change from line to line, we will indicate the dependence where needed. We call a constant universal when it only depends on the parameters fixed throughout the paper: $d$, the $C^2$ regularity of the domain $U$, the $C^{0,1}$ regularity and positivity of $Q$, and the $C^2$ regularity and strict sub/supersolution hypothesis of $g$.

\section{Solution concepts}\label{s.concepts}
We introduce the three different notions of solution that are united by Theorem \ref{t.main} and derive some useful properties based on these definitions.

\subsection{Viscosity solutions}\label{s.viscosity-solns-prelim}

We first define viscosity solutions of \eref{bernoulli-basic}, following \cite{Kim03}. Say that a function $\varphi$ touches $u$ from above in a set $E$ at a point $x_0 \in E$ if $0 = \max_{E}(u-\varphi) = u(x_0) - \varphi(x_0)$.  Touching from below is defined analogously.

\begin{definition}\label{d.stationary-viscosity}
Let $u \in C(U)$ be non-negative.
\begin{enumerate}[label = (\roman*)]
\item $u$ is a \emph{viscosity supersolution} of \eref{bernoulli-basic} in $U$ if, whenever $\varphi \in C^\infty(U)$ touches $u$ from below at $x \in U$, then either
    \[\Delta \varphi(x) \leq 0, \  \hbox{or} \  \varphi(x) = 0 \ \hbox{ and } \ |\grad \varphi(x)| \leq Q(x).\]
\item $u$ is a \emph{viscosity subsolution} of \eref{bernoulli-basic} in $U$ if, whenever $\varphi \in C^\infty(U)$ and $\varphi$ touches $u$ from above in $\overline{\{u>0\}} \cap U$ at $x$, then either
    \[\Delta \varphi(x) \geq 0, \  \hbox{or} \  \varphi(x) = 0 \ \hbox{ and } \ |\grad \varphi(x)| \geq Q(x).\]
\item $u$ is a \emph{viscosity solution} if it is both a viscosity supersolution and a viscosity subsolution.
\end{enumerate}
\end{definition}

\begin{remark}\label{r.strict-touching}
It is equivalent to just check the super and subsolution conditions for the case of strict touching since, if $\varphi$ touches $u$ from above/below at $x_0$ then $\varphi\pm |x-x_0|^4$ also touches from above/below strictly at $x_0$ with the same first and second order derivatives at $x_0$.  We will frequently make use of this trick below.
\end{remark}

We also make precise the meaning of smooth strict subsolution/supersolution used in the statement of the main theorem.
\begin{definition}\label{d.strict-sub}
A function $g \in C^2(\overline{U})$ is a smooth strict subsolution (resp. supersolution) of \eref{bernoulli-basic} if there are constants $\frac{1}{2}\geq a_0, \delta_0 >0$ such that
\begin{enumerate}[label = (\roman*)]
    \item $\Delta g \geq \delta_0$ (resp. $\Delta g \leq - \delta_0$) in $\{g > -a_0\}$, and
    \item $|\grad g(x)|^2 \geq (1+ \delta_0)Q(x)^2 $ (resp. $|\grad g(x)|^2 \leq (1-\delta_0)Q(x)^2 $) for all \\ {$x \in \{|g| \leq a_0\}$}.
\end{enumerate}
\end{definition}
Note that for $|a|\leq a_0/2$, the shifts $g+a$ are also strict sub/supersolutions with parameters $a_0$ replaced by $a_0/2$ and the same $\delta_0$.

The viscosity supersolution property is stable under uniform convergence (see, for instance  \cite{FeldmanPozar2025}*{Lemma 3.1}). In contrast, the stability of the viscosity subsolution property at the free boundary is subtle due to potential degeneracy in the limit. For example, when $Q\equiv 1$, $u_j(x) = \frac{1}{2}(x_d)_+ + \frac{1}{j}(-x_d)_+$ is a sequence of viscosity subsolutions converging locally uniformly to a limit which is not a viscosity subsolution in the standard sense. We refer to \cite{FeldmanKim} for further discussion.  This issue is a significant challenge in the proof of \tref{main} and shapes the details of the proofs in \sref{conclusion}, resulting in the noticeable asymmetry between the increasing and decreasing cases therein. 

The degenerate example above motivates the following relaxed notion of subsolution which is stable under limits.

\begin{definition}\label{d.relaxed-sub-stationary}
    Let $u \in C(U)$ be non-negative. Let $E$ be closed in $\overline{U}$ and containing $\{u>0\}$. The pair $(u,E)$ is a \emph{relaxed subsolution} of \eref{bernoulli-basic} if whenever $\varphi \in C^\infty(U)$ touches $u$ from above in $E \cap U$ at $x$, either $\Delta \varphi(x) \geq 0$, or $\varphi(x) = 0$ and $|\grad \varphi(x)| \geq Q(x)$.

    We say $u$ is a \emph{relaxed solution} of \eref{bernoulli-basic} if it is a viscosity supersolution and a relaxed subsolution.
\end{definition}

Note that a relaxed subsolution is a viscosity subsolution if and only if we can let $E = \overline{\{u > 0\}}\cap U$.

We recall the stability of the viscosity supersolution and relaxed subsolution properties under locally uniform limits, referring to \cite{FeldmanPozar2025}*{Lemma 3.1} and \cite{FeldmanKim}*{Theorem 2.8} for the proofs.  We use the notion of Kuratowski / upper half-relaxed limit of a sequence of sets $E_k$
\[{\limsup_{k \to \infty}}^*E_k:=\{x: \exists \hbox{ a subsequence $k_j \to \infty$ and $x_{k_j} \in E_{k_j}$ such that $x_{k_j} \to x$}\}.\]
\begin{lemma}\label{l.super-stability}
   Let $U$ be an open set, $Q$ be continuous and positive on $U$, and $u_k \in C(U)$ with $u_k \to u$ locally uniformly in $U$.
    \begin{enumerate}[label = \textup{(\roman*)}]
        \item\label{part.stability-super} If each $u_k$ is a supersolution of \eref{bernoulli-basic} in $U$, then so is $u$.
        \item\label{part.stability-pair} If each $(u_k, E_k)$ is a relaxed subsolution of \eref{bernoulli-basic} in $U$, then so is $(u, E^*)$, where $E^* = {\limsup_{k \to \infty}^*}E_k$. Furthermore, if $u_k$ is a monotone increasing sequence of subsolutions, i.e. relaxed subsolutions with $E_k := \overline{\{u_k>0\}}$, then $E^* \subset \overline{\{u>0\}}$ and $u$ is a subsolution.
    \end{enumerate}
\end{lemma}

\subsection{Perron's method and extremal viscosity solutions}\label{s.perron}

We now make precise the meaning of the Perron extremal solutions appearing in \tref{main}.

For the largest subsolution case, we let $g$ be a smooth strict supersolution of \eqref{e.bernoulli-basic} in the sense of \dref{strict-sub}. We define
\begin{equation}\label{eqn:S^g}
    \mathcal{S}^g := \Big\{v \in C(\overline U):\ v \hbox{ is a subsolution of \eref{bernoulli-basic} in } U, \ v \leq g_+ \hbox{ on } \overline U\Big\},
\end{equation}
and define the \emph{largest subsolution below $g$} as 
\begin{equation}\label{e.perron-def-sub}
    u(x) := \sup_{v \in \mathcal{S}^g} v(x).
\end{equation}

For the smallest supersolution case, we let $g$ be a smooth strict subsolution of \eqref{e.bernoulli-basic} in the sense of \dref{strict-sub}. We define
\begin{equation}\label{eqn:S_g}
    \mathcal{S}_g := \Big\{v \in C(\overline U):\ v \hbox{ is a supersolution of \eref{bernoulli-basic} in } U, \ v \geq g_+ \hbox{ on } \overline U\Big\},
\end{equation}
and define the \emph{smallest supersolution above $g$} as
\begin{equation}\label{e.perron-def-super}
u(x) := \inf_{v \in \mathcal{S}_g} v(x).
\end{equation}

The classes are trivially nonempty, as $0 \in \mathcal{S}^g$ and $\max g_+ \in \mathcal{S}_g$. This also immediately produces the maximum bounds $0 \leq u \leq \max g_+$ in both cases. A priori the extremal solutions as defined above are only LSC or USC as a supremum or infimum of continuous functions. However, one can show that both are in fact continuous viscosity solutions of \eref{bernoulli-basic}.

\begin{lemma}\label{l.extremal-viscosity}
Let $u$ be the smallest supersolution above $g$ as defined in \eqref{e.perron-def-super} (resp.\ the largest subsolution below $g$ as in \eqref{e.perron-def-sub}). Then $u \in C(\overline{U})$ is a viscosity solution of \eref{bernoulli-basic} in $U$ with $u = g_+$ on $\partial U$.
\end{lemma}

A direct proof is possible following the approach outlined in \cite{CaffarelliSalsa}*{Chapter 6}; however, we can also derive \lref{extremal-viscosity} via an independent method as a corollary of our arguments with the parabolic flow. We thus skip the proof for now, referring to Section~\ref{s.conclusion}, Propositions \ref{p.inner-smallest} and \ref{p.inner-largest}, for details.

Extremal solutions created by Perron's method satisfy the following local properties (respectively).  We state them here because these weaker properties are all we need to establish the conclusion \tref{main}\ref{part.downward}.

\begin{definition}\label{d.largest-smallest}
Let $u \in C(U)$ be non-negative.
\begin{enumerate}[label = (\roman*)]
\item Call $u$ a \emph{(local) smallest supersolution} of \eref{bernoulli-basic} in $U$ if $u$ is a viscosity supersolution of \eref{bernoulli-basic} in $U$ and, for every ball $B \subset\subset U$ and every viscosity supersolution $v \in C(U)$ of \eref{bernoulli-basic} in $U$ with $v = u$ on $U \setminus B$, it holds that $v \geq u$ in $U$.
\item Call $u$ a \emph{(local) largest subsolution} of \eref{bernoulli-basic} in $U$ if $u$ is a viscosity subsolution of \eref{bernoulli-basic} in $U$ and, for every ball $B \subset\subset U$ and every viscosity subsolution $v \in C(U)$ of \eref{bernoulli-basic} in $U$ with $v = u$ on $U \setminus B$, it holds that $v \leq u$ in $U$.
\end{enumerate}
\end{definition}

\begin{lemma}\label{l.local-minimality}
    If $u \in C(\overline{U})$ is the Perron smallest supersolution (resp. largest subsolution) of \eqref{e.bernoulli-basic} in $U$ above a smooth strict subsolution $g$ (resp. smooth strict supersolution $g$) then for every $x \in \partial \{u>0\} \cap U$ there is a ball $B_r(x)$ so that $u$ is a local smallest supersolution in $B_r(x)$ (resp. local largest subsolution).
\end{lemma}

The proof can be found at the end of \sref{downward}.

\subsubsection{Non-degeneracy properties of extremal solutions}\label{s.non-degen} Finally we recall some non-degeneracy properties of smallest supersolutions and largest subsolutions. Smallest supersolutions satisfy a strong non-degeneracy property in all dimensions.

\begin{theorem}[Smallest supersolution non-degeneracy]\label{t.smallest-super-non-degen}
    If $u$ is a (local) smallest supersolution in $B_1(0)$ and $0 \in \overline{\{u>0\}}$ then for all $0 < r \leq 1$
    \[ \sup_{\partial B_r(0)} u \geq c r.\]
\end{theorem}

Largest subsolutions are only known to satisfy a weaker non-degeneracy property when $d=2$.  

\begin{theorem}[Largest subsolution non-degeneracy \cite{OrcanEkmekci}]\label{t.largest-sub-non-degen}
    If $d=2$ and $u$ is a (local) largest subsolution in $B_1(0)$ and $0 \in \partial \{u>0\}$ then for all $0 < r \leq 1$
    \[ \sup_{\partial B_r(0)} u \geq c r.\]
\end{theorem}

Note that this is somewhat weaker than the non-degeneracy property of minimal supersolutions since it requires $0 \in \partial \{u>0\}$ instead of $0 \in \overline{\{u>0\}}$. In particular, it allows for degeneration of the type $\frac{1}{2}(x_d)_++\frac{1}{n}(-x_d)_+ \to \frac{1}{2}(x_d)_+$ as $n \to \infty$, which is a sequence of subsolutions of \eqref{e.bernoulli-basic} that are non-degenerate in the sense of \tref{largest-sub-non-degen}, but converge uniformly to a non-subsolution. Of course, these examples are not largest subsolutions.

The proof of \tref{smallest-super-non-degen} is a standard contradiction argument: if the supremum is too small, one can take the minimum with a standard radially symmetric supersolution barrier in $B_r(0)$ to create a smaller supersolution which is $0$ in $B_{r/2}(0)$. The proof of \tref{largest-sub-non-degen}, originally due to  Orcan-Ekmekci \cite{OrcanEkmekci}, is a bit more involved and makes essential use of the hypothesis $d=2$; the same proof in arbitrary dimensions only yields a lower bound of order $r^{d-1}$. We present both proofs in Appendix~\ref{s.smallest-super-nondegen} and \ref{s.largest-sub-nondegen} for completeness.

\subsection{Inner variational solutions}\label{s.inner-variational-stationary}
The following notion of solution, studied by Kriventsov and Weiss \cite{KriventsovWeiss}, encodes criticality of the energy \eref{energy} with respect to domain deformations.

\begin{definition}\label{d.stationary-inner}
Given an open set $U \subset \R^d$, a pair of functions $(u, \chi)$ with $u : U \to [0, \infty)$ and $\chi : U \to \{0, 1\}$ is an \emph{inner variational solution} of \eref{bernoulli-basic} if the following properties hold:
\begin{enumerate}[label = (\roman*)]
	\item $u \in C^{0,1}_{\textup{\rm loc}}(U)\cap C^2(\{u > 0\})$ with $\Delta u = 0$ in $\{u>0\}$.
	\item $\chi$ is Borel measurable.
	\item ${\bf 1}_{\{u > 0\}} \leq \chi$ $\mathcal{L}^d$-a.e. on $U$.
	\item For each vector field $\xi \in C^{0, 1}_c(U; \R^d)$,
	\begin{equation}\label{e.stationary-inner-var}
		\int_U \Big[\big(|\nabla u|^2 + Q^2\chi\big)  \grad \cdot \xi - 2 \nabla u \cdot D\xi \nabla u + \grad(Q^2)\cdot \xi  \chi \Big]  dx = 0.
	\end{equation}
\end{enumerate}
Overloading the terminology, we also say $u : U \to [0, \infty)$ is an \emph{inner variational solution} if there exists a $\chi$ so that $(u,\chi)$ is an inner variational solution.
\end{definition}

A recent significant result, due to Kriventsov and Weiss \cite{KriventsovWeiss}*{Theorem 1.2 (i)}, shows that if $(u,\chi)$ is an inner variational solution in a connected domain $U$, then either $u \equiv 0$ and $\chi = {\bf 1}_{U}$, or $\chi = {\bf 1}_{\{u>0\}}$ almost everywhere. In particular, whenever $u$ is a non-trivial inner variational solution with some $\chi$, it is also an inner variational solution with ${\bf 1}_{\{u>0\}}$. We will avoid using this powerful result for the smallest supersolution case, but we do need it, or at least a consequence, in the largest subsolution case.

As explained in \cite{KriventsovWeiss}, inner variational solutions are not necessarily viscosity \emph{super}solutions: the two-plane functions $\alpha |x \cdot e|$ are inner variational solutions (with $\chi \equiv 1$, $Q \equiv 1$) for every $\alpha > 0$, but fail the supersolution condition when $\alpha > 1$; likewise $(|v|, {\bf 1}_U)$ is an inner variational solution whenever $v$ is harmonic. The subsolution condition, however, does hold as we show next. 

\begin{lemma}\label{l.inner-variational-implies-subsoln}
    If $(u,\chi)$ is an inner variational solution of \eqref{e.bernoulli-basic} in $U$ then $u$ is a viscosity subsolution of \eqref{e.bernoulli-basic} in $U$.
\end{lemma}

This provides a very useful link between the viscosity and inner variational theories. Uniform stability is lacking for pure viscosity subsolutions, as discussed in \sref{viscosity-solns-prelim}, but subsolution stability can be recovered using the inner variational property. We exploit this idea below in \sref{long-time-limit}.

At a high level, the idea is that a smooth test function cannot touch $u$ from above at two-plane points and for touchings within the singular set the test function will be subharmonic. The proof below is short, but relies heavily on the substantial contributions of \cite{KriventsovWeiss}. We remark that the results of \cite{KriventsovWeiss} are stated for $Q \equiv 1$, but naturally carry over to the case of $Q \in C^{0,1}(U)$. 

\begin{proof}[Proof of \lref{inner-variational-implies-subsoln}] 

Define the ``singular set'' of $u$ to be the set of free boundary points of highest density,
\[\Sigma^H := \Big\{x \in \partial \{u>0\}: \lim_{r \to 0} \frac{|\{\chi = 1\} \cap B_r(x)|}{|B_r|} = 1\Big\}.\]
By \cite{KriventsovWeiss}*{Proposition 3.5} the set $\Sigma^H$ is closed in $U$ and by \cite{KriventsovWeiss}*{Lemma 8.3}, $u$ is a viscosity solution of \eqref{e.bernoulli-basic} in the open set $U \setminus \Sigma^H$. By  \cite{KriventsovWeiss}*{Theorem 1.2 (i)}, either $u \equiv 0$ in $U$ or $\chi = {\bf 1}_{\{u>0\}}$ a.e. If $u \equiv 0$ the subsolution condition is vacuous, so assume $\chi = {\bf 1}_{\{u>0\}}$. It remains to check the subsolution condition at touching points $x_0 \in \Sigma^H$.

Suppose $\varphi \in C^\infty$ is such that $\varphi$ touches $u$ from above in $\overline{\{u>0\}}$ at $x_0 \in \Sigma^H$; since $\Sigma^H \subset \partial\{u>0\}$ we have $u(x_0) = 0 = \varphi(x_0)$. We claim that $\grad \varphi (x_0) = 0$ and $D^2 \varphi(x_0) \geq 0$; in particular $\Delta\varphi(x_0) \geq 0$, which verifies the subsolution condition. Since $0 \leq u \leq \varphi_+$, we have $\{\varphi < 0\} \subset \{u = 0\}$ in a neighborhood of $x_0$, and since $\chi = {\bf 1}_{\{u>0\}}$, the definition of $\Sigma^H$ gives
  \begin{equation}\label{e.density-zero}
  {\liminf_{r \to 0}} \frac{|\{\varphi < 0\} \cap B_r(x_0)|}{|B_r|} \leq \lim_{r \to 0} \frac{|\{u = 0\} \cap B_r(x_0)|}{|B_r|} = 0.
  \end{equation}
If $\grad\varphi(x_0) \neq 0$, then by the Taylor expansion $\varphi(x) = \grad\varphi(x_0)\cdot(x - x_0) + O(|x-x_0|^2)$,
it is clear the density of $\{\varphi < 0\}$ at $x_0$ is $\tfrac12$, which contradicts \eref{density-zero}. 
Hence $\grad\varphi(x_0) = 0$. If $D^2\varphi(x_0)$ had a negative eigenvalue, say $D^2\varphi(x_0)e\cdot e =: -2\mu < 0$ with $|e| = 1$, then by the second-order Taylor expansion it holds that $\varphi(x) \leq -\mu |x-x_0|^2 + o(|x-x_0|^2)$ on the cone $\{x : |(x - x_0) - ((x-x_0)\cdot e)e| \leq \theta |x-x_0|\}$ for a suitable $\theta = \theta(D^2\varphi(x_0), \mu) > 0$, so $\{\varphi < 0\}$ has positive lower density at $x_0$, again contradicting \eref{density-zero}. This proves the claim.
\end{proof}

\subsubsection{Regularity and stability under limits}\label{s.inner-regularity-stability}
Inner variational solutions satisfy an elementary perimeter bound which formally follows from integrating $\Delta u$ over the positive set, integrating by parts, and using $|\grad u| \geq Q_{\min}$ on the free boundary.  The following is a small modification of \cite{KriventsovWeiss}*{Lemma 3.3}.
\begin{lemma}[Perimeter bound]\label{l.perimeter-bound}
Suppose that $(u,\chi)$ is an inner variational solution in $B_{2r}(0)$ with $\|\grad u\|_{L^\infty(B_{2r})} = L$.  Then $\chi \in BV(B_r(0))$ and
\[ \int_{B_r(0)} |\grad \chi| \leq C(L+r\|\grad Q^2\|_{L^\infty(B_{2r})})r^{d-1}.\]
\end{lemma}

Next we state a compactness and stability property of inner variational solutions.

\begin{lemma}[compactness for inner variational solutions]\label{l.inner-compactness-stability}
Let $(u_k, \chi_k)$ be inner variational solutions of \eref{bernoulli-basic} in $U$ such that, for every $V \subset\subset U$, 
it holds that $\sup_k (\|u_k\|_{L^\infty(V)}+\|\grad u_k\|_{L^\infty(V)} )<\infty$. Then there is a subsequence $k_j$ such that
\[u_{k_j} \to u \ \hbox{ locally uniformly in } U, \  \grad u_{k_j} \to \grad u \ \hbox{ in } L^2_{\rm loc}(U), \  \chi_{k_j} \to \chi \ \hbox{ in } L^1_{\rm loc}(U),\]
and $(u,\chi)$ is an inner variational solution of \eref{bernoulli-basic} in $U$.
\end{lemma}

The proof of this result can be found in the first paragraph of the proof of \cite{KriventsovWeiss}*{Theorem 9.3}. See also \cite{Jerison-Kamburov-I}*{Proposition 4.2}.

\subsection{Directional minimality}\label{s.directional-minimality}
The final property in our main theorem is a one-sided minimality of the energy.

\begin{definition}\label{d.directional}
 We say $u \in H^1_{\rm loc}(U) \cap C(U)$ non-negative is an \emph{upward (resp.\ downward) minimizer} of $J_Q(\cdot;U)$, see \eqref{e.energy}, if
\begin{enumerate}[label = (\roman*)]
\item the set $\{u>0\}$ is open and $u$ is harmonic in $\{u>0\}$, and
\item for every ball $B \subset\subset U$ and every $v \in H^1_{\rm loc}(U)$ with $v \geq u$ (resp.\ $0 \leq v \leq u$) and $u-v \in H^1_0(B)$,
\begin{equation}\label{eqn:directionalMinimality}
J_Q(u;B) \leq J_Q(v;B).
\end{equation}
\end{enumerate}
\end{definition}

The one-sided minimality properties are stable under limits.

\begin{lemma}\label{l.directionalStable}
Suppose $u_n\in C^{0,1}_{\textup{\rm loc}}(U)$ are such that 
\begin{equation}\label{eqn:locallocalLip}
\sup_{n}\|\nabla u_n\|_{L^\infty(V)}\leq L_V 
\end{equation}
    for a constant $L_V > 0$ and every $V\subset\subset U,$ and $u_n \to u$ locally uniformly. Suppose that $Q_n \in C(\overline{U})$, $Q_n \to Q$ locally uniformly in $U$, and $Q_{\min} \leq Q_n,Q \leq Q_{\max}$. Then if $u_n$ are downward (resp. upward) minimizers of $J_{Q_n}$ in $U$, then $u$ is a downward (resp. upward) minimizer of $J_{Q}$ in $U$.
\end{lemma}

We provide the proof in Appendix~\ref{s.directionalStable-proof}. The result seems to be new, the proof is relatively standard but involves some interesting tricks.   

Finally we note that downward minimizers satisfy a non-degeneracy property.
\begin{theorem}[Downward minimizer non-degeneracy, see \cite{VelichkovBook}*{Lemma 4.4}]\label{t.downward-min-non-degen}
    If $u$ is a downward minimizer in $B_1(0)$ and $0 \in \overline{\{u>0\}}$ then for all $0 < r < 1$
    \[ \sup_{\partial B_r(0)} u \geq c r.\]
\end{theorem}

\subsection{Blow-ups}\label{s.blow-ups} In this section we show some results about blow-ups in the context of extremal solutions, directional minimizers, and inner variational solutions.

First, as a corollary of Lemma~\ref{l.directionalStable}, we show that directional minimality is stable under blow-up limits.

\begin{corollary}\label{c.blow-up-directional}
    Let $u \in H^1(U) \cap C^{0,1}_{\textup{\rm loc}}(U)$. For any $x_0  \in \partial \{u>0\} \cap U$ the blow-up sequence for $0 < r < \frac{1}{2}d(x_0,\partial U)$
    \[u_{x_0,r}(x):= r^{-1}u(x_0+rx),\]
    is pre-compact in weak $H^1_{\textup{\rm loc}}(\R^d)$ and the topology of local uniform convergence. Let $u^0$ be a subsequential limit along a subsequence $r_j \to 0$ then:
        \begin{enumerate}[label=(\roman*)]
        \item \label{part.blow-up-directional-1} If $u$ is a downward minimizer of $J_{Q}$ in $U$ then $u^0$ is a downward minimizer of $J_{Q(x_0)}$ in $\R^d$. Furthermore $\overline{\{u_{x_0,r_j}>0\}}$ converges to $\overline{\{u^0>0\}}$ locally in Hausdorff distance.
    \item If $u$ is an upward minimizer of $J_{Q}$ in $U$ then $u^0$ is an upward minimizer of $J_{Q(x_0)}$ in $\R^d$.
    \end{enumerate}  
\end{corollary}

\begin{proof}
    Without loss of generality, $x_0 = 0$. The blow-up sequence $u_{r} := r^{-1}u(rx)$ satisfies $u_r(0) = 0$ and is uniformly Lipschitz on any $B(0,R)$ with $R>0$ for all $r$ sufficiently small.  Thus $\{u_r\}_{r >0}$ is precompact with respect to local uniform convergence and weak $H^1_{\textup{\rm loc}}(\R^d)$.  Let $u^0\in {C^{0,1}}(\R^d)$ with $u^0(0) = 0$ be any subsequential limit. Note also $Q(rx)\to Q(x_0)$ uniformly locally as $Q\in C^{0,1}(U)$. Thus Lemma~\ref{l.directionalStable} allows us to conclude.
    
    The local Hausdorff distance convergence of the positivity sets in the downward minimizing case follows from strong non-degeneracy of downward minimizers, see \tref{downward-min-non-degen} and \cite{VelichkovBook}*{Lemma 4.4 and Chapter 6.2}.
    
\end{proof}

Next we discuss regular points in the general setting of inner variational solutions. We show, as Corollary of results of \cite{KriventsovWeiss}, that when the blow-up limit is a one-plane solution the positivity sets also converge in Hausdorff distance so that $\ep$-regularity results apply. 

\begin{lemma}\label{l.one-plane-limit}
    Let $u$ be an inner variational solution of \eqref{e.bernoulli-basic} in $U$ satisfying the weak non-degeneracy property: for any $x \in \partial \{u>0\} \cap U$ and any $0 < r < \frac{1}{2}d(x,\partial U)$
    \[\sup_{B_r(x)} u \geq cr.\]
    For any $x_0  \in \partial \{u>0\} \cap U$ consider the blow-up sequence for $0 < r < \frac{1}{2}d(x_0,\partial U)$
    \[u_{x_0,r}(x):= r^{-1}u(x_0+rx) \ \hbox{ and } \ \chi_{x_0,r} := \chi(x_0+rx).\]
    The sequence $u_{x_0,r}$ is pre-compact in the local uniform and $H^1_{\textup{\rm loc}}(\R^d)$ topologies, and $\chi_{x_0,r}$ are precompact in $L^1_{\rm loc}(\R^d)$. If $(u^0,\chi^0)$ is a limit along a subsequence $r_j \to 0$ then $(u^0,\chi^0)$ is an inner variational solution. If $u^0 = Q(x_0)(x_d)_+$ then $\chi^0 = {\bf 1}_{\{x_d >0\}}$ and $\overline{\{u_{x_0,r_j}>0\}}$ converges locally in Hausdorff distance to $\{x_d\geq 0\}$.
\end{lemma}
The weak non-degeneracy hypothesis holds in all dimensions for smallest supersolutions \tref{smallest-super-non-degen}, downward minimizers \tref{downward-min-non-degen}, and in $d=2$ for largest subsolutions \tref{largest-sub-non-degen}.  As per the discussion and example in Section~\ref{s.non-degen} the weak non-degeneracy by itself is not enough to guarantee convergence of the positivity sets.
\begin{proof}
The compactness follows from \lref{inner-compactness-stability}. The inner variation stationarity condition \eqref{e.stationary-inner-var} gives $\int \chi^0 \grad \cdot \xi = 0$ for all $\xi \in C^{0,1}_c(\{x_d<0\})$ which implies that either $\chi^0 = {\bf 1}_{\R^d}$ or $\chi^0 = {\bf 1}_{\{x_d>0\}}$. But $(Q(x_0)(x_d)_+, {\bf 1}_{\R^d})$ is not an inner variational solution (plug into \eqref{e.stationary-inner-var} and integrate by parts), so we must be in the latter case.

Call $\Omega_j:=\{u_{x_0,r_j}>0\}$. We just show Hausdorff convergence of $\overline{\Omega}_j$ to $\{x_d\geq 0\}$ in $B_1$. Let $\ep>0$ and call $K_\ep:= \{x_d \leq -\ep\} \cap \overline{B_1}$. We claim that $\partial \Omega_j \cap K_\ep = \emptyset$ for sufficiently large $j$.  Suppose otherwise, then there is a subsequence $j_k \to \infty$ and points $y_{j_k}\in \partial \Omega_{j_k} \cap K_\ep$. Then by the non-degeneracy hypothesis $\sup_{B_{\ep/2}(y_{j_k})} u_{x_0,r_{j_k}} \geq c \ep$ for all $k$. This contradicts that $u_{x_0,r_j} \to 0$ locally uniformly on $\{x_d \leq -\ep/2\}$.

Since $\partial \Omega_j \cap K_\ep = \emptyset$ for sufficiently large $j$ we conclude that $\Omega_j \cap K_\ep$ is both open and closed in $K_\ep$. Since $K_\ep$ is connected this means either $K_\ep \subset \Omega_j$ or $\Omega_j \cap K_\ep = \emptyset$. The first possibility cannot occur for infinitely many $j$ since that would imply ${\bf 1}_{K_\ep} \leq {\bf 1}_{\Omega_j} \leq \chi_{x_0,r_j}$ which contradicts that $\chi_{x_0,r_j}$ converges in $L^1_{\rm loc}$ to $\{x_d \geq 0\}$. So $\Omega_j \cap K_\ep = \emptyset$ for sufficiently large $j$. The other side of Hausdorff convergence is easy from uniform convergence of the functions so we conclude.

\end{proof}

Finally, we classify the one-sided minimizers formed by two-plane profiles.

\begin{proposition}\label{p.two-plane-minimizing-props}
    Let $\phi_\alpha(s) = \alpha |s|$ the one dimensional two-plane solution, let $q >0$ and $J_q$ be the Alt-Caffarelli energy with $Q(x) \equiv q$.  Then
    \begin{enumerate}[label = (\roman*)]
        \item $\phi_\alpha(x_d)$ is not an upward minimizer of $J_q$ in $\R^d$ for any $\alpha>0$.
        \item $\phi_\alpha(x_d)$ is a downward minimizer of $J_q$ in $\R^d$ for $\alpha\geq q$, but not downward minimizing for $0 < \alpha < q$.
    \end{enumerate}
\end{proposition}

We defer the proof to Appendix~\ref{s.proof-two-plane-minimizing-props}. This is a direct computation in $d=1$, in higher dimensions the idea is to do the one dimensional perturbation in the $e_d$ direction and cut-off smoothly at a large length scale in the tangential directions to make a compact perturbation.

\section{The parabolic Bernoulli problem}\label{s.parabolic_bernoulli}

In this section we set up the parabolic Bernoulli problem with Dirichlet boundary data generated by the strict subsolution (or supersolution) ${g}$, define the solution notions for the evolutionary problem, and state the main results on the existence and long-time behavior of these solutions.

We write $U_\infty := U \times (0,\infty)$ and recall that we denote the parabolic boundary by $\partial_P U_\infty.$
The \emph{parabolic Bernoulli problem} with time-independent Dirichlet data ${g}_+$ is as follows:
\begin{equation}\label{e.bernoulli-parabolic}
     \begin{cases}
         \partial_t \mathfrak u =\Delta \mathfrak u & \hbox{in } \{\mathfrak u>0\} \cap U_\infty, \\
         |\grad \mathfrak u| = Q(x) &\hbox{on } \partial \{\mathfrak u>0\} \cap U_\infty, \\
         \mathfrak u = {g}_+ &\hbox{on } \partial_PU_\infty.
     \end{cases}
\end{equation}
We will require both viscosity and inner variational solution notions for the interior PDE in \eref{bernoulli-parabolic}. 

\subsection{Parabolic viscosity solutions}\label{subsec:paraVisc}

First we recall viscosity solution notions. We actually use a slightly different notion of viscosity solutions than the one based on touching from \cite{Kim03}. Our notion is stronger, as viscosity solutions in our sense are also viscosity solutions in the sense of \cite{Kim03}, but our existence proof using the semilinear approximation (see \sref{semilinear-intro} below) produces such solutions. In essence, we follow the philosophy of viscosity solutions more directly and require that $u$ satisfies parabolic strict comparison with respect to smooth strict sub/supersolution barriers. This would follow from the touching definition if one could rule out outward jumps of the positive phase in order to produce a first crossing location. We also introduce a notion of relaxed subsolution, parallel to the elliptic case in Section~\ref{s.viscosity-solns-prelim}, which is new in the literature.

   \begin{definition}\label{d.prec-ordering}
        Let $E$ be closed in $\overline{U} \times [0,T]$ and $F$ be a subset of $\overline{U} \times [0,T]$. We write $\mathfrak u \prec_E \mathfrak v$ on $F$ if $\mathfrak u < \mathfrak v$ on $E \cap F$. We use the shorthand $\mathfrak u \prec \mathfrak v$ to mean $\mathfrak u \prec_{E} \mathfrak v$ with $E=\overline{\{\mathfrak u>0\}}$.
    \end{definition}

    To introduce our notion of viscosity solution for \eref{bernoulli-parabolic}, we first define, in parallel with \dref{strict-sub}, a notion of classical strict sub- and supersolution. 
    
    \begin{definition}\label{d.parabolic-strict-sub}
    Say that $\varphi \in C^{2,1}(\overline{V} \times [a,b])$ is a classical strict subsolution (resp. supersolution) of \eref{bernoulli-parabolic} in $\overline{V} \times [a,b]$ if there are constants $\frac{1}{2} \geq a_0, \delta_0 > 0$ such that
\begin{enumerate}[label = (\roman*)]
    \item $\partial_t \varphi - \Delta \varphi \leq -\delta_0$ (resp. $\partial_t \varphi - \Delta \varphi \geq \delta_0$) in $\{\varphi > -a_0\} \cap (\overline{V} \times [a,b])$, and
    \item $|\grad \varphi|^2 \geq (1+\delta_0) Q^2$ (resp. $|\grad \varphi|^2 \leq (1-\delta_0) Q^2$) in $\{|\varphi| \leq a_0\} \cap (\overline{V} \times [a,b])$.
\end{enumerate}
Note for this definition to work as expected, we actually consider $\varphi$ and the sets above to be defined in an open neighborhood of $\overline{V}\times [a,b]$.
\end{definition}

    \begin{definition}\label{d.parabolic-supersolution} \label{d.parabolic-subsolution}
    We say that $\mathfrak u \in C(U \times (0,T];{[0,\infty)})$ is a \emph{supersolution} (resp. \emph{subsolution}) of \eref{bernoulli-parabolic} if for any parabolic cylinder $V \times (a,b] \subset\subset U \times (0,T]$ and any $\varphi \in C^\infty(\overline{V} \times [a,b])$ a classical strict subsolution (resp. supersolution) of \eref{bernoulli-parabolic} in $\overline{V} \times [a,b]$ with $\varphi \prec \mathfrak u$ (resp. $\mathfrak u \prec  \varphi $) on $\partial_P(V \times (a,b])$, it follows that $\varphi \prec \mathfrak u$ (resp. $\mathfrak u \prec \varphi$) in $V \times (a,b]$. We say that $\mathfrak  u$ is a viscosity solution of \eref{bernoulli-parabolic} if it is both a subsolution and a supersolution.
\end{definition}

\begin{remark}[Solutions of \eqref{e.bernoulli-basic} are stationary solutions of \eqref{e.bernoulli-parabolic}] \label{r.statToPara}
    If $u$ is a viscosity subsolution (resp. supersolution) of \eqref{e.bernoulli-basic} then $u$ is a (stationary) viscosity subsolution (resp. supersolution) of \eqref{e.bernoulli-parabolic}. In brief, the argument is as follows. Suppose $V \subset \subset U$ and $\varphi$ is a smooth strict supersolution of \eqref{e.bernoulli-parabolic} with $u \prec \varphi$ on $\partial_P(V \times (a,b])$. If $u \prec \varphi$ fails to hold in $V \times(a,b]$ then, by continuity of $u$ and $\varphi$ on $\overline{\{u>0\}}$ there is a first crossing location $(x_0,t_0) \in \overline{\{u>0\}} \cap (V \times (a,b])$. Then $x \mapsto u(x)-\varphi(x,t_0)$ has a local max at $x_0$. Since $u$ is time-independent $\partial_t \varphi(x_0,t_0) \leq 0$ so $0<(\partial_t - \Delta)\varphi(x_0,t_0) \leq -\Delta \varphi(x_0,t_0)$. So we conclude, from the strict supersolution property of $\varphi$ that
    \[\Delta \varphi(x_0,t_0) < 0 \ \hbox{ and in case $\varphi(x_0,t_0) = 0$ then } \ |\grad \varphi(x_0,t_0)| < Q(x_0).\]
    This contradicts the subsolution property of $u$.  The supersolution case is similar.
\end{remark}

Finally, we introduce a convenient relaxation of the subsolution notion which is unconditionally stable under uniform limits; compare with \dref{relaxed-sub-stationary} in the stationary case.

\begin{definition}\label{d.parabolic-relaxed-subsolution}
    Consider $\mathfrak u \in C(U \times (0,T];{[0,\infty)})$ and $E$ closed in $\overline{U} \times [0,T]$ with $\{\mathfrak u>0\} \subset E$. We say that $(\mathfrak u,E)$ is a 
    \emph{relaxed subsolution} of \eref{bernoulli-parabolic} if for any parabolic cylinder $V \times (a,b] \subset\subset U \times (0,T]$ and any $\varphi \in C^\infty(\overline{V} \times [a,b])$ a classical strict supersolution of \eref{bernoulli-parabolic} in $\overline{V} \times [a,b]$ with $\mathfrak u \prec_E \varphi$ on $\partial_P(V \times (a,b])$, it follows that  $\mathfrak u \prec_E \varphi$ in $V \times (a,b]$. We say that $(\mathfrak u,E)$ is a \emph{relaxed viscosity solution} if $\mathfrak u$ is a (standard) supersolution and $(\mathfrak u,E)$ is a relaxed subsolution.
\end{definition}
As in the stationary case, viscosity subsolutions are relaxed subsolutions with $E = \overline{\{\mathfrak u>0\}}$ and it is easy to show that if $(\mathfrak u,E)$ is a relaxed subsolution and $E \subset \overline{\{\mathfrak u>0\}}$ then $E = \overline{\{\mathfrak u>0\}}$ and $\mathfrak u$ is a subsolution in the sense of Definition~\ref{d.parabolic-subsolution}. Note also that restrictions of relaxed subsolutions to a parabolic cylinder sub-domain are still relaxed subsolutions.

We write $E_t:= \{x: (x,t)\in E\}$ for the time slices of $E$. In the monotone increasing case, all relaxed subsolutions are in fact subsolutions, as established by the following proposition.

\begin{proposition} \label{p.monotone-nondegen-relaxed} %
        Suppose that $(\mathfrak u,E)$ is a relaxed subsolution of \eref{bernoulli-parabolic} in $U \times (0,T]$, as in Definition \ref{d.parabolic-relaxed-subsolution}, with $\mathfrak{u} \in C(\overline{U} \times [0,T])$ and $t \mapsto \mathfrak u(\cdot,t)$ is monotone increasing.
    If $E_0 \cap U \subset \overline{\{\mathfrak u_0>0\}}$ then $E \cap (U \times (0,T]) \subset \overline{\{\mathfrak u>0\}}$ and $\mathfrak u$ is a standard subsolution, as in Definition \ref{d.parabolic-subsolution}.
    \end{proposition}
    
    \begin{proof}
    \emph{Step 1: Non-degeneracy of solutions.} We first show that if $(\mathfrak u,E)$ is a relaxed subsolution in $U \times (0,T]$ with $\overline{B_1(0)} \subset E_0^c$ and $0 \in E_T$ then
    \begin{equation}\label{eqn:nondegenStep}
     \sup_{x\in \partial B_r(0)} \mathfrak u(x,T) \geq cr \ \hbox{ for all } \ 0 < r < 1
     \end{equation}
    and a universal constant $c>0.$
    In other words ${0} \in \overline{\{\mathfrak{u}(\cdot,T)>0\}}$. 
    It suffices to show the result for $r=1$ since $(\mathfrak u(rx,r^2t)/r,{\tilde E})$, where $\tilde E : = \{(x,t) : (rx,r^2t)\in E\}$, is a relaxed subsolution of the same PDE and the other hypotheses for \eqref{eqn:nondegenStep} are also invariant under this scaling (up to changing the final time).

For $\kappa>0$, let $h$ be the radially symmetric barrier
\[h(x) = \frac{\kappa}{4d}\left(1-\exp\left(4d\left(\frac{1}{4}-|x|^2\right)\right)\right),\]
which satisfies 
\begin{align*}
\Delta h < 0 & \quad \text{in } \overline{\{h > 0\}} \cap B_1 = B_1 \setminus B_{1/2},\\
|\nabla h|=\kappa & \quad \text{on } \partial\{h > 0\} \cap B_1 = \partial B_{1/2}.
\end{align*}
Taking $\kappa$ to be a sufficiently small universal constant, we can ensure $h$ is a smooth strict supersolution of \eref{bernoulli-basic} and hence also of \eqref{e.bernoulli-parabolic}.

    Let $\varphi = h_+$ and suppose, toward a contradiction, that $\mathfrak u< \varphi$ on $\partial B_1 \times \{T\}$. Then, by monotonicity, $\mathfrak u< \varphi$ on $\partial B_1 \times (0,T]$. Also, since $(\overline{B_1(0)} \times \{0\} )\cap E = \emptyset$, we conclude $\mathfrak u \prec_E \varphi$ on $\overline{B_1(0)} \times \{0\}$ (since the strict inequality condition is vacuous).  Thus $\mathfrak u \prec_E \varphi$ on $\partial_P(B_1(0) \times (0,T])$ and Definition~\ref{d.parabolic-relaxed-subsolution} gives $\mathfrak u \prec_E \varphi$ in $B_1(0) \times (0,T]$. However, since $\varphi\equiv 0$ on $B_{1/2}(0)$ and $\mathfrak u(\cdot, T)$ is non-negative (because $t\mapsto \mathfrak u(\cdot, t)$ is monotone increasing and $\mathfrak u_0 = g_+ \geq 0$), this means $\mathfrak u(\cdot, T) = \varphi$ on $B_{1/2}(0)$, contradicting the hypothesis $0 \in E_T$. We conclude that $\mathfrak u \geq \varphi$ somewhere on $\partial B_1 \times \{T\}$ and since $\varphi\vert_{\partial B_1} = c$ for some universal constant $c$, we have established \eqref{eqn:nondegenStep} with $r=1$.
    
    
    \emph{Step 2: Application of non-degeneracy.} Now consider $(\mathfrak u,E)$ as in the proposition statement. Let $x_0 \in E_{t_0} \cap U$ with $t_0 \in (0,T]$. If $x_0 \in E_0$ then we use the initial data and monotonicity hypotheses to conclude $x_0\in E_0 \cap U \subset \overline{\{\mathfrak u_0>0\}} \subset \overline{\{\mathfrak u(\cdot,t_0)>0\}}$. So assume $x_0 \not\in E_0$. Since $E_0$ is closed and so $U \setminus E_0$ is open, the (rescaled) hypotheses of \eqref{eqn:nondegenStep} hold in a sufficiently small ball centered at $x_0$ with time interval $(0,t_0]$, and we conclude that $x_0 \in \overline{\{\mathfrak u(\cdot,t_0)>0\}}$. Altogether we find $E_{t_0} \cap U \subset \overline{\{\mathfrak u(\cdot,t_0)>0\}}$, giving the desired subset relation.
    \end{proof}

\subsection{Parabolic comparison principles}\label{s.parabolic-comparison}

The identification of the long-time limits of \eref{bernoulli-parabolic} with the Perron solutions rests on a \emph{strict} comparison principle for \eref{bernoulli-parabolic}.

\begin{theorem}[Strict comparison]\label{t.comparison}
    If $U$ is a bounded domain and $\mathfrak u$ and $\mathfrak v$ in $C(\overline{U} \times [0,T])$ are respectively a viscosity subsolution and a viscosity supersolution of \eqref{e.bernoulli-parabolic} in $U \times (0,T]$ with $\mathfrak u \prec \mathfrak v$ on a neighborhood of the parabolic boundary $\partial_P(U \times (0,T])$, then $\mathfrak u \prec \mathfrak v$ on $\overline{U} \times [0,T]$.
\end{theorem}

This result follows almost immediately from the arguments of Kim \cite{Kim03}, though the statement is not exactly the same as \cite{Kim03}*{Theorem 2.2}, which is set on the whole space without a spatial boundary. The strict ordering in a neighborhood of the parabolic boundary is inherited by the sup/inf-convolved regularizations on a slightly smaller parabolic domain and then \cite{Kim03}*{Theorem 2.2} shows that an interior crossing results in a contradiction.

\subsection{Parabolic inner variational solutions}\label{s.parabolic-inner} Next, we introduce the notion of parabolic inner variational solutions.
\begin{definition}\label{d.parabolic-inner}
A pair $(\mathfrak u,\chi)$ is an \emph{inner variational solution} of \eref{bernoulli-parabolic} in $U\times (0,T)$, including $T =+\infty$, if the following hold:
\begin{enumerate}[label = (\roman*)]
	\item $\mathfrak u \in C(U\times (0,T))$ is non-negative, $\grad \mathfrak u \in L^\infty_{\rm loc}(U\times (0,T);\R^d)$, and $\partial_t \mathfrak u \in L^2(U\times (0,T))$.
	\item $\chi : U\times (0,T) \to \{0,1\}$ is Borel measurable.
	\item ${\bf 1}_{\{\mathfrak u > 0\}} \leq \chi$ $\mathcal{L}^{d+1}$-a.e.\ in $U\times (0,T)$.
	\item $\partial_t \mathfrak{u} = \Delta \mathfrak{u}$ in $\{\mathfrak{u}>0\} \cap (U \times (0,T))$ in the strong sense. 
	\item For every $\xi \in C^{1}_c(U\times (0,T); \R^d)$,
	\begin{equation}\label{e.parabolic-inner-var}
		\int_{U\times (0,T)} \Big[(|\nabla \mathfrak u|^2 + Q^2\chi)  \grad \cdot \xi - 2 \grad \mathfrak  u \cdot D\xi \grad \mathfrak u + \grad (Q^2) \cdot \xi  \chi - 2 (\xi \cdot \grad \mathfrak u)  \partial_t \mathfrak u\Big]   dx\,  dt = 0.
	\end{equation}
\end{enumerate}
\end{definition}

It is elementary to check that solutions of the stationary problem \eqref{e.bernoulli-basic} in the inner variational sense (see \dref{stationary-inner}) are time independent inner variational solutions of the parabolic problem \eqref{e.bernoulli-parabolic}.

We further remark that our construction of inner variational solutions to the parabolic problem \eqref{e.bernoulli-parabolic} will yield local Lipschitz estimates for $u$ and local perimeter bounds for $\chi.$ It is plausible, however, that a perimeter bound can be proved directly along the lines of \cite{KriventsovWeiss}*{Lemma 3.3}.

\subsection{The semilinear approximation}\label{s.semilinear-intro}

It is well known that the parabolic Bernoulli problem \eref{bernoulli-parabolic} arises as the singular limit of a semilinear reaction-diffusion equation (see \cites{BerestyckiCaffarelliNirenberg, CaffarelliVazquez1995,Weiss03}). We use these semilinear approximations, as done for example in \cite{Weiss03}, to show existence of solutions of the sharp interface problem \eref{bernoulli-parabolic} with desirable properties, such as monotonicity. Precisely, we consider the semilinear problem
\begin{equation}\label{e.bernoulli-parabolic-semilinear}
\begin{cases}
         \partial_t\mathfrak{u}_\ep =\Delta \mathfrak{u}_\ep -  Q(x)^2 \frac{1}{\ep}\beta \left(\frac{\mathfrak{u}_\ep}{\ep}\right) & \hbox{ in } U_\infty, \\
         \mathfrak{u}_\ep = g_\eps &\hbox{ on } \partial_P U_\infty.
\end{cases}
\end{equation}
Here $g_{\ep}$ is an appropriate regularization of the data $g_+$ in \eref{bernoulli-parabolic}, see below for more details. The reaction term $\beta \in  C^\infty(\R)$ satisfies
\begin{equation}\label{e.beta-conditions}
\begin{minipage}{\linewidth} 
\begin{itemize}
\item $\operatorname{supp} \beta = [0,1]$ and $\beta > 0$ on $(0,1)$, 
\item $\beta$ is nondecreasing on $[0,1/2]$ and nonincreasing on $[1/2,1]$,
\item $\int_{\R}\beta(s)   ds = \tfrac12$.
\end{itemize}
\end{minipage}
\end{equation}
We denote 
\[
\beta_\eps(z) := \frac{1}{\ep}\beta\left(\frac{z}{\ep}\right) \quad\hbox{and}\quad \mathcal{B}_\ep(z) := \int_0^z \beta_\eps(s)   ds .
\]
Note that $0 \leq \mathcal{B}_\ep \leq \frac12$, $\mathcal{B}_\ep(z) = \frac12$ for $z \geq \ep$, and $\mathcal{B}_\ep(z) = 0$ for $z \leq 0$. We define
\begin{equation}\label{eqn:chiEps}
 \chi_\ep(x,t) := 2\mathcal{B}_\ep(\mathfrak{u}_\ep(x,t)),
 \end{equation}
and sometimes use the shorthand $\chi^0_{\ep}$  to denote $\chi_\ep(\cdot,0)$. Note that $\chi_\ep \in [0,1]$ plays the role of ${\bf 1}_{\{\mathfrak u>0\}}$ in the context of the semilinear problem, and $\grad \chi_\ep = 2\beta_\eps(\mathfrak{u}_\ep)\grad \mathfrak{u}_\ep$. This motivates the definition of the Bernoulli free energy of a pair $(\mathfrak u,\chi)$ in $U$: 
\[J(u,\chi;U) := \int_U |\grad u|^2 + Q(x)^2 \chi(x) \ dx.\]

\subsubsection{Recovery sequences for sub and supersolutions}\label{s.recovery-sequences} 
Let us explain how we regularize $g_+$ to create the data $g_\ep$ for the parabolic problem. The goal is to carry the strict sub or supersolution property of $g$ for \eqref{e.bernoulli-parabolic} to a strict sub or supersolution property of $g_\ep$ for \eqref{e.bernoulli-parabolic-semilinear}. We will do this by composing $g$ with certain one-dimensional sub and supersolution profiles. This initial / boundary data construction is useful not only for obtaining well-prepared solutions of \eqref{e.bernoulli-parabolic-semilinear} but also for the perturbed test function method for proving the viscosity solution property in the limit $\ep \to 0$.  The proofs of all the results stated in this section are included in Appendix~\ref{s.well-prepared-data-proofs}.

First we explain the subsolution one-dimensional profile construction; a version of this was sketched in \cite{Petrosyan} and \cite{Kim03}. 
\begin{lemma}\label{l.profile-sub}
   For each $\theta \in (0,1)$ there is $\Phi = \Phi_{1,\theta} \in C^\infty(\R)$ with the following properties.
\begin{enumerate}[label = (\roman*)]
    \item\label{part.profile-sub-a} $\Phi'' = \theta \beta(\Phi)$ on $\R$.
    \item\label{part.profile-sub-b} $\sqrt{1-\theta} \leq \Phi' \leq 1$ on $\R$, and $\Phi(s) = s$ for $s \geq 1$.
    \item\label{part.profile-sub-c} There is $s_0 \in [-c_\theta, 0]$, with $c_\theta := \frac{1-\sqrt{1-\theta}}{\sqrt{1-\theta}}$, such that $\Phi(s_0) = 0$ and $\Phi > 0$ exactly on $(s_0,\infty)$. Moreover, $s \leq \Phi(s) \leq 1$ for all $s \leq 1$, and $\sup_{\R} |(\Phi)_+ - s_+| \leq 1$.
\end{enumerate} 
\end{lemma}

Next we construct the supersolution profiles. Similar constructions were carried out in \cites{CaffarelliVazquez1995, Petrosyan, Kim03}.

\begin{lemma}\label{l.profile-super}
    For each $\hat\theta \in (1,\infty)$ there are a constant level $\kappa = \kappa_{\hat\theta} \in (0,1)$ and constants $C_{\hat\theta} \geq 1 \geq c_{\hat\theta} > 0$, all depending only on $\hat\theta$ and $\beta$, and $\Psi = \Psi_{1,\hat\theta} \in C^\infty(\R)$ with the following properties.
\begin{enumerate}[label = (\roman*)]
    \item\label{part.profile-super-a} $0 \leq \Psi'' \leq \hat\theta  \beta(\Psi)$ on $\R$.
    \item\label{part.profile-super-b} $0 < \Psi' \leq 1$ on $\R$, and $\Psi(s) = s$ for $s \geq 1$.
    \item\label{part.profile-super-c} $\max(s,\kappa) < \Psi(s) < 1$ for all $s < 1$, and $\Psi(s) \downarrow \kappa$ as $s \to -\infty$. In particular $\Psi > \kappa > 0$ on $\R$ and $\sup_{\R}|\Psi - s_+| \leq 1$.
    \item\label{part.profile-super-e} $\max\{\Psi(s) - \kappa ,\Psi'(s) , \Psi''(s)\} \leq C_{\hat\theta}  e^{c_{\hat\theta}  s}$ for all $s \leq 0$.
    \item\label{part.profile-super-f} $\hat\theta \mapsto \kappa_{\hat\theta}$ is monotone increasing on $(1,\infty)$, with $\kappa_{\hat\theta}  \searrow 0$ as $\hat\theta \searrow 1$, and $\kappa_{\hat\theta} \nearrow 1$ as $\hat\theta \nearrow +\infty$.
\end{enumerate}
\end{lemma}
For $\ep > 0$ define the rescaled profiles
\begin{equation}\label{e.profile-ep-def}
\Phi_{\ep,\theta}(s) := \ep  \Phi_{1,\theta}(s/\ep), \  \Psi_{\ep,\hat\theta}(s) := \ep  \Psi_{1,\hat\theta}(s/\ep).
\end{equation}
 Then $\Phi_{\ep,\theta}$ and $\Psi_{\ep,\hat\theta}$ belong to $C^\infty(\R)$ and satisfy
 \[\Phi_{\ep,\theta}'' = \theta \beta_\eps(\Phi_{\ep,\theta}) \quad \hbox{and} \quad 0 \leq \Psi_{\ep,\hat\theta}'' \leq \hat\theta  \beta_\eps(\Psi_{\ep,\hat\theta}) \  \hbox{on } \R,\]
 together with the rescaled versions of the remaining properties of \lref{profile-sub} and \lref{profile-super}.

\begin{lemma}\label{l.well-prepared}
Assume that $g$ is a smooth strict subsolution (resp. smooth strict supersolution) of \eqref{e.bernoulli-basic} in the sense of \dref{strict-sub} with parameters $\delta_0,a_0>0$. For any $\frac{1}{1+\delta_0} \leq \theta < 1$ (resp. $1<\hat \theta \leq\frac{1}{1-\delta_0}$), there exists $\ep_0 > 0$ depending on $(a_0,Q_{\min},Q_{\max},\beta,\|{g}\|_{C^2(\overline U)})$ and on $\theta$ (resp. $\hat\theta$) such that for every $\ep \in (0,\ep_0]$ the function
\[g_\ep := \big(\Phi_{\ep,\theta}({g})\big)_+ \  \Big(\hbox{resp. } g_\ep := \Psi_{\ep,\hat\theta}({g})\Big)\]
satisfies the following properties.
\begin{enumerate}[label = (\roman*)]
    \item\label{part.gep-1} $g_\ep \in C^{0,1}(\overline{U})$ with $\|\grad g_\ep\|_{L^\infty(U)} \leq \|\grad {g}\|_{L^\infty(U)}$, and $g_\ep = {g}$ on $\{{g} \geq \ep\}$.
    \item\label{part.gep-2} ${g}_+ \leq g_\ep \leq {g}_+ + \ep$ on $\overline{U}$.
    \item\label{part.gep-3} $g_\ep$ is a viscosity subsolution (resp. classical supersolution) of the stationary equation
    \[\Delta v = Q(x)^2 \beta_\eps(v) \  \hbox{in } U.\]
    \item\label{part.gep-4} With $\chi^0_\eps := 2\mathcal{B}_\ep(g_\ep)$,
    \[J(g_\ep, \chi^0_\eps; U) \leq \int_U |\grad {g}|^2   dx + Q_{\max}^2  |U| =: E_0.\]
    \item\label{part.gep-5}  In the subsolution case, $\{g_\ep > 0\} = \{{g} > \ep s_0\}$ with $\ep s_0 \in [-c_\theta \ep,  0]$ the zero of $\Phi_{\ep,\theta}$. In particular
    \[\{{g} > 0\} \subset \{g_\ep > 0\} \subset \{{g} > -c_\theta  \ep\}.\]
\end{enumerate}
\end{lemma}

\subsubsection{Well-prepared solutions of the semilinear problem}\label{s.well-prepared-semilinear}

We can now state the main existence and regularity results for \eref{bernoulli-parabolic-semilinear} that are sufficient for our purposes. There are several parts, all of which follow by adapting known techniques established, for example, in \cites{CaffarelliVazquez1995,Petrosyan}. As none of the precise results we need are covered exactly by the results in the existing literature, we provide details of the proof in Appendix~\ref{s.semilinear}.

\begin{proposition}\label{p.semilinear-wellposed}
Let ${g}$ be a smooth strict subsolution (resp.\ supersolution) of \eref{bernoulli-basic} in the sense of \dref{strict-sub}. For all $\ep>0$ sufficiently small, there is a solution $\mathfrak{u}_\ep \in C(\overline U \times [0,\infty)) \cap C^{2,1}(U_\infty)$ of \eref{bernoulli-parabolic-semilinear} with well-prepared initial/boundary data $g_\eps \in C^{0,1}(\overline U)$ as in \lref{well-prepared} that satisfies the following properties.
\begin{enumerate}[label = (\roman*)]
\item\label{part.monotone} (Monotonicity) $t \mapsto \mathfrak{u}_\ep(x,t)$ is monotone increasing (resp.\ decreasing) for every $x \in \overline U$, and
  \begin{equation}\label{e.Linfty} 0 \leq \mathfrak{u}_\ep \leq \|{g}\|_{L^\infty(U)} + \ep \  \hbox{in } \overline U \times [0,\infty). \end{equation}
\item\label{part.H1-continuity} (Energy dissipation) 
For every $T > 0$,
  \begin{equation}\label{e.dissipation-eps}
  \tfrac12 J(\mathfrak{u}_\ep(T),\chi_\ep(T);U) + \int_0^T\!\!\int_U (\partial_t \mathfrak{u}_\ep)^2 dx \, dt = \tfrac12 J(g_\ep,\chi^0_\eps;U) \leq \tfrac12 E_0.
  \end{equation}
\item\label{part.lipschitz} (Interior Lipschitz bound) For every parabolic cylinder $\mathcal{Q}_{2r}(x,t) \subset U_\infty$ with $r \leq 1$,
  \begin{equation}\label{e.localEst}
  \|\grad \mathfrak{u}_\ep\|_{L^\infty(\mathcal{Q}_r(x,t))} \leq C\left(r^{-1}\|\mathfrak{u}_\ep\|_{L^\infty(\mathcal{Q}_{2r}(x,t))} + 1\right).
  \end{equation}
 Moreover $\sup_{t>0}\|\grad \mathfrak{u}_\ep(\cdot,t)\|_{L^\infty(V)} \leq C_V$ for every $V \subset\subset U$.
\item\label{part.attainment} (Continuity up to the parabolic boundary) There is a modulus of continuity $\varpi$, depending only on $||g||_{L^{\infty}}$, $||\nabla g||_{L^{\infty}}$ and universal parameters but independent of $\ep$, such that
  \begin{equation}\label{e.attainment}
  |\mathfrak{u}_\ep(x,t) - \mathfrak{u}_\ep(y,s)| \leq \varpi\big(|x-y| + |t-s|\big) \ \hbox{for all } (x,t),(y,s) \in \overline U \times [0,\infty).
  \end{equation}
  In particular, $\mathfrak{u}_\eps$ attains the  data $g_\eps$ in a continuous manner on $\partial_P U_{\infty}$.
\item\label{part.trace} (Positivity set does not jump outward at $t=0$) In the subsolution case, for any $K \subset\subset \overline{U} \cap \{g<0\}$ there exists $\tau>0$ such that for every $\kappa \in (0,1]$ and all $\ep>0$ sufficiently small, depending on $\kappa$, $K \subset \{\mathfrak{u}_\ep(\cdot,t) < \kappa\ep\}$ for all $0 \leq t \leq \tau$.

\end{enumerate}
\end{proposition}

\subsection{Existence and limit theorems for the parabolic Bernoulli problem}\label{s.limit-theorems}

In the next two sections, we will establish the following results about  \eref{bernoulli-parabolic}. 

\begin{theorem}\label{t.parabolic-existence}

Suppose $g$ is a smooth strict subsolution (resp. supersolution) of \eqref{e.bernoulli-basic} in $\overline{U}$. Then there exists a pair $(\mathfrak{u},\chi)$ such that 
\begin{enumerate}[label = (\roman*)]
\item\label{part.continuity} $\mathfrak{u} \in C(\overline{U}\times[0,\infty))$ with $\mathfrak{u} = {g}_+$ on $\partial_P U_\infty$, $0 \leq \mathfrak{u} \leq \max_{\overline{U}} g_+$, and modulus of continuity estimate
\[|\mathfrak{u}(x,t) - \mathfrak{u}(y,s)| \leq \varpi\big(|x-y| + |t-s|\big) \ \hbox{for all } (x,t),(y,s) \in\overline U \times [0,\infty),\]
where $\varpi$ is the modulus from \eref{attainment}.
\item $\mathfrak{u}$ is monotone increasing (resp.  decreasing) in time. 
\item $(\mathfrak{u},\chi)$ is an inner variational solution of \eref{bernoulli-parabolic} in $U_\infty$.
\item $\mathfrak{u}$ is a viscosity solution (resp. relaxed viscosity solution) of \eref{bernoulli-parabolic}  in $U_\infty$. 
\item %
The dissipation inequality
\begin{equation}\label{eqn:diss2}
\frac{1}{2}J(\mathfrak{u}(T),\chi(T);U) + \int_0^T\int_U (\partial_t \mathfrak{u})^2  \, dx  \, dt  \leq \frac{1}{2}E_0
\end{equation}
holds for a.e. $T \in (0,\infty)$;
the interior Lipschitz bound 
\begin{equation}\label{eqn:LipEst}
\|\grad \mathfrak{u}\|_{L^\infty(\mathcal{Q}_r(x,t))} \leq C\big(r^{-1}\|\mathfrak{u}\|_{L^\infty(\mathcal{Q}_{2r}(x,t))} + 1\big)
\end{equation} holds for $\mathcal{Q}_{2r}(x,t) \subset U_\infty$ and $r \leq 1$ with a universal constant $C>0$;
the weighted perimeter bound 
\begin{equation}\label{eqn:chiEst2}
\int_{U_\infty}\eta |\grad\chi| \, dx \, dt \leq C_\eta \|\eta\|_{L^2((0,\infty);H^1(U))}
\end{equation} holds for any $\eta \in C^1_c(U_\infty;[0,\infty))$, with $C_\eta$ depending on $\eta$ only through $d_\eta = \dist({\rm spt} \, \eta, \partial_P U_\infty)$ and $T_\eta = \inf \{|b-a| : {\rm spt} \, \eta \subset U\times(a,b)\}$, and on $g$ only through $\max_{\overline{U}} |g|$ and $\|\nabla g\|_{L^2(U)}$.
\end{enumerate}

\end{theorem}

The function $\mathfrak{u}$ is obtained as a locally uniform subsequential limit of the semilinear approximations $\mathfrak{u}_\ep$ constructed in \pref{semilinear-wellposed}. The inner variational property is proved in Section \ref{s.inner-limits} (\pref{eps-inner-limit}). The attainment of the boundary data and the viscosity properties are proved in Section \ref{s.viscosity-limits}, where the proof of \tref{parabolic-existence} is assembled.


We also show that given \emph{any} inner variational solution of \eref{bernoulli-parabolic} the long time limit $T\to \infty$ converges to a stationary inner variational solution, as long as we have the additional regularity estimates \eqref{eqn:diss2}--\eqref{eqn:chiEst2} (which always hold for the semilinear limit).

\begin{theorem}[long-time limit, variational form]\label{t.longtime-inner}
    Let $(\mathfrak{u},\chi)$ be a bounded inner variational solution of \eref{bernoulli-parabolic} in $U_\infty$ satisfying \eqref{eqn:diss2}--\eqref{eqn:chiEst2}. Then there exist a sequence $t_i \to \infty$, a function $u_\infty\in H^1_{\rm loc}(U)$, and a Borel function $\chi_\infty : U \to \{0,1\}$ such that $\mathfrak{u}(\cdot,t_i) \to u_\infty$ locally uniformly in $U$ and in $H^1_{\rm loc}(U)$, $\chi(\cdot,t_i) \to \chi_\infty$ in $L^1_{\rm loc}(U)$, and $(u_\infty, \chi_\infty)$ is an inner variational solution of \eqref{e.bernoulli-basic} in $U$. Moreover, there is a universal constant $C>0$ such that
    \begin{equation}\label{eqn:inftyLip}
        \|\grad u_\infty\|_{L^\infty(B_r(x))} \leq C\big(r^{-1}\|u_\infty\|_{L^\infty(B_{2r}(x))} + 1\big)
    \end{equation}
    for any $B_{2r}(x) \subset U$ with $r\leq 1$. Finally, for any $V\subset\subset U$, there is a constant $C_V = C_V(E_0, \|\mathfrak{u}\|_{L^\infty(U_\infty)})>0$ such that
    \begin{equation}\label{eqn:inftyPerim}
        \int_{V} |\grad\chi_\infty| dx \leq C_V.
    \end{equation}
    
\end{theorem}

\begin{remark}\label{r.longtime-inner-indicator}
The conclusion of Theorem \ref{t.longtime-inner} can be strengthened using \cite{KriventsovWeiss}. Indeed, \cite{KriventsovWeiss}*{Theorem 1.2 (i)} implies either $u_\infty \equiv 0$ or $\chi_\infty = {\bf 1}_{\{u_\infty>0\}}$ a.e.; in either case $(u_\infty, {\bf 1}_{\{u_\infty>0\}})$ is an inner variational solution of \eqref{e.bernoulli-basic} in $U$. We will not need to use this stronger conclusion in what follows.
\end{remark}

For viscosity solutions we prove the following limit theorem in Section \ref{s.viscosity-limits}. This is essentially a consequence of \pref{supersolution-limit}, \cref{monotone-long-time}, and \lref{subsolution-inner-limit}.
\begin{theorem}[long-time limit, viscosity form]\label{t.longtime-viscosity} Let $\mathfrak u$ be bounded and monotone in time in $U_{\infty}$.
    \begin{enumerate}[label=(\roman*)]
        \item\label{part.longtime-viscosity-1} Suppose $\mathfrak u$ is a viscosity solution of \eqref{e.bernoulli-parabolic} in $U_{\infty}$ and $\mathfrak u(\cdot,t) \nearrow u_\infty$ locally uniformly in $U$ as $t \to \infty$. Then $u_\infty$ is a viscosity solution of \eqref{e.bernoulli-basic} in $U$.
        \item\label{part.longtime-viscosity-2} Suppose that $(\mathfrak{u},{\chi})$ is an inner variational solution and viscosity supersolution of \eqref{e.bernoulli-parabolic} which satisfies \eqref{eqn:diss2}--\eqref{eqn:chiEst2}. Then there is a sequence $t_i\to \infty$ such that $\mathfrak{u}(\cdot,t_i) \to u_\infty$ locally uniformly, ${\chi}(\cdot, t_i) \to \chi_\infty$ in $L^1_{\textup{\rm loc}}$, and $u_\infty$ is a viscosity solution of \eqref{e.bernoulli-basic}.
    \end{enumerate}
\end{theorem}

\section{Limits in the inner variational sense}\label{s.inner-limits}
In this section, we prove that the proposed double limit---with $\eps\to 0$ and then $t\to \infty$---leads to a variational solution of the Bernoulli problem. We encapsulate the first limit in the following proposition, which provides the variational half of \tref{parabolic-existence}. We will assume henceforth that the reaction term $\beta$ satisfies the properties \eref{beta-conditions}.

\begin{proposition}[$\ep \to 0$ limit, variational form]\label{p.eps-inner-limit}
Suppose $g_\eps\in H^1(U;[0,\infty))$ satisfy $\sup_\eps \|g_\eps\|_{L^\infty(U)} <\infty$ and $J(g_\eps,\chi^0_\eps;U) \leq E_0$. Let $\mathfrak{u}_\ep$ be  solutions of \eqref{e.bernoulli-parabolic-semilinear} with $\chi_\ep$ as in \eqref{eqn:chiEps}. There are a subsequence $\ep_j \downarrow 0$ and a pair $(\mathfrak{u},\chi)$ such that:
\begin{enumerate}[label = (\roman*)]
\item\label{part.eps-convergence} (Compactness) $\mathfrak{u}_{\ep_j} \to \mathfrak{u}$ locally uniformly in $U_\infty$, $\grad \mathfrak{u}_{\ep_j} \to \grad \mathfrak{u}$ strongly in $L^2_{\rm loc}(U_\infty)$ and weakly in $L^2(U\times(0,T))$ for every $T>0$, and $\chi_{\ep_j} \to \chi$ in $L^1_{\rm loc}(U_\infty)$ and pointwise a.e., with $\chi \in \{0,1\}$ a.e.;
\item\label{part.eps-inner} (Existence) $(\mathfrak{u},\chi)$ is an inner variational solution of \eref{bernoulli-parabolic} in $U_\infty$ in the sense of \dref{parabolic-inner};
\item\label{part.eps-estimates} (Estimates) the dissipation inequality \eqref{eqn:diss2}, the interior Lipschitz bound \eqref{eqn:LipEst}, and the weighted perimeter bound \eqref{eqn:chiEst2} hold.
\end{enumerate}
\end{proposition}

For the proof of \pref{eps-inner-limit}, we study the semilinear problem in its (inner) variational form and prove several compactness results. 
We point out that our arguments rely on Proposition \ref{p.semilinear-wellposed}, which is only stated for well-prepared data $g_\eps$; however, parts \ref{part.H1-continuity} and \ref{part.lipschitz} of Proposition \ref{p.semilinear-wellposed} and the bound
\begin{equation}\label{eqn:inftyBound}
0\leq \mathfrak{u}_\eps\leq \|g_\eps\|_{L^\infty(U)}
\end{equation} 
hold for any solution of the semilinear problem, as can be seen from the proof. Furthermore, for any $\mathcal{U}\subset\subset U\times (0,\infty)$ there is a modulus of continuity $\omega_{\mathcal{U}}$ such that the estimate \eref{attainment} holds on $\mathcal{U}$. For the reader who wishes to avoid verifying these more general facts, we note that it suffices for the purposes of this paper to consider the functions $g_\eps$ constructed in \lref{well-prepared} and solutions of the semilinear equation \eqref{e.bernoulli-parabolic-semilinear} as provided by Proposition~\ref{p.semilinear-wellposed}.

Once we complete the proof of \pref{eps-inner-limit}, we will turn to the proof of \tref{longtime-inner} in Subsection \ref{subsec:varTtoInfty}, which identifies the limit of $\mathfrak{u}$ as $T\to \infty$.

\subsection{Inner variational formula for the semilinear problem}\label{s.semilinear-inner-formula}
Take any smooth vector field $\xi \in C^\infty_c(U \times (0,\infty);\R^d)$. Multiplying the equation \eqref{e.bernoulli-parabolic-semilinear} on both sides by $2\xi \cdot \grad \mathfrak{u}_\ep$, integrating on $U_\infty : = U \times (0,\infty)$, and carrying out appropriate integration by parts yields the following inner variation formula for the semilinear problem
\begin{equation}\label{eqn:innerVarEps}
		\int_{U_\infty} (|\nabla \mathfrak{u}_\ep|^2 + Q^2\chi_\ep) \grad \cdot \xi - 2 \nabla \mathfrak{u}_\ep \cdot D\xi \nabla \mathfrak{u}_\ep +\grad (Q^2) \cdot \xi \chi_\ep - 2\xi \cdot \grad \mathfrak{u}_\ep \partial_t\mathfrak{u}_\ep \,  dx \, dt = 0.
	\end{equation}
	Recall that $\chi_\ep$ is defined in \eqref{eqn:chiEps}. For details see \cite{Weiss03}*{Definition 6.1}. To pass to the limit in \eqref{eqn:innerVarEps}, we need to show that $\nabla \mathfrak{u}_\eps$ converges strongly in $L^2$, $\chi_\eps$ converges strongly in $L^1$, and $\partial_t\mathfrak{u}_\ep$ converges weakly in $L^2.$ Towards this end, we derive some energetic estimates.

\subsection{Compactness of semilinear approximation}\label{subsec:strongGrad}
    
Since $J(g_\eps,\chi^0_\eps;U) \leq E_0$, \eqref{e.dissipation-eps} implies $\mathfrak{u}_\eps$ is uniformly bounded in $H^1(U\times (0,T))$, so that along a subsequence $\ep \to 0$ (not relabeled) $\mathfrak{u}_\eps$ converges to $\mathfrak{u}$ weakly in $H^1_{\textup{\rm loc}}(U_\infty)$. Adapting the argument of \cite{CaffarelliVazquez1995}*{Lemma 7.2}, it is possible to upgrade from weak convergence to strong convergence.

     
     \begin{lemma}\label{lem:strongGrad}
     Suppose $\mathfrak{u}_{\ep}$ converges to $\mathfrak{u}$ weakly in $H^1_{\textup{\rm loc}}(U_\infty)$. Then, along a further subsequence (not relabeled), $\nabla \mathfrak{u}_\eps$ converges to $\nabla \mathfrak{u}$ strongly in $L^2_{\rm loc}(U_\infty).$
     \end{lemma}
    \begin{proof}
    It suffices to show convergence of the norms 
    \[\|\eta^{1/2}\grad \mathfrak{u}_\ep\|_{L^2(U_\infty)} \to \|\eta^{1/2}\grad \mathfrak{u}\|_{L^2(U_\infty)}\] for arbitrary non-negative $\eta \in C^1_c(U_\infty)$. We multiply \eqref{e.bernoulli-parabolic-semilinear} by $\eta \mathfrak{u}_\ep$ where $\eta\in C_c^1(U_\infty)$ is a smooth, non-negative test function and use $Q^2\beta_\eps(\mathfrak{u}_\eps) \mathfrak{u}_\ep\eta \geq 0$ to find
    $$\int_{U_\infty} \partial_t \mathfrak{u}_\eps    \mathfrak{u}_\eps   \eta  \, dx \, dt \leq \int_{U_\infty} \Delta \mathfrak{u}_\eps   \mathfrak{u}_\eps   \eta  \, dx \, dt. $$
    Integrating both sides by parts (left-hand side in time, right-hand side in space), we have
    $$-\int_{U_\infty} \frac{1}{2}(\mathfrak{u}_\eps)^2 \partial_t \eta  \, dx \, dt \leq -\int_{U_\infty} |\nabla \mathfrak{u}_\eps|^2 \eta   \, dx \, dt  -\int_{U_\infty} \mathfrak{u}_\eps \nabla \mathfrak{u}_\eps\cdot \nabla \eta \,  dx \, dt,$$
    leading to 
	\begin{equation}\label{eqn:limsupForGrad}    
    \limsup_{\eps\to 0}\int_{U_\infty} |\nabla \mathfrak{u}_\eps|^2 \eta  \, dx \, dt \leq \int_{U_\infty} \frac{1}{2}\mathfrak{u}^2 \partial_t \eta  \, dx \, dt -  \int_{U_\infty} \mathfrak{u} \nabla \mathfrak{u}\cdot \nabla \eta   \, dx \, dt
    \end{equation}
    where we have used the strong convergence of $\mathfrak{u}_\eps$ in $L^2.$
    
    We claim that 
    \begin{equation}\label{eqn:heatEqnRelation}
    \int_{U_\infty} \frac{1}{2} \mathfrak{u}^2 \partial_t \eta \,  dx \, dt -  \int_{U_\infty} \mathfrak{u} \nabla \mathfrak{u}\cdot \nabla \eta  \,  dx \, dt = \int_{U_\infty} |\nabla \mathfrak{u}|^2 \eta   \, dx \, dt.
    \end{equation}
Assuming this identity holds, we insert \eqref{eqn:heatEqnRelation} into \eqref{eqn:limsupForGrad} and find
    \begin{equation}\label{eqn:limsupForGradFinal}    
    \limsup_{\eps\to 0}\int_{U_\infty} |\nabla \mathfrak{u}_\eps|^2 \eta \,  dx \, dt \leq \int_{U_\infty} |\nabla \mathfrak{u}|^2 \eta \,  dx \, dt.
    \end{equation}
    Since $\eta \geq 0$, weak lower semicontinuity gives $\int \eta |\grad \mathfrak{u}|^2 \leq \liminf_{\ep \to 0}\int\eta|\grad \mathfrak{u}_\ep|^2$, so in fact $\int \eta|\grad \mathfrak{u}_\ep|^2 \to \int\eta|\grad \mathfrak{u}|^2$. This finishes the proof of the lemma.

To see why \eqref{eqn:heatEqnRelation} holds, we note that $\partial_t \mathfrak{u} = \Delta \mathfrak{u}$ in $\{\mathfrak{u}>0\}$ holds in the strong sense, and consequently also in the weak sense with
    \begin{equation}\label{eqn:weakHeat}
    \int_{U_\infty} \partial_t \mathfrak{u} \phi  \,  dx \, dt = - \int_{U_\infty} \nabla \mathfrak{u} \cdot \nabla \phi  \,  dx \, dt
    \end{equation}
    for every continuous and spatially Lipschitz (uniformly in $t$) $\phi$ compactly supported in the open set $\{\mathfrak{u} > 0\} \cap U_\infty$.  
    The strong relation can be deduced exactly as in \cite{CaffarelliVazquez1995}*{Lemma 7.1} (this relies on the local version of Proposition \ref{p.semilinear-wellposed}\ref{part.attainment}). The weak form can be deduced via integration by parts.
    
     For $\delta > 0$ the function $\phi := \mathfrak{u}^\delta \eta$, with $\mathfrak{u}^\delta :=  \max\{\mathfrak{u}-\delta,0\}$, is spatially Lipschitz (uniformly in $t$) with ${\rm spt}(\phi) \subset \{\mathfrak{u} \geq \delta\} \cap {\rm spt}\, \eta$, hence admissible; using $\partial_t \mathfrak{u}   \mathfrak{u}^\delta = \partial_t \left[\tfrac12 (\mathfrak{u}^\delta)^2\right]$ a.e.\ and integrating by parts in time, we find
    \begin{equation}\nonumber
    \int_{U_\infty} \frac{1}{2}(\mathfrak{u}^\delta)^2\partial_t \eta   \, dx \, dt =  \int_{U_\infty} \mathfrak{u}^\delta \nabla \mathfrak{u}^\delta \cdot \nabla \eta \,  dx  \, dt + \int_{U_\infty} |\nabla \mathfrak{u}^\delta|^2 \eta  \, dx \,  dt,
    \end{equation}
    As $\delta \to 0$, $\mathfrak{u}^\delta \to \mathfrak{u}$ uniformly and $\grad \mathfrak{u}^\delta = \grad \mathfrak{u} {\bf 1}_{\{\mathfrak{u}>\delta\}} \to \grad \mathfrak{u} {\bf 1}_{\{\mathfrak{u}>0\}} = \grad \mathfrak{u}$ in $L^2({\rm spt} \, \eta)$, where the last equality holds a.e.\ because $\grad \mathfrak{u} = 0$ a.e.\ on $\{\mathfrak{u}=0\}$. This establishes \eqref{eqn:heatEqnRelation}. 
\end{proof}

\subsection{Perimeter estimate and $\chi$ compactness in $L^1$}\label{s.perimeter-chi-compactness}
We will prove a localized $L^1$ estimate for $\grad\chi_\ep$. Fix a non-negative test function $\eta \in C_c^\infty(U_\infty)$, let $d_\eta := \dist({\rm spt}\, \eta, \partial_P U_\infty) > 0$ and $T_\eta = \inf \{|b-a| : {\rm spt} \, \eta \subset U\times(a,b)\}$, and let $L_\eta$ be the bound for $\|\grad \mathfrak{u}_\ep\|_{L^\infty({\rm spt}\, \eta)}$ provided by \eqref{e.localEst} and \eqref{eqn:inftyBound}, which depends on $\eta$ only through $d_\eta$. Recalling the definition of $\chi_\eps$ in \eqref{eqn:chiEps} and using $Q \geq Q_{\min}$ together with the equation \eqref{e.bernoulli-parabolic-semilinear}, we find
\begin{equation}\label{eqn:chiEpsBV}
\begin{aligned}
\int_{U_\infty} \eta |\nabla \chi_\eps| \,  dx  \, dt & = \int_{U_\infty} 2\eta  \beta_\eps(\mathfrak{u}_\eps)  |\nabla \mathfrak{u}_\eps| \,  dx \, dt \\
& \leq \frac{2L_\eta}{Q_{\min}^2} \int_{U_\infty} \eta  Q^2 \beta_\eps(\mathfrak{u}_\eps) \, dx  \,dt\\
& = \frac{2L_\eta}{Q_{\min}^2}\int_{U_\infty} \big({-\grad \eta \cdot \grad \mathfrak{u}_\ep} - \eta  \partial_t \mathfrak{u}_\ep\big) \,  dx  \, dt\\
& \leq C_\eta  \|\eta\|_{L^2((0,\infty);H^1(U))},
\end{aligned}
\end{equation}
where 
the final inequality uses \eqref{e.dissipation-eps}. Note that only spatial regularity of $\eta$ is used, and that $C_\eta$ depends on $\eta$ only through $d_\eta$ and $T_\eta$.

To pass to the limit in $\eps$, we need (mild) regularity in time for $\chi_\eps$ to recover $L^1$ convergence. In our primary application within this paper, one can argue via monotonicity. However, a more general argument due to Weiss \cite{Weiss03}*{Proposition 4.1} yields $L^1$ convergence without invoking monotonicity. We localize Weiss' argument to obtain the following.

\begin{lemma}\label{lem:chiConverge}
Let $\chi_\eps$ be defined in \eqref{eqn:chiEps}. Then up to a subsequence (not relabeled), we have  
$\chi_\eps \to \chi \in L^1_{ loc}(U_\infty; \{0,1\})$ in $L^1_{ loc}(U_\infty)$. Furthermore, 
$$\int_{U_\infty} \eta |\nabla \chi| \, dx \,  dt \leq  C_\eta\|\eta\|_{L^2((0,\infty);H^1(U))} \quad \text{ for all non-negative $\eta\in C_c^\infty (U_\infty)$},$$
where $C_\eta$ depends on $\eta$ only through $d_\eta = {\rm dist}({\rm spt}\, \eta, \partial_P U_\infty)$ and $T_\eta = \inf \{|b-a| : {\rm spt} \, \eta \subset U\times(a,b)\}$. 
\end{lemma}

As the argument is, by and large, the same as that of \cite{Weiss03}*{Proposition 4.1}, we only sketch the proof in Appendix~\ref{s.proof-chiConverge}.

\subsection{Parabolic inner variation formula} \label{subsec:paraInnerVar}

We now have all the ingredients necessary to prove \pref{eps-inner-limit}.

\begin{proof}[Proof of \pref{eps-inner-limit}] Combining parts \ref{part.H1-continuity}, \ref{part.lipschitz}, and \ref{part.attainment} of Proposition \ref{p.semilinear-wellposed} with Lemmas \ref{lem:strongGrad} and  \ref{lem:chiConverge}, we may choose a subsequence of $\ep \to 0$ (not relabeled) so that $\mathfrak{u}_\ep$ converges to $\mathfrak{u}\in C(U_\infty)$ locally uniformly, $\grad \mathfrak{u}_\ep \to \grad \mathfrak{u}$ strongly in $L^2_{\rm loc}(U_\infty)$, $\partial_t \mathfrak{u}_\ep \rightharpoonup \partial_t \mathfrak{u}$ weakly in $L^2(U_\infty)$, and $\chi_\ep \to \chi \in L^1_{\rm loc}( U_\infty ;\{0,1\})$ pointwise a.e. and in $L^1_{\rm loc}(U_\infty)$. This proves \ref{part.eps-convergence} and allows us to take $\eps \to 0$ in \eqref{eqn:innerVarEps} 
to find that $(\mathfrak{u},\chi)$ satisfies \eqref{e.parabolic-inner-var}.

We show that ${\bf 1}_{\{\mathfrak{u} > 0\}} \leq \chi$. For any $(x_0,t_0)$ in $\{\mathfrak{u}>0\}$, the locally uniform convergence gives that, for $\ep>0$ sufficiently small, $\mathfrak{u}_\ep \geq \ep$ on a neighborhood of $(x_0,t_0)$. Thus $\chi_\ep = 2\mathcal{B}_\ep(\mathfrak{u}_\ep) = 1$ on this neighborhood for $\ep>0$ sufficiently small and so $\chi$, the a.e. limit of $\chi_\ep$, is also equal to $1$ there.
	Thus, all requirements of \dref{parabolic-inner} hold (except for $\nabla \mathfrak{u} \in L^\infty_{\textup{\rm loc}}(U_\infty)$ proven next), proving \ref{part.eps-inner}.

The estimates of \ref{part.eps-estimates} follow by passing to the limit in the corresponding estimates for $\mathfrak{u}_\ep$. The dissipation inequality follows from \eqref{e.dissipation-eps} and weak lower semi-continuity of the energy. The interior Lipschitz bound from Proposition \ref{p.semilinear-wellposed}\ref{part.lipschitz} passes to the limit by the locally uniform convergence. The weighted perimeter bound follows from \eqref{eqn:chiEpsBV} and the lower semicontinuity of the weighted total variation under $L^1_{\rm loc}$ convergence. 
\end{proof}

\subsection{Long time limit}\label{subsec:varTtoInfty}

We now let $t\to \infty$ in the parabolic problem to obtain an inner variational solution of the stationary Bernoulli problem, thus proving \tref{longtime-inner}.

\begin{proof}[Proof of \tref{longtime-inner}]
A localization argument in time for \eqref{e.parabolic-inner-var} gives
\begin{equation}\label{eqn:innerVarLocalized}
\begin{aligned}
		\int_{U} \Big[(|\nabla \mathfrak{u}(\cdot,t)|^2 + Q^2\chi(\cdot,t)) \grad \cdot \xi - 2 \nabla \mathfrak{u}(\cdot,t) \cdot D\xi \nabla \mathfrak{u}(\cdot,t) & \\
		+\grad (Q^2) \cdot \xi \chi(\cdot,t) - 2(\xi \cdot \grad \mathfrak{u}(\cdot,t))  \partial_t \mathfrak{u}(\cdot,t) \Big] &   dx  = 0
\end{aligned}
	\end{equation}
	for all $\xi\in C_c^1(U;\R^d)$ and a.e. $t$. We will choose an appropriate subsequence of $t_i \to \infty$ to pass to the limit in this relation.
	
\emph{Step 1: Existence of the locally uniform limit.} 
Consider any sequence $t\to \infty$. By the uniform bound on the Dirichlet energy from \eqref{eqn:diss2}, $\mathfrak{u}(\cdot,t)$ converges to some $u_\infty$ weakly in $H^1(U)$ up to a subsequence in $t$ (not relabeled). The uniform boundedness, interior Lipschitz estimate \eqref{eqn:LipEst}, and the Arzel\`a-Ascoli theorem show convergence of the sequence is locally uniform and that the limit inherits the estimate \eqref{eqn:LipEst}. In the next two steps, we argue that the sequence $t\to \infty$ can be chosen to also give convergence of $\nabla \mathfrak{u}(\cdot,t)$ and $\chi(\cdot,t)$.

\emph{Step 2: Strong convergence of the gradient.}
The relation \eqref{eqn:heatEqnRelation} implies
\begin{equation}\nonumber
    -\int_{U_\infty} \partial_t\mathfrak{u}\,  \mathfrak{u} \eta  \,  dx  \, dt -  \int_{U_\infty} \mathfrak{u} \nabla \mathfrak{u}\cdot \nabla \eta  \,  dx  \, dt = \int_{U_\infty} |\nabla \mathfrak{u}|^2 \eta  \,  dx  \, dt
    \end{equation}
    for $\eta\in C^1_c(U_\infty)$
    after an integration by parts in time.
    Localizing this relation in time, for a.e. $t\in (0,\infty)$ we have
    \begin{equation}\label{eqn:heatEqn}
    -\int_{U} \partial_t\mathfrak{u}(\cdot, t) \mathfrak{u}(\cdot, t) \eta \,  dx  -  \int_{U} \mathfrak{u}(\cdot, t) \nabla \mathfrak{u}(\cdot, t)\cdot \nabla \eta  \,  dx  = \int_{U} |\nabla \mathfrak{u}(\cdot, t)|^2 \eta   \, dx 
    \end{equation}
    for any $\eta\in C^1_c(U)$.
    
    Suppose that the sequence $t\to \infty$ is chosen so that 
    \begin{equation}\label{eqn:vanishingtime}
    \int_{U} |\partial_t \mathfrak{u}(\cdot,t)|^2 \, dx \to 0
    \end{equation} and \eqref{eqn:heatEqn} holds for each $t$ in the sequence.
    It follows that 
    \begin{equation}\label{eqn:limsupMe}
     \limsup_{t\to \infty}\int_{U} |\nabla \mathfrak{u}(\cdot, t)|^2 \eta   \, dx \leq  -  \int_{U} u_\infty \nabla u_\infty\cdot \nabla \eta  \,  dx,
     \end{equation}
     where we have used that $\mathfrak{u}(\cdot, t)$ converges to $u_\infty$ weakly in $H^1(U)$ and locally uniformly in $U$.
    Using \eqref{eqn:vanishingtime} along a subsequence (as in Step 1) to pass to the limit in the weak form of the heat equation \eqref{eqn:weakHeat}, we also find that $u_\infty$ is harmonic in its positivity set with 
    $$\int_{U}\nabla u_\infty \cdot \nabla \phi \, dx = 0$$
    for all Lipschitz $\phi$ with ${\rm spt}\, \phi \subset \subset \{u_\infty>0\}.$
    Taking $\phi  = \eta \max\{u_\infty-\delta,0\}$ and then sending $\delta \downarrow 0$ gives 
    $$-\int_{U} u_\infty \nabla u_\infty\cdot \nabla \eta  \,  dx = \int_{U} |\nabla u_\infty|^2 \eta   \, dx, $$
    which with \eqref{eqn:limsupMe} gives that $\nabla \mathfrak{u}(\cdot,t)\to \nabla u_\infty$ in $L^2_{\textup{\rm loc}}(U).$

\emph{Step 3: selection of good times.} We choose times so that \eqref{eqn:heatEqn} and \eqref{eqn:vanishingtime} hold (thereby ensuring Step 2 works) and $\chi(\cdot, t)$ converges to $\chi_\infty$ satisfying the desired perimeter bounds.  Define the sets $U_k := \{x \in U : {\rm dist}(x,\partial U) > 1/k\}$, $k \in \N$. Applying \eqref{eqn:chiEst2} for well-chosen test functions, we have that
$$\int_{U_k\times(i,i+1)} |\nabla \chi(x,t)| \, dx \, dt \leq C_k \quad  \hbox{ for } i\in \N,$$
for a constant independent of $i.$ By Markov's inequality, for any $\theta\in (0,1)$ and $i\in \N$ there is a subset $I_{i,\theta}^k \subset (i,i+1)$ such that
\begin{equation}\label{eqn:chiBoundy}
 \int_{U_k} |\nabla \chi(x,t)| \, dx \leq C_k/\theta \quad \text{ for all }t\in I^k_{i,\theta} \quad \text{ and }\quad  \mathcal{L}^1((i,i+1)\setminus I^k_{i,\theta})<\theta.
 \end{equation}
Similarly, we can find $I_i\subset (i,i+1)$ with $\mathcal{L}^1((i,i+1)\setminus I_{i})<1/2$ such that
 \begin{equation}\label{eqn:timeBound}
 \int_{U} |\partial_t \mathfrak{u}(x,t)|^2 \, dx \leq 2\int_{U\times(i,i+1)} |\partial_t \mathfrak{u}|^2 \, dx\, dt \quad \text{ for all }t\in I_{i}.
 \end{equation}
 Choose 
\[
t_i\in I_i\cap \bigcap_{k\leq i} I^k_{i,2^{-(k+1)}}
\]
 satisfying \eqref{eqn:diss2}, \eqref{eqn:innerVarLocalized}, \eqref{eqn:heatEqn}, and ${\bf 1}_{\{\mathfrak{u}(\cdot,t_i)>0\}}\leq\chi(\cdot,t_i)$ a.e.\ in $U$ (each of which holds for a.e.\ $t$),
 which is possible since the set intersection has measure at least $2^{-(i+2)}.$
 From $\partial_t \mathfrak{u}\in L^2(U_\infty)$ and \eqref{eqn:timeBound}, it follows that \eqref{eqn:vanishingtime} holds for $t_i\to \infty$ (and any subsequence). Applying Step 1 to $t_i$, we pass to a subsequence (not relabeled) along which $\mathfrak{u}(\cdot,t_i)\to u_\infty$ weakly in $H^1(U)$ and locally uniformly in $U$; by Step 2, also $\nabla\mathfrak{u}(\cdot,t_i)\to\nabla u_\infty$ in $L^2_{\rm loc}(U)$.
 Similarly, from \eqref{eqn:chiBoundy}
 $$\limsup_{i\to \infty}\int_{U_k} |\nabla \chi(x,t_i)| \, dx \leq C_k,$$
 for a new constant $C_k > 0$.
 Finally, applying compactness in $BV$, we diagonalize in $k$ to find 
 our desired sequence $t\to \infty$, a subsequence of $t_i$, satisfying the conclusions of Steps 1 and 2 and
  such that $\chi(\cdot, t) \to \chi_\infty$ in $L^1_{\textup{\rm loc}}(U)$ as $t\to \infty$ with
\[\int_{U_k} |\nabla \chi_\infty| \, dx \leq C_k \quad \hbox{for all } k\in \N.\]

\emph{Step 4: Stationary limit.}
 For that sequence, the previously derived compactness properties show that we may pass to the limit in \eqref{eqn:innerVarLocalized} to recover \eqref{e.stationary-inner-var}.
	Finally, since $\mathfrak{u}(\cdot,t)$ converges to $u_\infty$ locally uniformly, the inequality $\one_{\{\mathfrak{u}(\cdot,t) > 0 \}}\leq \chi(\cdot, t)$ passes to the limit $\one_{\{u_\infty > 0 \}}\leq \chi_\infty.$
	Consequently, $u_\infty$ satisfies Definition \ref{d.stationary-inner} as desired.
\end{proof}

\section{Limits in the viscosity sense}\label{s.viscosity-limits}

In this section we prove the comparison-based limit theorems. This includes the $\ep\to 0$ limit of the semilinear problem, which together with \pref{eps-inner-limit} will yield \tref{parabolic-existence}, and the long-time limits. In contrast to the variational setting, monotonicity plays an important role here, and we remind the reader that solutions of the semilinear equation found in Proposition \ref{p.semilinear-wellposed} are increasing or decreasing depending on whether the data $g$ is a subsolution or supersolution, respectively.

\subsection{Some technical lemmas}\label{s.technical-lemmas}
We prove two lemmas with the aim of understanding when strict ordering in the limit, as determined by $\prec_E$, can be lifted to an approximating sequence.

\begin{lemma}\label{l.limit-max}
    Suppose that $F$ is compact, $E_j$ are closed in $F$, and $w_j$ are continuous on $F$ with $w_j \to w$ uniformly on $F$.  Call $E^* = {\limsup}_{j\to\infty}^* E_j$. Then
    \[\max_{E^* \cap F} w = \limsup_{j \to \infty}\max_{E_j \cap F} w_j .\]
\end{lemma}
\begin{proof}
    Let $x^* = \mathop{\textup{argmax}}_{E^* \cap F} w$. Then there is a subsequence $j_k$ such that $E_{j_k} \ni x_{j_k} \to x^*$ so
    \[\max_{E^* \cap F} w = w(x^*) = \lim_{k \to \infty} w_{j_k}(x_{j_k}) \leq \limsup_{j \to \infty} \max_{E_j \cap F} w_j.\]
    On the other hand let $y_j = \mathop{\textup{argmax}}_{E_j \cap F} w_j$. Take a subsequence $y_{j_k} \to y \in E^* \cap F$ such that $w_{j_k}(y_{j_k}) \to \limsup_{j \to \infty} \max_{E_j \cap F} w_j$; by uniform convergence of $w_j$ to $w$, $w(y_{j_k}) \to \limsup_{j \to \infty} \max_{E_j \cap F} w_j$ as well. Thus
    \[\max_{E^* \cap F} w \geq w(y) = \lim_{k \to \infty} w(y_{j_k}) =  \limsup_{j \to \infty} \max_{E_j \cap F} w_j.\]
\end{proof}
\begin{lemma}\label{l.prec-ordering-convergence}
    Suppose that $F$ is compact, $E$ is closed in $F$, and $u$ and $v$ are continuous functions with $u \prec_E v$ on $F$. If $u_j \to u$ and $v_j \to v$ uniformly on $F$ and ${\limsup}_{j\to\infty}^*E_j \subset E$, then $u_j \prec_{E_j} v_j$ for $j$ sufficiently large.
\end{lemma}
\begin{proof}
Call $E^*=\limsup_{j\to\infty}^*E_j$; then $E^* \subset E$ by assumption. Note that $u \prec_E v$ on $F$ is equivalent to $\max_{E \cap F} (u - v) < 0$. By \lref{limit-max} 
applied to $w_j := u_j - v_j \to u - v=:w$ uniformly on $F$,
  \[0>\max_{E \cap F} (u - v) \geq \max_{E^* \cap F} (u - v) = \limsup_{j \to \infty} \max_{E_j \cap F} (u_j - v_j).\]
  So $\max_{E_j \cap F} (u_j - v_j) < 0$ for $j$ sufficiently large, i.e. $u_j \prec_{E_j} v_j$ on $F$.
\end{proof}
\subsection{Semilinear limit}\label{s.semilinear-limit}

Our next result shows that the $\ep \to 0$ limit of the semilinear approximations \eref{bernoulli-parabolic-semilinear} is a relaxed viscosity solution of the parabolic Bernoulli problem \eref{bernoulli-parabolic}.

\begin{proposition}\label{p.semilinear-limit-viscosity}
     Suppose that $\mathfrak{u}_{\ep_j}$ solve \eqref{e.bernoulli-parabolic-semilinear} in $U \times (0,T]$ and $\mathfrak{u}_{\ep_j} \to \mathfrak{u}$ locally uniformly in $U \times (0,T]$. Define 
     \[E:=\overline{\bigcup_{0<\kappa\le 1}{\limsup}_{j\to\infty}^*\{\mathfrak u_{\ep_j}>\kappa\ep_j\}}.\]
     Then $(\mathfrak{u},E)$ is a relaxed viscosity solution of \eqref{e.bernoulli-parabolic} in $U \times (0,T]$.
\end{proposition}
The supersolution part of this result was essentially proved in \cite{Kim03}*{Theorem 1.3} (see also \cite{CaffarelliVazquez1995}*{Sections 6--7} and \cite{Petrosyan}). The subsolution part is significantly trickier due to the possibility of degeneracy and that is the new contribution of Proposition~\ref{p.semilinear-limit-viscosity}.

\begin{proof} We check the relaxed subsolution property and the supersolution property of the limit.

\emph{Step 1: Relaxed Subsolution Property.} First we check $\{\mathfrak u>0\}\subset E$: if $\mathfrak u(x,t)>0$, the locally uniform convergence gives $\mathfrak u_{\ep_j}(x,t)>\ep_j$ for all large $j$, so $(x,t)\in{\limsup}_{j\to\infty}^*\{\mathfrak u_{\ep_j}>\ep_j\}\subset E$ (take $\kappa=1$). 

Let $V \times (a,b]$ be a parabolic cylinder and $\varphi$ a smooth strict supersolution of \eqref{e.bernoulli-parabolic} on ${\overline{V} \times [a,b]}$ with $\mathfrak u\prec_{E} \varphi$ on $\partial_P(V \times (a,b])$. We set the stage for a proof by contradiction to show that $\mathfrak u \prec_{E} \varphi$ on $V \times (a,b]$. 
    
    First, we claim that for $0 < \delta \leq \delta_0$ sufficiently small, ${(\varphi-\delta)}$ is still a smooth strict supersolution of \eqref{e.bernoulli-parabolic} on ${\overline V \times [a,b]}$ with $\mathfrak u\prec_{E} \varphi-\delta$ on $\partial_P(V \times (a,b])$. The strict supersolution property for sufficiently small $\delta>0$ is immediate from \dref{parabolic-strict-sub}.  Since $E \cap \partial_P(V \times (a,b])$ is compact and $\varphi-\mathfrak u$ is a positive continuous function on that set, it follows that $(\varphi-\delta)-\mathfrak u$ is also positive on $E \cap \partial_P(V \times (a,b])$ for sufficiently small $\delta>0$.  

    Next, consider $\varphi_{\ep,{\hat\theta}} := \Psi_{\ep,{\hat\theta}}(\varphi-\delta)$, where $\Psi_{\ep,{\hat\theta}}$ is defined in \eqref{e.profile-ep-def}. By \lref{well-prepared}\ref{part.gep-2}, $|\varphi_{\ep,{\hat\theta}} -(\varphi-\delta)_+| \leq \ep$ and $\varphi_{\ep,{\hat\theta}} > \kappa_{\hat\theta} \ep$ for some $\kappa_{\hat\theta} \in (0,1)$. By \lref{well-prepared}\ref{part.gep-3}, there is ${\hat\theta}_0>1$ depending on the strict supersolution property of $\varphi-\delta$ so that whenever $1 < {\hat\theta} \leq {\hat\theta}_0$, $\varphi_{\ep,{\hat\theta}}$ is a (stationary) strict classical supersolution of the semilinear equation \eref{bernoulli-parabolic-semilinear}. 
    
   We claim that, for any $1 < {\hat\theta} \leq {\hat\theta}_0$ and $\ep_j$ sufficiently small
     \begin{equation}\label{eqn:contraDo}
    {\mathfrak u_{\ep_j}\leq \varphi_{\ep_j,{\hat\theta}}} \quad \text{ in } V \times (a,b].
    \end{equation}     
Define $E_j := \{\mathfrak u_{\ep_j} > \kappa_{\hat\theta}\ep_j\}$.  Since $\varphi_{\ep_j,{\hat\theta}} > \kappa_{\hat\theta}\ep_j$,
      \begin{equation}\label{e.kappaj-consequence}
      \hbox{$\mathfrak{u}_{\ep_j} <\varphi_{\ep_j,{\hat\theta}}$ on the complement of $E_j$.}
      \end{equation}
      Invoking the uniform convergence of $\varphi_{\ep_j,{\hat\theta}}$ to $(\varphi-\delta)_+$, of $\mathfrak u_{\ep_j}$ to $\mathfrak u$, and the fact that $E \supset {\limsup}_{j\to\infty}^*E_j$ --- which implies 
      \[{\limsup}_{j\to\infty}^*(E_j\cap \partial_P(V \times (a,b])) \subset E\cap \partial_P(V \times (a,b])\]
      --- we may apply \lref{prec-ordering-convergence} to find that 
    \begin{equation}\nonumber %
    {\mathfrak u_{\ep_j} \prec_{E_j} \varphi_{\ep_j,{\hat\theta}}}\quad \text{ on } \partial_P(V \times (a,b])
    \end{equation}
     for $\ep_j$ sufficiently small. Combining this with \eqref{e.kappaj-consequence} we get $\mathfrak u_{\ep_j} <  \varphi_{\ep_j,{\hat\theta}}$ on $\partial_P(V \times (a,b])$. The classical comparison principle for \eref{bernoulli-parabolic-semilinear} then implies \eqref{eqn:contraDo}. 
     
     We are now ready to show $\mathfrak u \prec_{E} \varphi$ on $V \times (a,b]$. Suppose towards a contradiction that this does not hold. Then there exists $(x_0,t_0) \in E \cap (V \times (a,b])$ with 
    \begin{equation}\label{eqn:contraMe}
    \mathfrak u(x_0,t_0) \geq \varphi(x_0,t_0).
    \end{equation}  
    We consider two cases; either (i) $\varphi(x_0,t_0) >0$ or (ii) $\varphi(x_0,t_0) \leq 0$. In case (i) choose $\delta>0$ smaller if necessary so that $\varphi(x_0,t_0) - 2 \delta >0$.   

\emph{Case (i):} For any ${\hat\theta} \in (1,{\hat\theta}_0]$, 
      \begin{equation}\label{eqn:contradicted}
      0\geq \mathfrak u_{\ep_j}(x_0,t_0) -\varphi_{\ep_j,{\hat\theta}}(x_0,t_0) \to \mathfrak u(x_0,t_0)-(\varphi(x_0,t_0) - \delta)>0,
      \end{equation} where the first inequality follows from \eqref{eqn:contraDo} and the second from \eqref{eqn:contraMe}, and we derive a contradiction for $\ep_j>0$ sufficiently small.  
      
\emph{Case (ii):} Since $\varphi(x_0,t_0) \leq 0$, $\varphi - \delta < -\delta/2$ in a neighborhood ${\mathcal{N}\subset V\times (a,b]}$ of $(x_0,t_0)$. Since $(x_0,t_0) \in E$ and $\mathcal{N}$ is an open neighborhood, there is $(y_0,s_0) \in \mathcal{N}$ and $\kappa_0\in(0,1]$ such that  
\begin{equation}\label{e.point-kappa-choice}
(y_0,s_0) \in {\limsup}_{j\to\infty}^* \{\mathfrak{u}_{\ep_j} > \kappa_0 \ep_j\}.
\end{equation}
   Now choose ${\hat\theta} \in (1,{\hat\theta}_0]$ sufficiently close to $1$ so that $\kappa_{\hat\theta} \leq \frac{\kappa_0}{2}$; this is possible since, by \lref{profile-super}\ref{part.profile-super-f}, $\kappa_{\hat\theta} \searrow 0$ monotonically as ${\hat\theta} \searrow 1$. By \lref{profile-super}\ref{part.profile-super-e}, \eqref{e.profile-ep-def}, and the choice of $\kappa_{\hat\theta}\leq \frac{1}{2}\kappa_0$ above, we have for all $\ep_j>0$ sufficiently small
     \[\varphi_{\ep_j,{\hat\theta}} \leq \tfrac{3}{2}\kappa_{\hat\theta} \ep_j \leq \tfrac{3}{4}\kappa_0\ep_j\quad \text{ in }\mathcal{N}.\]
      On the other hand, by \eqref{e.point-kappa-choice}, there is a sequence $\ep_{j_k} \to 0$ and $(y_k,s_k) \in \{\mathfrak{u}_{\ep_{j_k}} > \kappa_0\ep_{j_k}\}$ with $(y_k,s_k) \to (y_0,s_0)$. Thus, for $k$ sufficiently large $(y_k,s_k) \in \mathcal{N} \cap \{\mathfrak{u}_{\ep_{j_k}} > \kappa_0 \ep_{j_k}\}$ and 
      \[\mathfrak u_{\ep_{j_k}}(y_k,s_k) > \kappa_0\ep_{j_k} >\tfrac{3}{4}\kappa_0\ep_{j_k} \geq {\varphi_{\ep_{j_k},{\hat\theta}}(y_k,s_k)},\] 
      and we derive a contradiction with \eqref{eqn:contraDo} for $k$ sufficiently large.

\emph{Step 2: Supersolution Property.} We now check the supersolution condition, which is similar to the above argument, but a bit more straightforward.
Suppose $V \times (a,b] \subset\subset U \times (0,T]$ and $\varphi \in C^\infty(\overline{V} \times [a,b])$ is a smooth strict subsolution of \eref{bernoulli-parabolic} in $\overline{V} \times [a,b]$ with $\varphi \prec \mathfrak u$ on $\partial_P(V \times (a,b])$. We want to show $\varphi \prec \mathfrak u$ in $V \times (a,b]$. Suppose this fails; then there is a first touching point (from below) denoted by $(x_0,t_0)$ with $\varphi(x_0,t_0) = \mathfrak u(x_0,t_0)$.

As in Step 1, $\varphi+ \delta$ is also a smooth strict subsolution for all small $\delta$. Note that $\varphi+ \delta \prec \mathfrak u$ is equivalent to $\varphi+ \delta \prec_{\overline{\{\varphi+\delta>  0\}}} \mathfrak u$. Also note that, by the subsolution condition
\[|\nabla \varphi| > Q \quad \text{ on }\partial \{\varphi>0\}\cap\left(\overline{V} \times [a,b]\right),\]
we have
\begin{equation}\label{eqn:containerMe}
{\limsup_{\delta\to 0}}^* \overline{\{\varphi+\delta>  0\}}\cap \left(\overline{V} \times [a,b]\right) = \overline{\{\varphi>  0\}}\cap \left(\overline{V} \times [a,b]\right).
\end{equation}
We apply Lemma \ref{l.prec-ordering-convergence} to conclude for $\delta>0$ sufficiently small that 
 \[\varphi +\delta \prec \mathfrak u  \quad \text{ on }\partial_P(V \times (a,b]).\]

Consider $\varphi_{\ep,\theta} := (\Phi_{\ep,\theta}(\varphi+\delta))_+$. By \lref{well-prepared}\ref{part.gep-3}, there is ${\theta}_0<1$ depending on the strict subsolution property of $\varphi+\delta$ so that whenever ${\theta}_0 \leq {\theta} < 1$, $\varphi_{\ep,{\theta}}$ is a (stationary) viscosity subsolution of the semilinear equation \eref{bernoulli-parabolic-semilinear}; we choose and fix such a $\theta$. Invoking the uniform convergence of $\varphi_{\ep_j,\theta}$ to $(\varphi+\delta)_+$, of $\mathfrak u_{\ep_j}$ to $\mathfrak u$, and  \lref{well-prepared}\ref{part.gep-5} and \eqref{eqn:containerMe}, we may again apply \lref{prec-ordering-convergence} to find that $\varphi_{\ep_j} \prec \mathfrak u_{\ep_j}$ on $\partial_P(V \times (a,b])$ for $\ep_j$ sufficiently small. Application of the standard comparison principle for \eqref{e.bernoulli-parabolic-semilinear} then yields $\varphi_{\ep_j} \leq \mathfrak u_{\ep_j}$ on $V \times (a,b]$, but for small $\eps_j$ this contradicts the fact that $\varphi(x_0,t_0)+\delta>\mathfrak u(x_0,t_0)$ (cf. \eqref{eqn:contradicted}).
\end{proof}

In the monotone increasing case, we recover a standard viscosity solution.

\begin{corollary}\label{cor:increasing}
    Suppose that $\mathfrak{u}_{\ep_j}$ solve \eqref{e.bernoulli-parabolic-semilinear} as in Proposition \ref{p.semilinear-wellposed} with $g$ a strict subsolution and let $(\mathfrak{u},E)$ be as in \pref{semilinear-limit-viscosity}.
    Then $\mathfrak{u}$ is monotonically increasing in $t$, is continuous on $\overline{U}_\infty$ with $\mathfrak{u} = {g}_+$ on $\partial_PU_\infty$, $E \cap (U\times(0,T]) = \overline{\{\mathfrak u>0\}} \cap (U\times(0,T])$, and $\mathfrak u$ is a viscosity solution of \eqref{e.bernoulli-parabolic} in $U \times (0,T]$.
\end{corollary}

This is an immediate consequence of  \pref{semilinear-limit-viscosity} and \pref{monotone-nondegen-relaxed}, as Proposition \ref{p.semilinear-wellposed} parts \ref{part.attainment} and \ref{part.trace} ensure that the hypothesis $E_0 \cap U \subset \overline{\{\mathfrak{u}(\cdot,0)>0\}} = \overline{\{g_+>0\}}$ of  \pref{monotone-nondegen-relaxed} holds. Indeed, if $x \in U \setminus \overline{\{g>0\}}$, then $g(x)<0$, since a zero of $g$ with $g \leq 0$ nearby would be an interior maximum, contradicting $|\grad g|^2 \geq (1+\delta_0)Q^2$ on $\{g=0\}$; so \pref{semilinear-wellposed}\ref{part.trace}, applied to a small closed ball around $x$, gives $x \notin E_0$. \pref{semilinear-wellposed}\ref{part.attainment} gives the continuity up to $\partial_PU_\infty$ and the parabolic boundary data.

\subsection{Proof of \tref{parabolic-existence}}\label{s.proof-parabolic-existence}

The pieces are now in place to assemble the existence theorem for the parabolic flow \eref{bernoulli-parabolic}. 

\begin{proof}[Proof of \tref{parabolic-existence}]
Let $\ep_j \downarrow 0$ and $(\mathfrak{u},\chi)$ be the subsequence and limit pair of \pref{eps-inner-limit}. By Proposition \ref{p.semilinear-wellposed}\ref{part.attainment}, $\mathfrak{u}$ extends to $\mathfrak{u} \in C(\overline U \times [0,\infty))$ with $\mathfrak{u} = {g}_+$ on $\partial_P U_\infty$, and $t\mapsto\mathfrak{u}(x,t)$ is nondecreasing (resp.\ nonincreasing) for each $x \in \overline{U}$ in the case of strict subsolution (resp.\ supersolution) Dirichlet data. The pair $(\mathfrak{u},\chi)$ is an inner variational solution of \eref{bernoulli-parabolic} and satisfies the desired estimates by \pref{eps-inner-limit} \ref{part.eps-inner} and \ref{part.eps-estimates}. When $g$ is a strict subsolution, $\mathfrak{u}$ is a viscosity solution of \eref{bernoulli-parabolic} by Corollary \ref{cor:increasing}; when $g$ is a strict supersolution, $\mathfrak{u}$ is a relaxed viscosity solution by \pref{semilinear-limit-viscosity}.
\end{proof}

\subsection{Long time limit}\label{s.long-time-limit}

Here we characterize the $t \to \infty$ limit of the parabolic flow \eqref{e.bernoulli-parabolic}. We verify that the appropriate sub/supersolution notions are preserved in the infinite time limit and show how to rule out degeneracy of the limit.

We begin by showing the infinite-time limit of a relaxed subsolution of the parabolic flow \eqref{e.bernoulli-parabolic} is a relaxed subsolution of the stationary problem \eqref{e.bernoulli-basic}.

\begin{proposition}\label{p.subsolution-limit}
    Let $(\mathfrak u,E)$ be a bounded relaxed subsolution of \eqref{e.bernoulli-parabolic} in $U_{\infty}$ such that $\mathfrak u(\cdot,t) \to u_\infty$ locally uniformly in $U$ as $t \to \infty$. Define $E_\infty := {\limsup}_{t\to \infty}^*E_t$. Then $(u_\infty,E_\infty)$ is a relaxed subsolution of \eqref{e.bernoulli-basic} in $U$.
\end{proposition}

\begin{proof}%
The uniform convergence of $\mathfrak{u}(\cdot,t)\to u_\infty$ ensures that $\{u_\infty > 0 \}\subset E_\infty,$ as is necessary for a relaxed solution. We now verify the necessary PDE relations.

    Let $\varphi$ be a smooth test function touching $u_\infty$ from above at $x_0 \in E_\infty \cap U$. Suppose, towards a contradiction, that
    \[\Delta\varphi(x_0)<0 \ \hbox{ and if $\varphi(x_0) = 0$ then} \ |\grad \varphi(x_0)| < Q(x_0).\]
    By choosing $r>0$ sufficiently small, we can assume, without loss of generality, that $\overline{B_r(x_0)} \subset U$, $u_\infty \prec_{E_\infty} \varphi$ on $\partial B_r(x_0)$ (recall Remark~\ref{r.strict-touching}), and $\varphi$ is a smooth strict supersolution of \eqref{e.bernoulli-basic} in $\overline{B_r(x_0)}$.  Choose $a_1>0$ sufficiently small so that $\varphi+a$ is still a smooth strict supersolution of \eqref{e.bernoulli-basic} in $\overline{B_r(x_0)}$ with $u_\infty \prec_{E_\infty} (\varphi+a)$ on $\partial B_r(x_0)$ for all $|a| \leq a_1$. Note also that $u_\infty \prec_{E_\infty} (\varphi+a)$ on $\overline{B_r(x_0)}$ for $a>0$, since $u_\infty \leq \varphi$ on $E_\infty \cap \overline{B_r(x_0)}$.
    
    Apply \lref{prec-ordering-convergence} and the definition of $E_\infty = {\limsup}_{t\to\infty}^* E_t$ to find that there exists $T$ sufficiently large so that for all $t \geq T$
\begin{equation}\label{e.prec-br-bd}
\mathfrak{u}(\cdot,t) \prec_{E_t} (\varphi - a_1) \ \hbox{ on } \ \partial B_r(x_0)
\end{equation}
    and, also for all $t \geq T$,
\begin{equation}\label{e.prec-br-in}
 \mathfrak{u}(\cdot,t) \prec_{E_t} (\varphi + a_1) \ \hbox{ on } \ \overline{B_r(x_0)}.
 \end{equation}
Let $\zeta:\R \to \R$ be a smooth monotone decreasing function with $\zeta(s) = 1$ for $s \leq 0$ and $\zeta(s) = -1$ for $s \geq 1$ and $|\zeta'|  \leq C$. Then define for any $\eta>0$ 
        \[ \psi(x,t) = \varphi(x) + a_1 \zeta(\eta(t-T)).\]
   Since $|\partial_ta_1 \zeta(\eta(t-T))| \leq C\eta a_1$ we can choose $\eta>0$ sufficiently small so that $\psi$ is a smooth strict supersolution of \eqref{e.bernoulli-parabolic} in $\overline{B_r(x_0)} \times [T,\infty)$.  By \eqref{e.prec-br-bd}, \eqref{e.prec-br-in}, and the definition of $\psi$ we find
   \[ \mathfrak{u}\prec_{E} \psi \ \hbox{ on } \ \partial_P(B_r(x_0) \times (T,\infty)). \]
   Therefore, by the relaxed subsolution property of $(\mathfrak{u},E)$ (Definition \ref{d.parabolic-relaxed-subsolution}), $\mathfrak{u} \prec_E \psi$ in $B_r(x_0) \times (T,\infty)$.
   
   Since $x_0 \in E_\infty$, there are $t_j \to \infty$ and $x_j \in E_{t_j}$ with $x_j \to x_0$. Taking a subsequence we may assume $t_j \geq T + 1/\eta$ and $x_j \in B_r(x_0)$.  Since $(x_j,t_j) \in E$ and $\zeta(\eta(t_j-T)) = -1$,
   \[\mathfrak{u}(x_j,t_j) < \psi(x_j,t_j) = \varphi(x_j) - a_1.\]
   Letting $j \to \infty$, using the uniform convergence of $\mathfrak{u}(\cdot,t)$ to $u_\infty$ on $\overline{B_r(x_0)}$ and the continuity of $u_\infty$ and $\varphi$, we find that
   \[u_\infty(x_0) \leq \varphi(x_0) - a_1 < \varphi(x_0)=u_\infty(x_0), \]
which is a contradiction. \end{proof}

Next, we show that the infinite-time limit of a supersolution of the parabolic flow \eqref{e.bernoulli-parabolic} is a supersolution of the stationary problem \eqref{e.bernoulli-basic}.

\begin{proposition}\label{p.supersolution-limit}
    Let $\mathfrak u$ be a bounded viscosity supersolution of \eqref{e.bernoulli-parabolic} in $U \times (0,\infty)$ such that $\mathfrak{u}(\cdot,t) \to u_\infty$ locally uniformly in $U$ as $t \to \infty$. Then $u_\infty$ is a viscosity supersolution of \eqref{e.bernoulli-basic} in $U$.
\end{proposition}

The argument is similar to, but simpler than, the one for the previous proposition; a proof is given in Appendix~\ref{s.proof-supersolution-limit}.

As observed previously, the subsolution condition may degenerate in the limit which is why we only recover a relaxed subsolution in the conclusion of \pref{subsolution-limit}. However, in the monotone increasing case, degeneracy can be ruled out in the long-time limit by pure viscosity theoretic arguments, allowing us to conclude that the long-time limit is a viscosity subsolution in the usual sense.

\begin{corollary}\label{c.monotone-long-time}
    Let $\mathfrak u$ be a bounded viscosity subsolution of \eqref{e.bernoulli-parabolic} in $U \times (0,\infty)$ which is monotonically increasing in time and such that $\mathfrak u(\cdot,t) \nearrow u_\infty$ locally uniformly in $U$ as $t \to \infty$. Then $u_\infty$ is a viscosity subsolution of \eqref{e.bernoulli-basic} in $U$.
\end{corollary}
\begin{proof}
    Apply \pref{subsolution-limit} to get $(u_\infty,E_\infty)$ is a relaxed subsolution of \eqref{e.bernoulli-basic} in $U$. Recall $E_T = \overline{\{u>0\}} \cap \{t = T\}$.  We claim that $E_T \subset \overline{\{u_\infty>0\}}$ for all $T>0$ and then
    \[E_\infty ={\limsup_{t \to \infty}}^*E_t \subset \overline{\{u_\infty>0\}}.\]
    So $u_\infty$ is a subsolution of \eqref{e.bernoulli-basic} in $U$.  If $(x,t) \in E_T$ there is $(x_k,t_k) \to (x,t)$ with $u(x_k,t_k)>0$. By monotonicity $u_\infty(x_k) \geq u(x_k,t_k) >0$ and so $x \in \overline{\{u_\infty>0\}}$.
\end{proof}

In summary, if $g$ is a smooth strict subsolution and  $(\mathfrak{u},E)$ is as in \pref{semilinear-limit-viscosity}, then Corollary~\ref{cor:increasing} implies $\mathfrak u$ is increasing in time and a viscosity solution of \eref{bernoulli-parabolic}. By \pref{supersolution-limit}, the infinite-time limit $u_{\infty}$ is a supersolution of \eref{bernoulli-basic}. Combining this fact with \cref{monotone-long-time} shows $u_{\infty}$ is a viscosity solution of \eref{bernoulli-basic} when the Dirichlet data $g$ of \eref{bernoulli-parabolic} is a smooth strict subsolution. 

When $g$ is a smooth strict supersolution, we obtain a monotone decreasing limit of solutions of \eref{bernoulli-parabolic}, and it is not clear to us whether degeneracy can be ruled out by pure viscosity theory arguments. We recover the subsolution property of the limit from the inner variational property, applying the theory of \cite{KriventsovWeiss} via \lref{inner-variational-implies-subsoln}.

\begin{lemma}\label{l.subsolution-inner-limit}
    Let $(\mathfrak{u},\chi)$ be a bounded inner variational solution and viscosity supersolution of \eqref{e.bernoulli-parabolic} satisfying \eqref{eqn:diss2}--\eqref{eqn:chiEst2} and $\mathfrak{u}(\cdot,t) \to u_\infty$ locally uniformly in $U$ as $t\to\infty$. Then there is a sequence $t_i\to \infty$ such that $\chi(\cdot,t_i) \to \chi_\infty$ in $L^1_{\textup{\rm loc}}(U)$ and $(u_\infty,\chi_\infty)$ is both an inner variational solution and a viscosity solution of \eqref{e.bernoulli-basic}.
\end{lemma}

We remark that if $(\mathfrak{u},\chi)$ is monotone in time, the convergence hypothesis, $\mathfrak{u}(\cdot,t) \to u_\infty$ locally uniformly in $U$ as $t\to\infty$, holds for free since the $t\to \infty$ limit is unique and must be locally uniform by Arzel\`a--Ascoli and the Lipschitz estimate \eqref{eqn:LipEst}.

\begin{proof}[Proof of Lemma \ref{l.subsolution-inner-limit}]
    By \tref{longtime-inner} there is a subsequence $t_i \to \infty$ so that $\mathfrak{u}(\cdot,t_i) \to u_\infty$  in $H^1_{\rm loc}(U)$, $\chi(t_i)\to \chi_\infty$ in $L^1_{\textup{\rm loc}}(U)$, and $(u_\infty, {\chi_\infty})$ is an inner variational solution of \eqref{e.bernoulli-basic}. The desired interior Lipschitz bound is contained in \tref{longtime-inner}. \lref{inner-variational-implies-subsoln} implies that $u_\infty$ is a viscosity subsolution of \eqref{e.bernoulli-basic}, and \pref{supersolution-limit} implies that $u_\infty$ is a viscosity supersolution of \eqref{e.bernoulli-basic}.
\end{proof}

Note that this approach of establishing non-degeneracy also works in the increasing case.  However, since the proof of \cref{monotone-long-time} above is self-contained, we opted to avoid using the results of \cite{KriventsovWeiss} as a black box in that setting.

\section{Proof of the main theorem}\label{s.conclusion}
In this section we assemble the results from the previous sections to complete the proof of \tref{main}. We begin with the proof of \tref{main} \ref{part.inner}, first tackling smallest supersolutions and then largest subsolutions. While the proofs are analogous, there are notable technical differences when dealing with non-degeneracy in the case of largest subsolution. In both cases, we must also contend with the strict ordering requirement in the comparison principle \tref{comparison}. This is managed by introducing an additional $1$-parameter family of shifts in the data. 
We will also establish \tref{main} \ref{part.visc} within the proof of \tref{main} \ref{part.inner}. The proof of \tref{main} \ref{part.downward}, which is a short argument independent of the rest of the paper, is carried out in the final subsection.

Let us briefly sketch the argument behind the proof of \tref{main} \ref{part.inner} in the smallest supersolution case. In the preceding sections, we have established that for the solution $\mathfrak{u}(t)$ of the parabolic Bernoulli problem with initial data given by a smooth strict subsolution $g$, the limit $u_\infty = \lim_{t \to \infty} \mathfrak{u}(t)$ is both an inner variational and a viscosity solution of \eqref{e.bernoulli-basic} satisfying $u_{\infty} \geq g_+$ by monotonicity. Thus, if $u$ is the smallest supersolution of \eqref{e.bernoulli-basic} above $g$, it follows that $u \leq u_\infty$. To establish the reverse inequality, we would like to use the comparison principle for the parabolic flow, Theorem~\ref{t.comparison}, but we lack strict ordering on the parabolic boundary.  This is a real issue, as we expect that $u$ and $u_\infty$ may not agree when a slight perturbation of $g$ upwards or downwards results in a (discontinuous) jump in the corresponding minimal supersolution. However, this issue is non-generic, and we show it is possible to perturb it away by considering a monotone family of shifted data $g+a$ for $a < 0$, the corresponding parabolic solutions $\mathfrak{u}_a$, and their long-time limits $u_{a,\infty}$.  
We recover the desired inner variational property by sending $a\uparrow 0$. To ensure strict ordering on the parabolic boundary, we will also need to assume $g$ is positive on $\partial U$ in the smallest supersolution case. See Figure \ref{fig:precFail} for when the strict ordering might fail if $g$ is merely non-negative on $\partial U$.

\subsection{The inner variational property: smallest supersolutions}\label{s.inner-smallest}

\begin{figure}
    \begin{center}
\includegraphics[width=.48\textwidth,trim=45 10 20 60,clip]{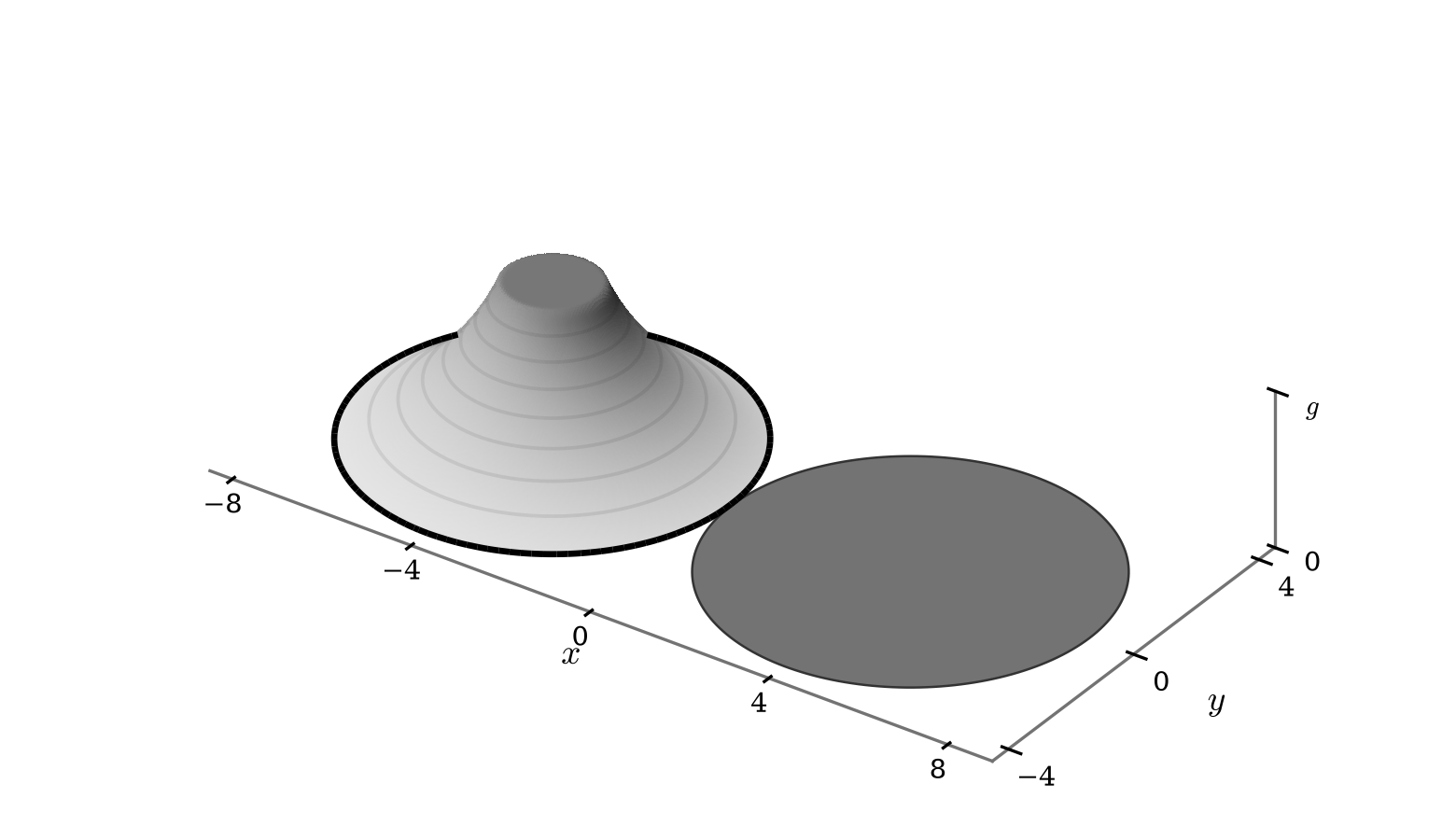}
\includegraphics[width=.48\textwidth,trim=45 10 20 60,clip]{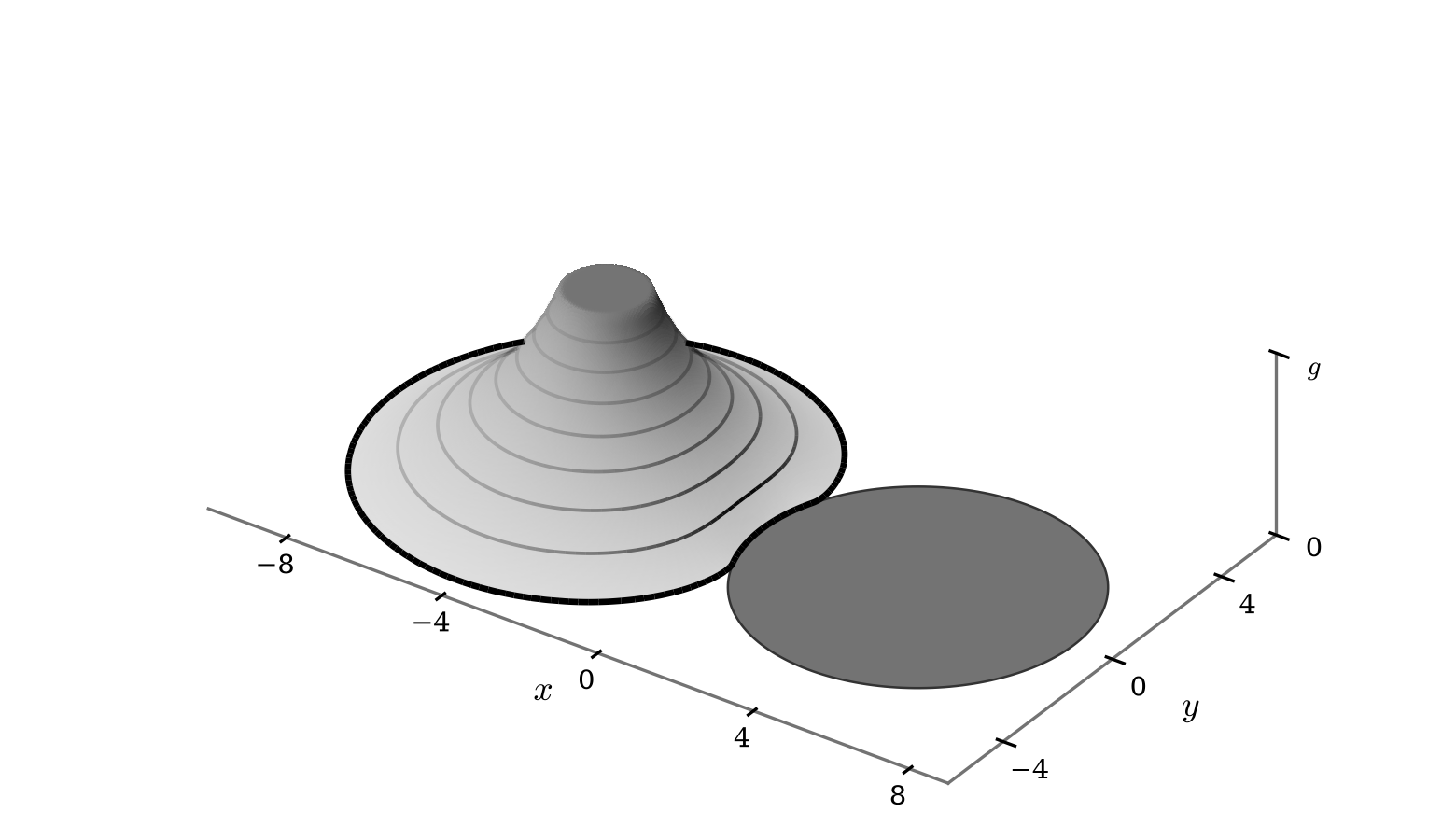}
\end{center}
    \caption{With boundary data $g = 0$ in $B_4((0,4))$ and $g = c_{1}$ (left diagram) or $g = c_{2}$ (right diagram) in $B_1((0,-4))$ with $0<c_{1} < c_{2}$, we have found the smallest supersolutions $u_1$ and $u_2$, respectively. Since the free boundary is able to meet the domain boundary (a possibility since $g = 0$ within $\partial U$), the ordering $u_1\prec u_2$ fails on $\partial U$.}
    \label{fig:precFail}
\end{figure}

\begin{proposition}\label{p.inner-smallest}
Let ${g}$ be a smooth strict subsolution of \eref{bernoulli-basic} on $\overline{U}$ in the sense of \dref{strict-sub} with $g|_{\partial U}>0$, and let $u$ be the smallest supersolution of \eref{bernoulli-basic} above ${g}$ in the sense of \eref{perron-def-super}. Then $u$ is both an inner variational and viscosity solution of \eref{bernoulli-basic} in $U$.
\end{proposition}

\begin{proof}
Let $v$ be any supersolution of \eqref{e.bernoulli-basic} in $U$ above $g_+$, i.e., $v\in \mathcal{S}_g$ from \eqref{eqn:S_g}. Recall from \dref{strict-sub} that for $|a|\leq a_0/2$, the shifts $(g+a)_+$ are also strict subsolutions with parameters $a_0$ replaced by $a_0/2$ and the same $\delta_0$.

For $-a_0/2<a<0$, by \tref{parabolic-existence}, there exists $\mathfrak{u}_a(x,t)$  an inner variational and viscosity solution of the parabolic Bernoulli problem \eqref{e.bernoulli-parabolic} that is monotone increasing in time and such that $\mathfrak{u}_a(x,t) = (g(x)+a)_+$ on $\partial_P U_\infty$. 
Recall that $g>0$ on $\partial U$. Combining this with the continuity of $\mathfrak{u}_a$ on the parabolic boundary (\tref{parabolic-existence}\ref{part.continuity}) and the fact that $\{\mathfrak{u}_a>0\}$ does not jump outward (which follows from  \pref{semilinear-wellposed}\ref{part.trace}), we have $\mathfrak{u}_a  = (g+a)_+ \prec g  \leq v$ in a neighborhood of $\partial_P U_\infty$.  Consequently, monotonicity, the comparison principle Theorem \ref{t.comparison}, and Remark \ref{r.statToPara}
imply
\[(g+a)_+\leq \mathfrak{u}_a \prec v \quad \text{ in } U_\infty.\]
Letting $t\to \infty$ (the full limit exists by monotonicity, see the remark after \lref{subsolution-inner-limit}), we find
\[(g+a)_+\leq u_{a,\infty} : = \lim_{t\to \infty}\mathfrak{u}_a(\cdot,t) \leq v \quad \text{ in } U.\]
By \tref{longtime-inner} and \tref{longtime-viscosity}\ref{part.longtime-viscosity-1}, $u_{a,\infty}$ is an inner variational and viscosity solution of \eqref{e.bernoulli-basic} that satisfies the Lipschitz estimate \eqref{eqn:inftyLip} and perimeter bound \eqref{eqn:inftyPerim} for its associated `positivity set' $\chi_{a,\infty}$; both estimates are uniform in $a$. Furthermore, the same argument as above shows $a \mapsto u_{a,\infty}$ is monotone increasing (i.e. take $v = u_{a',\infty}$ for $a'>a$). Thus, by Lemma \ref{l.super-stability}\ref{part.stability-super} and Lemma \ref{l.inner-compactness-stability}, there is a subsequence $a_k \uparrow 0$ such that the limit $\hat u := \lim_{a_k\uparrow 0} u_{a_k,\infty}$ exists and $\hat{u}$ is an inner variational solution and viscosity supersolution of \eqref{e.bernoulli-basic}. By \lref{super-stability}\ref{part.stability-pair} $\hat u$ is also a viscosity subsolution. Taking the limit in the previous displayed inequality yields
\[g_+\leq \hat u = \lim_{a_k\uparrow 0} u_{a_k,\infty} \leq v \quad \text{ in } U .\]
 The modulus $\varpi$ of \eqref{e.attainment} is uniform in $a$, so $\hat u \in C(\overline U)$ and hence $\hat{u} \in \mathcal{S}_g$. Taking the infimum over $v\in \mathcal{S}_g$, see \eqref{e.perron-def-super}, we find $\hat u \leq u$. Since $\hat{u}\in \mathcal{S}_g$, we conclude that $\hat u = u$ and obtain the theorem.
\end{proof}

\subsection{The inner variational property: largest subsolutions}\label{s.inner-largest}

\begin{proposition}\label{p.inner-largest}
Let ${g}$ be a smooth strict supersolution of \eref{bernoulli-basic} on $\overline{U}$ in the sense of \dref{strict-sub}, and let $u$ be the largest subsolution below ${g}$ in the sense of \eref{perron-def-sub}. Then $u$ is both an inner variational and viscosity solution of \eref{bernoulli-basic} in $U$.
\end{proposition}
Note we do not need to require $g|_{\partial U}>0$ in this case. The proof is very similar to the smallest supersolution case, with the key technical difference being that subsolutions are not stable under uniform limits. Thus, the stability argument depends on Lemma \ref{l.inner-variational-implies-subsoln}.

\begin{proof}[Proof of Proposition \ref{p.inner-largest}]

Let $v$ be any subsolution of \eqref{e.bernoulli-basic} in $U$ below $g_+$, i.e., $v\in \mathcal{S}^g$ from \eqref{eqn:S^g}. Recall from \dref{strict-sub} that for $|a|\leq a_0/2$, the shifts $g+a$ are also strict supersolutions with parameters $a_0$ replaced by $a_0/2$ and the same $\delta_0$.

For $0<a<a_0/2$, by \tref{parabolic-existence}, there exists $\mathfrak{u}^a(x,t)$ an inner variational and relaxed viscosity solution of the parabolic Bernoulli problem \eqref{e.bernoulli-parabolic} that is monotone decreasing in $t$ and such that $\mathfrak{u}^a(x,t) = (g(x)+a)_+$ on $\partial_P( U \times (0,\infty))$; in particular, $\mathfrak{u}^a$ is a viscosity supersolution. As $v$ is \emph{below} $g_+$, we have
$v \leq g_+ \prec (g+a)_+$ in $\overline U$. By continuity of $\mathfrak{u}^a$ on the parabolic boundary (\tref{parabolic-existence}\ref{part.continuity}), it follows that
$v \prec \mathfrak{u}^a$ in a neighborhood of $\partial_P U_\infty$.
Applying monotonicity, the comparison principle Theorem \ref{t.comparison}, and Remark \ref{r.statToPara},
we obtain
\[v\prec \mathfrak{u}^a \leq (g+a)_+ \quad \text{ in } U_\infty.\]
Letting $t\to \infty$ gives
\[v\leq u^a_\infty : = \lim_{t\to \infty}\mathfrak{u}^a(\cdot,t) \leq (g+a)_+ \quad \text{ in } U .\]
By Theorem \ref{t.longtime-inner} and \pref{supersolution-limit}, $u^a_{\infty}$ is  an inner variational and viscosity \emph{super}solution of \eqref{e.bernoulli-basic} that satisfies the Lipschitz estimate \eqref{eqn:inftyLip} and perimeter bound for its associated `positivity set' $\chi^a_{\infty}$ \eqref{eqn:inftyPerim}; both estimates are uniform in $a$. Thus, by Lemma \ref{l.super-stability}\ref{part.stability-super} and Lemma \ref{l.inner-compactness-stability}, there is a subsequence $a_k \downarrow 0$ such that the limit $\check u := \lim_{a_k\downarrow 0} u^{a_k}_{\infty}$ exists and $\check{u}$ is an inner variational solution and viscosity supersolution of \eqref{e.bernoulli-basic}.
Finally, we apply \lref{inner-variational-implies-subsoln} to conclude that $\check u$ is also a viscosity subsolution. The modulus $\varpi$ of \eqref{e.attainment} is uniform in $a$, so $\check u \in C(\overline U)$ and hence $\check u \in \mathcal{S}^g$. Taking the supremum over $v\in \mathcal{S}^g$, see \eqref{e.perron-def-sub}, we find $\check u \geq u$. Since $\check u \in \mathcal{S}^g$, we conclude that $\check u = u$ and obtain the theorem.
\end{proof}

\subsection{Extremal solutions are directionally minimal}\label{s.downward}

In this section we prove part \ref{part.downward} of \tref{main}. The argument is short and independent of the rest of the paper. The idea is to minimize $J_Q$ with $u$ itself as the obstacle, and invoke the extremality of $u$ to conclude that the minimizer is $u$.

\begin{proposition}\label{p.downward}
  If $u$ is a smallest supersolution (resp.\ largest subsolution) of \eqref{e.bernoulli-basic} in $U$ in the sense of \dref{largest-smallest} then $u$ is a downward (resp.\ upward) minimizer of $J_Q(\cdot ;U)$.
\end{proposition}
\begin{proof}
First note that $u$ is a viscosity solution of \eqref{e.bernoulli-basic} and is locally Lipschitz in $U$, see \cite{FeldmanKim}*{Lemma 2.12 and Lemma 2.17}; in particular, $u \in H^1_{\rm loc}(U) \cap C(U)$.

    \emph{Step 1: smallest supersolution case.} Condition (i) of \dref{directional} holds as $u$ is a viscosity solution by Proposition \ref{p.inner-smallest}, so it remains to prove the energy inequality \eqref{eqn:directionalMinimality} for downward competitors. We minimize the energy \eqref{e.energy} with $u$ as an obstacle from above. Let $B \subset \subset U$ be a ball and let $w$ be a minimizer of
   \begin{equation}\label{e.obstacle-remin}
       \min \{ {J_Q}(w;B): \ w \in u+H^1_0(B), \ 0 \leq w \leq u\},
   \end{equation}
   which exists by the direct method (the class is nonempty as $u$ belongs to it and $J_Q$ is lower semi-continuous). Further $w$ is continuous in $\overline{B}$ by the regularity theory for quasi-minimizers of the Dirichlet energy see \cite{GiustiBook}*{Theorem 7.8}, and also \cite{GiustiBook}*{Example 6.4} for how a Lipschitz obstacle such as $u$ can be handled in the context of the quasi-minimizer theory. We claim $   w\equiv u$, which by definition of $w$, implies the desired downward minimality for $u$. To obtain the claim, we show $w$ is a viscosity supersolution of \eqref{e.bernoulli-basic} in $U$; since $u$ is the smallest supersolution, $u \leq w$, but since $w\leq u$ by assumption, it follows that $w = u$.

   To verify $w$ (extended by $u$ outside of $B$) is a viscosity supersolution, suppose $\varphi \in C^\infty(U)$ touches $w$ strictly from below at $x_0 \in U$ (recall Remark~\ref{r.strict-touching}). If $w(x_0) = u(x_0)$ then, since $w \leq u$, $\varphi$ also touches $u$ from below at $x_0$, and the supersolution property of $u$ yields the desired outcome. Otherwise $w(x_0) < u(x_0)$; since $u \equiv w$ outside of $\overline{B}$, the touching point $x_0 \in B$ and, by continuity of the functions, $w < u$ in a neighborhood $B_r(x_0) \subset B$. Suppose for contradiction that the supersolution condition fails at $x_0$. Then for small $\delta > 0$ the upward perturbation $w \vee (\varphi + \delta)_+$ strictly decreases the energy, by the computation of \cite{FeldmanKimPozar}*{proof of Lemma 3.3}: failure of both alternatives means $\Delta\varphi(x_0) > 0$, and if in addition $\varphi(x_0) = 0$, then $|\grad\varphi(x_0)| > Q(x_0)$. Thus the energy difference formula \cite{FeldmanKimPozar}*{Lemma A.1}  shows $J(w \vee (\varphi + \delta)_+; B) < J(w;B)$ for $\delta$ small. Furthermore, for $\delta$ small, the perturbation is supported in $B_r(x_0)$ and satisfies the constraint $0 \leq w \vee (\varphi + \delta)_+ \leq u$, so it is admissible in \eref{obstacle-remin}. This contradicts the minimality of $w$ and proves the claim.

    \emph{Step 2: largest subsolution case.} Parallel to Step 1, condition (i) of \dref{directional} holds as $u$ is a viscosity solution by Proposition \ref{p.inner-largest}. To obtain directional minimality \eqref{eqn:directionalMinimality}, we minimize the energy with $u$ as an obstacle from below. Let $B \subset \subset U$ and let $w$ be a minimizer of
   \begin{equation}\label{e.obstacle-remin-sub}
       \min \{ J_Q(w;B): \ w \in u+H^1_0(B), \ w \geq u\},
   \end{equation}
   which exists and is continuous as before. As above it suffices to show $w \equiv u$ and this follows from showing $w$ is a viscosity subsolution of \eqref{e.bernoulli-basic} in $U$.

   To verify $w$ is a viscosity subsolution, suppose that $\varphi \in C^\infty(U)$ touches $w$ strictly from above in $\overline{\{w>0\}} \cap U$ at $x_0$. If $x_0 \in \overline{\{u>0\}}$ and $w(x_0) = u(x_0)$ then, since $w \geq u$, $\varphi_+$ also touches $u$ from above at $x_0 \in \overline{\{u>0\}}$, and the subsolution property of $u$ allows us to conclude. Otherwise either 
   $$\text{(i)} \qquad w(x_0) > u(x_0) \geq 0,$$ 
   or 
   $$\text{(ii)} \qquad x_0 \not\in \overline{\{u>0\}}.$$ 
   Since $u \equiv w$ outside of ${B}$, we have 
   $$\text{(i)} \implies x_0 \in B$$ 
   and 
   $$\text{(ii)} \implies x_0 \in \overline{B}.$$ 
   Suppose, for contradiction, the subsolution condition fails at $x_0$. Applying \cite{FeldmanKimPozar}*{Lemma 3.3} once again, the downward perturbation $w \wedge (\varphi - \delta)_+$ strictly decreases the energy for small $\delta>0$. To obtain a contradiction with \eqref{e.obstacle-remin-sub}, the perturbation must be an admissible competitor. Precisely, it must remain above the obstacle $u$ and only modify $w$ inside $B$, i.e., $w \wedge (\varphi-\delta)_+ = u$ outside of $B$. 
   
   First, we show that the perturbation is above the obstacle. In case (i), $u < w \leq \varphi_+$ near $x_0$ and so by continuity $u\leq w \wedge (\varphi-\delta)_+$ near $x_0$ for small $\delta$. In case (ii), $u=0$ in a neighborhood of $x_0$ so $w \wedge (\varphi-\delta)_+ \geq 0 =u$ near $x_0$ as well.
   
   Next, we show that the perturbation is supported in $B$. Since the touching is strict $\{(\varphi-\delta)_+ < w\} \subset B_{o(1)}(x_0)$ as $\delta \to 0$. So if $x_0 \in B$ (whether in case (i) or (ii)) just choose $\delta>0$ sufficiently small so that $\{(\varphi-\delta)_+ < w\} \subset B$. If $x_0 \in \partial B$, then $w(x_0) = u(x_0)$ and we must also be in case (ii) $x_0 \not\in \overline{\{u>0\}}$. Consequently, $u \equiv 0 $ in a small ball $B_r(x_0)$ and $w(x_0) = u(x_0) = 0$. Since $w = u$ outside $B$, we get $w \equiv 0$ in $B_r(x_0) \setminus B$. Choosing $\delta>0$ small enough so that $\{(\varphi - \delta)_+ < w\} \subset B_r(x_0)$ (possible again by the strict touching), we conclude $w \wedge (\varphi - \delta)_+ = 0 = w = u$ in $B_r(x_0) \setminus B$, and therefore the perturbation is supported in $B$. Thus it is a valid competitor for \eref{obstacle-remin-sub}, providing a contradiction. 
\end{proof}

Finally, for completeness, we include the proof of \lref{local-minimality}, as it was used in the proof of the previous result.

\begin{proof}[Proof of \lref{local-minimality}]
   We do the smallest supersolution case first. Since $u$ is a viscosity supersolution it cannot touch $g$ from above, that would contradict the definition of viscosity solution. So $u>g$ in $U$.  

If $x \in  \partial \{u>0\} \cap U$ then let $B_r(x)$ be a sufficiently small ball centered at $x$ so that $B_r(x) \subset\subset U$ and $g \leq 0$ on $B_r(x)$.  Such a ball exists by the previous paragraph. Now let $B \subset \subset B_r(x)$ and $v$ be a supersolution in $B_r(x)$ with $v = u$ on $B_r(x) \setminus B$. Extend $v = u$ in $U \setminus B_r(x)$, this is a supersolution in $U$ since it is equal to $u$ in a neighborhood of $\partial B_r(x)$. Then $v = u \geq g$ in $U \setminus B$ and $v \geq 0 \geq g$ in $B$ so $v \in \mathcal{S}_g$ and so $v \geq u$ by definition of the smallest supersolution.

Next let $u$ be a largest subsolution below $g$. Let $x \in \partial \{u>0\}$. Similar to before $g$ is a smooth strict supersolution so it cannot touch $u$ from above at $x$, i.e. $g(x)>0$. Choose a closed ball $\overline {B_r(x)}\subset U$ on which $g>0$, and hence $g$ is also strictly superharmonic on $B_r(x)$. Let $B \subset \subset B_r(x)$ and $v$ be a viscosity subsolution of \eref{bernoulli-basic} in $B_r(x)$ with $v = u$ on $B_r(x) \setminus B$. Extend $v = u$ in $U \setminus B_r(x)$; as before this is a subsolution in $U$. Then $v$ is subharmonic in $B$ and $v=u \leq g$ on $\partial B$. So by maximum principle $v \leq g$ in $B$. Thus $v \in \mathcal{S}^g$ and so $v \leq u$ by definition of largest subsolution.
\end{proof}

\section*{Acknowledgments}
F.A. acknowledges support from the NSF grant DMS-2246611. W.F. was supported by the NSF grant DMS-2407235 and appreciated conversations with Dennis Kriventsov, Bozhidar Velichkov, and Sebastian Munoz. K.S. was supported by funding from the
NSF RTG grant DMS-2136198. Although the main ideas in this work originated prior to the AIM workshop ``Time-dependent Bernoulli free boundary problems", all three authors benefited greatly from the conversations and exchange of ideas that took place during the workshop. We thank the AIM staff for their support and hospitality.

\section*{AI use declaration}
We detail the specific uses of AI (Claude Fable 5 and Opus 5.5) in this paper. 
\begin{itemize}
\item The idea for the example in \fref{necklace} and the code to generate the two figures. 
\item An outline of Appendix~\ref{s.semilinear} was drafted by the authors, but the details of some of the key technical steps were initially filled in by AI agents. The appendix was then thoroughly revised by the authors for clarity and brevity. 
\item An AI suggestion generalized and clarified our proof of Lemma~\ref{l.directionalStable} by introducing the lattice property idea.
\item AI agents were used to formally verify the results of the paper in Lean 4, see \cite{GitHubRepo}. This revealed several issues which were fixed by the authors. The current version of the paper closely corresponds to the Lean formalization; any deviations are listed at \cite{GitHubRepo}.
\item General minor edits throughout the paper for clarity, correctness and style.
\end{itemize}
All other aspects of this work (most notably, the main proof strategy) were conceived by the authors, who take responsibility for the entire manuscript.

\section*{Formal verification in Lean 4}
This paper is accompanied by an autoformalization using Lean 4 \cite{GitHubRepo}.  The formalization was carried out by AI agents, principally Claude Opus 5.5, under the close direction of the second author. There are some minor differences between the formal verification and the written manuscript; these deviations were collected over multiple rounds of AI agent searches with adversarial checks. After a careful review, we determined that these deviations do not critically impact the mathematical accuracy of the paper and opted to leave them unchanged for the purpose of clarity and brevity. The full list of deviations can be found at \cite{GitHubRepo}.  The correspondence between the verified headline results and the headline results of the paper has also been reviewed and confirmed by the second author.

\appendix

\section{Additional proofs}\label{s.additional-proofs}

In this appendix, we include some proofs that were skipped in the main body of the paper.

\subsection{Smallest supersolution non-degeneracy}\label{s.smallest-super-nondegen} We present a proof of non-degeneracy of smallest supersolutions \tref{smallest-super-non-degen}. This is relatively standard and several closely related proofs appear in the literature, see \cite{CaffarelliSalsa}*{Lemma 6.9} and \cite{VelichkovBook}*{Lemma 4.4}, but none apply precisely in our setting.

\begin{proof}[Proof of \tref{smallest-super-non-degen}]
By rescaling it suffices to consider the case $r=1$ and assume that $u$ is a (local) smallest supersolution in a slightly larger ball $B_{1+\ep}$. Let $h$ be the harmonic function in $B_1 \setminus B_{1/2}$ which is equal to $1$ on $\partial B_1$ and $0$ on $\partial B_{1/2}$.  Extend $h$ to be $0$ in $B_{1/2}$. For $c_0(d,Q_{\min})$ sufficiently small $c_0h$ is a supersolution of \eqref{e.bernoulli-basic} in $B_1$.  If $\sup_{\partial B_1} u \leq c_0/2$ then 
\[v=\min\{u,c_0h\} \ \hbox{ in } B_1 \ \hbox{ and } \ v = u \ \hbox{ outside of } B_1\]
 is a continuous supersolution of \eqref{e.bernoulli-basic}, since a minimum of two viscosity supersolutions is a viscosity supersolution, and $v = 0$ in a neighborhood of $0$. Since $0 \in \overline{\{u>0\}}$, we conclude that $v$ is strictly smaller than $u$ somewhere, contradicting the smallest supersolution property of $u$.
\end{proof}

\subsection{Largest subsolution non-degeneracy}\label{s.largest-sub-nondegen} Next we present a proof of non-degeneracy of largest subsolutions in two dimensions. This result was originally proved by Orcan-Ekmekci \cite{OrcanEkmekci}, quoted here as \tref{largest-sub-non-degen} above. Since that proof is brief, elegant, and perhaps not so well known, we recall it here.

\begin{lemma}[Subsolution non-degeneracy at outer-regular points]\label{l.outer-reg-non-degen}
    If $u \in C(\overline{B_1(0)})$ is a subsolution of \eqref{e.bernoulli-basic} in $B_1(0)$ and $\{u>0\}$ has an exterior touching ball of radius $1/4$ at $0 \in \partial \{u>0\}$ then
    \[ \max_{\partial B_1(0)} u \geq c.\]
\end{lemma}
\begin{proof} The proof is a straightforward barrier argument with the definition of subsolution. Let $B_{1/4}(z)$ be the exterior tangent ball to $\{u>0\}$ at $x=0$. Define a radially symmetric strictly superharmonic barrier $h$ in the annulus $B_{1/2}(z) \setminus B_{1/4}(z)$ that is zero on $\partial B_{1/4}(z)$ and equal to $\max_{\partial B_1(0)} u = \max_{\overline{B_1(0)}} u$ on $\partial B_{1/2}(z)$. If $ \max_{\partial B_1(0)} u$ is too small the subsolution property of $u$ will be violated. \end{proof}

\begin{lemma}\label{l.nondegen-equivalences}
    Suppose $u$ is a non-negative $L$-Lipschitz function on $B_1(0)$ for some $L \geq 1$ which is harmonic in $\{u>0\}$ with $0  \in \partial \{u>0\}$ and satisfies for some constants $c_i>0, \, i=1,2,3,$ one of the following conditions:
    \begin{enumerate}[label = (\roman*)]
        \item\label{part.avg} There exists $c_1>0$ so that 
        \[\dashint_{\partial B_r(0)} u(x)~dS(x) \geq c_1 r \ \hbox{ for } \ 0 < r < 1.\]
        \item\label{part.sup} There exists $c_2>0$ so that 
        \[ \sup_{\partial B_r(0)} u \geq c_2 r \ \hbox{ for } \ 0 < r < 1.\]
        \item\label{part.laplace} There exists $c_3>0$ so that 
        \[\int_{B_r(0)} \Delta u ~dx \geq c_3 r^{d-1} \ \hbox{ for } \ 0 < r < 1.\]
    \end{enumerate}
    Then the other two conditions $ j \in \{\ref{part.avg},\ref{part.sup},\ref{part.laplace}\}$ for $j \neq i$ also hold with constants $c_j$ depending only on $c_i$, $d$ and $L$.
\end{lemma}
\begin{proof}
Note that \ref{part.avg} implies \ref{part.sup} immediately with $c_2 = c_1$. The other direction \ref{part.sup} implies \ref{part.avg} follows from the Lipschitz estimate and non-negativity of $u$. So it suffices to show \ref{part.avg} is equivalent to \ref{part.laplace}.

Note that
\[ \frac{d}{dr}\dashint_{\partial B_r(0)} u(x)~dS(x) = \frac{1}{|\partial B_1| r^{d-1}} \int_{B_r(0)} \Delta u ~dx\]
so
\[0 = u(0) = \dashint_{\partial B_r(0)} u(x)~dS(x) -\frac{1}{|\partial B_1|}\int_0^r\frac{1}{s^{d-1}} \int_{B_s(0)} \Delta u ~dx~ds.\]
It follows that \ref{part.laplace} implies \ref{part.avg}. 

Finally to show that \ref{part.avg} implies \ref{part.laplace} use that
\begin{align*}
    \frac{1}{|\partial B_1|}\int_{\kappa r}^r\frac{1}{s^{d-1}} \int_{B_s(0)} \Delta u ~dx~ds &= \dashint_{\partial B_r(0)} u(x)~dS(x) - \dashint_{\partial B_{\kappa r}(0)} u(x)~dS(x) \\
    &\geq (c-L\kappa) r \geq \frac{1}{2}cr
\end{align*}
choosing $\kappa = \frac{1}{2L}c$.  On the other hand, since $\Delta u$ is a non-negative measure,
 \[\frac{1}{|\partial B_1|}\int_{\kappa r}^r\frac{1}{s^{d-1}} \int_{B_s(0)} \Delta u ~dx~ds \leq \frac{1}{|\partial B_1|\kappa^{d-1}}\frac{1}{r^{d-2}}\int_{B_r(0)} \Delta u ~dx.\]
Combining the with the previous lower bound shows \ref{part.avg} implies \ref{part.laplace}.
\end{proof}
\begin{lemma}\label{l.largest-sub-sphere-exterior}
    If $u$ is a largest subsolution in $B_1$ and $0 \in \partial \{u>0\}$ then for any $0 < r < 1$, $ \partial B_r(0) \setminus \overline{\{u>0\}} \neq \emptyset$.
\end{lemma}
\begin{proof}[Sketch]
    Suppose, towards a contradiction, that $\partial B_r(0) \subset \overline{\{u>0\}}$. Let $v$ be the harmonic lift of $u$ in $B_r(0)$ which is equal to $u$ outside of $B_r(0)$.  Since $u$ is harmonic in $\{u>0\} \cap B_1$ and non-negative $u$ is subharmonic in $B_1(0)$. Thus $u \leq v $ in $B_1(0)$.  Since $0 \in \partial \{u>0\} \subset \{u=0\}$ we must have $v > u$ in $B_r(0)$ by strong maximum principle. We claim that $v$ is a subsolution of \eref{bernoulli-basic} in $\{v>u\} \cup B_1(0)\setminus \overline{\{u>0\}}$. Note that by the hypotheses $\partial B_r(0)$ does not intersect this set so nothing needs to be checked there.  Thus the largest subsolution property of $u$ implies that $v \leq u$ which is a contradiction.
\end{proof}

\begin{proof}[Sketch of proof of \tref{largest-sub-non-degen}]
    By Lemma~\ref{l.nondegen-equivalences} it suffices to show that $\int_{B_r(0)} \Delta u \geq cr$ for all $0 < r < 1$. By hyperbolic scaling invariance of the largest subsolution property it suffices to show $\int_{B_1(0)} \Delta u \geq c$.
    
    For all $0 < r < 1$ there is $x_r \in \partial B_r(0) \setminus \overline{\{u>0\}}$ and let $\rho_r = d(x_r,\overline{\{u>0\}})$. Then $B_{\rho_r}(x_r)$ touches $\{u>0\}$ from the outside at some point $y_r$ so \lref{outer-reg-non-degen} (rescaled and applied at all radii $0 \leq \rho \leq 4\rho_r$) and Lemma~\ref{l.nondegen-equivalences} imply that 
    \[\int_{B_{2\rho_r}(x_r)}  \Delta u \geq c\rho_r. \]
    Since $\{(r-2\rho_r,r+2\rho_r)\}_{r \in (0,1]}$ form an open cover of $(0,1]$ there is a disjoint subcollection $\{(r_j-2\rho_j,r_j+2\rho_j)\}_{j=1}^\infty$ such that $(r_j - 10\rho_j,r_j+10\rho_j)$ covers $(0,1]$. Note that the disjointness of the intervals of radii implies disjointness of the balls $B_{2\rho_j}(x_{r_j})$. Thus
    \[ \sum 20\rho_j \geq 1 \ \hbox{ and so } \ \int_{B_1(0)} \Delta u \geq \sum_j \int_{B_{2\rho_j}(x_{r_j})} \Delta u \geq c \sum \rho_j \geq c.\]
    This concludes the proof. \end{proof}

\subsection{Proof of Lemma~\ref{l.directionalStable}}\label{s.directionalStable-proof}
For the proof we will make use of the following lattice property of the energy $J_Q(\cdot;U)$.  For any $u,v \in H^1(U) \cap C(U)$
\begin{equation}\label{e.J-lattice-prop}
J_Q(\max\{u,v\};U)+J_Q(\min\{u,v\};U) = J_Q(u;U)+J_Q(v;U).
\end{equation}
This just follows from the corresponding lattice properties of the Dirichlet energy and the super-level indicators ${\bf 1}_{\{u>0\}}$.

We just do the case $Q_n = Q$ for all $n$, the generalization to a uniformly convergent sequence of $Q_n$ is not difficult (in fact pointwise convergence, the uniform bounds $0< Q_{\min} \leq Q_n \leq Q_{\max}$, and an argument with dominated convergence and Fatou's lemma will suffice). It is clear from the locally uniform convergence that $u$ is harmonic in its positivity set. We must verify that \eqref{eqn:directionalMinimality} holds for $u$.

The upward and downward minimality cases are very similar so we will just argue for upward minimizers.

We will show directly that $u$ is an upward minimizer of $J_Q(\cdot;U)$. Let $B \subset\subset U$ and $v \in H^1_{\rm loc}(U)$ be an admissible competitor in $B$ i.e. $v \geq u$ and $u = v$ in $U \setminus B$. Call $L := L_{\overline{B}}$ from \eqref{eqn:locallocalLip}. Without loss of generality we can assume that $B = B_1(0)$. 

Let $r< 1$ and $B_r : = B_r(0)$. We take $\varphi\in C^\infty(B;[0,1])$ to be a standard cutoff with $\varphi = 1$ in  $B_{r}$, $\varphi  = 0$ in $B_1^c$, and $|\grad \varphi| \leq C/(1-r)$.  Define 
\[
v_n := \varphi \max\{v,u_n\} + (1-\varphi) u_n.
\]
Then $v_n \geq u_n$ and $u_n-v_n \in H^1_0(B)$ so that $v_n$ is an admissible upward competitor for $u_n$. Since $u_n$ is an upward minimizer
\begin{equation}
J_Q(u_n;B_1) \leq J_Q(v_n;B_1).
\end{equation}
Now split both integrals into the part on $B_r$ and the part on $B_1 \setminus B_r$ and use that $v_n = \max\{v,u_n\}$ on $B_r$ to get
\begin{equation}\label{e.lattice-application}
J_Q(u_n;B_r)+J_Q(u_n;B_1\setminus B_r)  \leq J_Q(\max\{v,u_n\};B_r)+J_Q(v_n;B_1 \setminus B_r)
\end{equation}
Next, apply the lattice property \eqref{e.J-lattice-prop}
\[J_Q(\max\{v,u_n\};B_r) = J_Q(v;B_r)+J_Q(u_n;B_r)-J_Q(\min\{v,u_n\};B_r)\]
to rewrite the first term on the right-hand side of \eqref{e.lattice-application}. Canceling the $J_Q(u_n;B_r)$ terms that now appear on both sides of the inequality, dropping the term $J_Q(u_n;B_1\setminus B_r) \geq 0$ from the left-hand side, and rearranging, we obtain
\begin{equation}\label{e.better-u-upper}
J_Q(\min\{v,u_n\};B_r) \leq J_Q(v;B_r)+J_Q(v_n;B_1 \setminus B_r).
\end{equation}

We first address the limit of the left-hand side of \eqref{e.better-u-upper}. Note that $\min\{v,u_n\} \to u$ uniformly and, since the gradient $|\grad \min\{v,u_n\}|^2 \leq L^2+|\grad v|^2$ is uniformly bounded in $L^2$, the gradients converge weakly in $L^2$ to $\grad u$.   Also note the lower semi-continuity ${\bf 1}_{\{u>0\}} \leq \liminf {\bf 1}_{\{\min\{v,u_n\}>0\}}$ which follows from the pointwise convergence $\min\{v,u_n\} \to u$. So using weak lower-semicontinuity of the Dirichlet energy and Fatou's lemma for the indicators of the positivity sets, we find
\[ J_Q(u;B_r) \leq \liminf_{n \to \infty} J_Q(\min\{v,u_n\};B_r).\]
Applying this limit in \eqref{e.better-u-upper} we arrive at
\begin{equation}\label{eqn:coincide}
J_Q(u,B_{r}) \leq J_Q(v,B_r)+\limsup_{n \to \infty}J_Q(v_n;B_1 \setminus B_r).
\end{equation}
We can then estimate the error term on the annulus by
\[
J_Q(v_n;B_1 \setminus B_r) \leq\frac{C}{(1-r)^2}\int_{B_1 \setminus B_{r}}(u_n-v)^2\, dx +C\int_{B_1 \setminus B_r} \left(|\grad u_n|^2+|\grad v|^2 \right) dx +C(1-r).
\]
Taking the limit $n \to \infty$ on both sides, and using \eqref{eqn:locallocalLip} to estimate $|\grad u_n|^2 \leq L^2$ pointwise and uniformly in $n$ yields
\[
\limsup_{n \to \infty}J_Q(v_n;B_1 \setminus B_r) \leq \frac{C}{(1-r)^2}\int_{B_1 \setminus B_{r}}(u-v)^2\, dx +C\int_{B_1 \setminus B_r} \left(L^2+|\grad v|^2 \right) dx +C(1-r).
\]
We estimate the first term on the right-hand side:
\[\frac{1}{(1-r)^2}\int_{B_1 \setminus B_{r}}(u-v)^2 \, dx \leq C\int_{B_1 \setminus B_{r}} |\grad (u-v)|^2\, dx\]
since $(u-v)$ has zero trace on $\partial B_1$ (this is the Poincare inequality on the annular domain $B_1\setminus B_r$ with zero outer-trace). Now all four annular error terms on the right vanish as $ r \nearrow 1$ by monotone convergence theorem. 

Combining all this with \eqref{eqn:coincide} we obtain
\[\lim_{r\nearrow 1}J_Q(u;B_r) \leq \lim_{r \nearrow 1}J_Q(v;B_r) +0\]
and a final monotone convergence gives
\[J_Q(u;B_1) \leq J_Q(v;B_1)\]
which is the desired upward minimality property. \qed

\subsection{Proof of \pref{two-plane-minimizing-props}}\label{s.proof-two-plane-minimizing-props} 

We just do the $q=1$ case and denote $J_1 = J$.

\textit{(i)} We directly check that such functions are not upward minimal. Since $\phi_{\alpha}(x_d)$ is subharmonic, but not harmonic, in $B_1(0)$, harmonic replacement in $B_1(0)$ is an increasing perturbation which strictly decreases the Dirichlet energy without changing the closure of the positivity set.  Thus, $\phi_\alpha$ is not upward minimal.

\textit{(ii)} Next we check that $\phi_\alpha(x_d)$ is not downward minimizing for $\alpha < 1$. First we check in dimension $d=1$ and extend the idea to higher dimensions. 

\emph{Step 1: Not downward minimizing for $d = 1$ and $\alpha <1$.} For $\ep \in (0,1)$ to be determined, consider the perturbation
\[
\phi_\alpha^\ep(s):=\frac{\alpha}{1-\ep}(s-\ep)_++\frac{\alpha}{1-\ep}(-s-\ep)_{+} = \frac{\alpha}{1-\ep}(|s|-\ep)_+.
\]
Then $0\leq \phi_\alpha^\ep\leq \phi_\alpha$, and $\phi_\alpha^\ep = \phi_\alpha$ at
$s = -1,1$, so $\phi_\alpha^\ep$ is an admissible downward competitor for $\phi_\alpha$
in $[-1,1]$. Since $\phi_\alpha^\ep$ has slope $\pm\alpha/(1-\ep)$ on
$\{\phi_\alpha^\ep>0\}=[-1,1]\setminus[-\ep,\ep]$, which has length $2(1-\ep)$, we find that
\[J(\phi_\alpha^\ep;[-1,1])=2\Big[\frac{\alpha^2}{1-\ep}+(1-\ep)\Big]
=J(\phi_\alpha;[-1,1])-\frac{2\ep}{1-\ep}\big(1-\ep-\alpha^2\big),\]
which is strictly smaller than $J(\phi_\alpha;[-1,1])=2(\alpha^2+1)$ whenever
$0<\ep<1-\alpha^2$.

\emph{Step 2: Not downward minimizing for $d > 1$ and $\alpha <1$.}
To extend this idea to higher dimensions, we must construct a local competitor. Let $L \gg 1$ and let $\eta\in C^\infty(\R)$ be a decreasing cut-off function with $\eta = 1$ in $B_{1/2}$ and $\eta = 0$ outside $B_1$.
Define
\[\psi^\ep(x):=\phi_\alpha^{{[\ep\eta(|x'|/L)]}}(x \cdot e_d) \ \hbox{ in } \ B_L'(0) \times [-1,1], \ \hbox{ and } \ \psi^\ep = \phi_\alpha \hbox{ elsewhere.}\]
Then $0\leq \psi^\ep\leq\phi_\alpha$
and $\psi^\ep = \phi_\alpha$ on $\partial (B_L'(0) \times [-1,1])$, so $\psi^\ep$ is an
admissible downward competitor in the cylinder. By the one-dimensional computation, the
vertical energy of the column over $x'$ is exactly
\[\int_{-1}^{1}|\partial_d\psi^\ep|^2+{\bf 1}_{\{\psi^\ep>0\}} dx_d
=2(\alpha^2+1)+2\ep\eta(|x'|/L)\Big[\frac{\alpha^2}{1-\ep\eta(|x'|/L)}-1\Big].\]
For the horizontal part, on $\{|x_d|>\ep\eta(|x'|/L)\}$, where $\psi^\ep$ is positive, a direct calculation shows
\[\nabla'\psi^\ep=\frac{\alpha\ep(|x_d|-1)}{(1-\ep\eta(|x'|/L))^2} \nabla'\left[\eta(|x'|/L)\right]
\quad\hbox{so}\quad |\nabla'\psi^\ep|\leq 4\|\eta'\|_\infty\alpha\ep/L\]
for $\ep\leq1/2$, while $\nabla'\psi^\ep=0$ on $\{|x_d|<\ep\eta(|x'|/L)\}$. Integrating in
$x'$ and rescaling gives
\begin{align*}
J(\psi^\ep;B_L'(0) \times [-1,1]) &- J(\phi_\alpha;B_L'(0) \times [-1,1])\\
&= 2L^{d-1} \left[\ep\int_{B_1'}\eta(|y|)\Big[\frac{\alpha^2}{1-\ep\eta(|y|)}-1\Big] dy 
+O(\alpha^2\ep^2 L^{-2})\right].
\end{align*}
For $0<\ep<1-\alpha^2$ the integrand is nonpositive, and on $B_{1/2}'$, where
$\eta\equiv1$, it equals $-\tfrac{1-\ep-\alpha^2}{1-\ep}<0$, so the coefficient of
$L^{d-1}$ is at most $-\tfrac{2\ep(1-\ep-\alpha^2)}{1-\ep} |B_{1/2}'|<0$. Taking $L$
large the $O(\alpha^2\ep^2 L^{-2})$ error from the cut-off is dominated, and $\psi^\ep$ has strictly
smaller energy than $\phi_\alpha$ in $B_L'(0)\times[-1,1]$. Thus $\phi_\alpha$ is not
downward minimizing for any $\alpha<1$.

\emph{Step 3: Downward minimizing for $\alpha \geq 1$.} 
Finally, we show that $\phi_\alpha(x_d)$ is downward minimizing for $\alpha\geq1.$ It suffices to show it is downward minimizing for $d=1$, since any competitor for $d>1$ will pay the $1$-dimensional energy along slices in the $x_d$ direction plus energy due to the derivative in the $x'$ direction.
We consider a competitor $\phi$ with $0\leq\phi\leq \phi_\alpha$ and $\phi = \phi_\alpha$ on $[-1,1]^c$. Note that $\phi(0) = 0$ so we may define $s_0 : = \sup \{s : \phi(s) = 0\}.$ The energy of $\phi$ decreases if we replace $\phi$ by $\one_{[0,s_0]^c}\phi$. By convexity of the Dirichlet energy, we have that
$$J(\phi;[0,1]) \geq \frac{\alpha^2}{1-s_0} + (1-s_0).$$
Thus downward minimality of $\phi_\alpha$ on $[0,1]$ is a consequence of the inequality 
\[\frac{\alpha^2}{x} + x  \geq \alpha^2 + 1 \quad \text{ for } x\in (0,1],\]
which easily follows from looking at the roots of the quadratic found by multiplying both sides by $x$.
As the same argument holds on $[-1,0]$ and any downward competitor on $[-1,1]$ can be decomposed into downward competitors for $[-1,0]$ and $[0,1]$ (because $\phi_\alpha(0) = 0$), we conclude that $\phi_\alpha$ is downward minimizing on $[-1,1]$. The argument is similar for general intervals. \qed

\subsection{Proof of Lemma~\ref{lem:chiConverge}}\label{s.proof-chiConverge}
Let $\mathcal{U} \subset\subset U_\infty $ be a convex set.
We note that it suffices to prove that the energy density $e_\eps : = |\nabla \mathfrak{u}_\eps|^2 + Q(x)^2\chi_\eps$ is $L^1(\mathcal{U})$ precompact: by Lemma \ref{lem:strongGrad}, $|\nabla \mathfrak{u}_\eps|^2 \to |\nabla \mathfrak{u}|^2$ in $L^1(\mathcal{U})$ along the subsequence, and $\chi_\eps = Q^{-2}(e_\eps - |\nabla \mathfrak{u}_\eps|^2)$ with $Q^2 \in C^{0,1}$ and $Q \geq Q_{\min} > 0$, so $L^1$-precompactness of $e_\eps$ yields that of $\chi_\eps$. 
Introducing a (spatial) mollifier $\phi_\delta$ supported in $B(0,\delta)\subset \R^d$ and taking $\delta <\delta_0 := {\rm dist}(\mathcal{U},\partial U_\infty)/2$, one computes for $(x,t)\in \mathcal{U}$ that
\begin{equation}\nonumber
\begin{aligned}
& |\partial_t (e_\eps*\phi_\delta) (x,t)| \\
& = \left|-\int_{U} 2 (\partial_t \mathfrak{u}_\eps)^2(z,t) \phi_\delta(x-z)  dz + \int_{U}2\partial_t \mathfrak{u}_\eps \nabla \mathfrak{u}_\eps(z,t)\cdot \nabla \phi_\delta(x-z)   dz\right| \\
& \leq C \left( \|\phi_\delta\|_{\infty} + \|\nabla \phi_\delta\|_{\infty} \right) \left(\|\partial_t \mathfrak{u}_\eps(\cdot, t)\|_{L^2(U)}^2 + \|\nabla \mathfrak{u}_\eps(\cdot, t)\|_{L^2(U)}^2 \right).
\end{aligned}
\end{equation} 
where we used $\partial_t e_\eps = 2\grad \mathfrak{u}_\eps\cdot\grad\partial_t \mathfrak{u}_\eps + 2Q^2\beta_\eps(\mathfrak{u}_\eps)\partial_t \mathfrak{u}_\eps$, integrated by parts in $z$, and substituted $\Delta \mathfrak{u}_\eps = \partial_t \mathfrak{u}_\eps + Q^2\beta_\eps(\mathfrak{u}_\eps)$; the terms involving $\beta_\eps$ cancel exactly because the weight $Q^2$ in $e_\eps$ matches the equation. From this and the dissipation equality \eqref{eqn:dissipation} we find
\begin{equation}\label{eqn:timeDerivative}
\int_{\mathcal{U}} |\partial_t (e_\eps*\phi_\delta)|   dx  dt \leq C_\delta.
\end{equation}

Using the bound \eqref{eqn:chiEpsBV} for some $\eta  = 1$ on $\{{\rm dist}(x,\mathcal{U}) < \delta_0\}$ and the convexity of $\mathcal{U}$ (which allows for the fundamental theorem of calculus along lines), one can also show that
\begin{equation}\label{eqn:spatialEst}
\sup_{\eps>0} \|\chi_\eps - \chi_\eps*\phi_\delta\|_{L^1(\mathcal{U})}\leq C\delta,
\end{equation}
where $C$ only depends on $\mathcal{U}$.
By the Fr\'echet--Kolmogorov Theorem, we also have that
\begin{equation}\label{eqn:nablaEst}
\sup_{\eps>0} \||\nabla \mathfrak{u}_\eps|^2 - |\nabla \mathfrak{u}_\eps|^2*\phi_\delta\|_{L^1(\mathcal{U})} \to 0 \quad \text{ as }\delta\to 0. 
\end{equation}
Together with the Lipschitz continuity of $Q^2$, \eqref{eqn:spatialEst} also controls the weighted term: for $|y| \leq \delta$, $\|Q^2\chi_\eps - (Q^2\chi_\eps)(\cdot-y)\|_{L^1(\mathcal{U})} \leq Q_{\max}^2\|\chi_\eps - \chi_\eps(\cdot-y)\|_{L^1(\mathcal{U})} + \|Q^2\|_{C^{0,1}} \delta |\mathcal{U}|$, so $\sup_\eps\|Q^2\chi_\eps - (Q^2\chi_\eps)*\phi_\delta\|_{L^1(\mathcal{U})} \leq C\delta$.

With these estimates in hand, we can show that $e_\eps$ is Cauchy in $L^1(\mathcal{U})$ as follows. Up to a diagonalization argument, we can assume for every $\delta>0$ (for a sequence $\delta\to 0$) that $e_\eps*\phi_\delta$ converges in $L^1$ as $\eps\to 0$ by \eqref{eqn:timeDerivative} and the spatial regularization. Then for $\eps_1$ and $\eps_2$, we have
$$\|e_{\eps_1} - e_{\eps_2}\|_{L^1(\mathcal{U})} \leq \|e_{\eps_1} - e_{\eps_1}*\phi_\delta\|_{L^1(\mathcal{U})} + \|e_{\eps_1}*\phi_\delta - e_{\eps_2}*\phi_\delta\|_{L^1(\mathcal{U})} + \|e_{\eps_2}*\phi_\delta - e_{\eps_2}\|_{L^1(\mathcal{U})}. $$ By choosing $\delta$ small, the first and third term may be taken arbitrarily small due to \eqref{eqn:spatialEst} and \eqref{eqn:nablaEst}. For this choice of $\delta$, the middle term vanishes as $\eps_1,\eps_2\to 0,$ thereby proving that the sequence is Cauchy. This shows $e_\eps$ converges and thereby $\chi_\eps$ converges to $\chi$ in $L^1(\mathcal{U})$.

By an abuse of notation, we reuse the parameters $\delta$ and $\delta_0$ in the following.
Note that by definition $0\leq \chi_\eps\leq 1$, so to show that $\chi$ only takes the values $0$ and $1$, we show that for any $\delta > 0$ it holds that $\mathcal{L}^{d+1}(\{\delta< \chi < 1-\delta\}\cap \mathcal{U})  = 0.$ 
For this, we introduce a secondary parameter $\delta_0$ and use that $\beta > 0$ on $(0,1)$, implying that
\begin{equation}\label{eqn:chi1}
c_{\beta,\delta_0}\frac{1}{\eps}\mathcal{L}^{d+1}(\{\delta_0< \mathfrak{u}_\eps/\eps < 1-\delta_0\} \cap \mathcal{U}) \leq \int_{\mathcal{U}} \beta_\eps(\mathfrak{u}_\eps) \leq C_{\mathcal{U}} 
\end{equation}
where we have used $\beta_\eps(\mathfrak{u}_\eps) \leq Q_{\min}^{-2} Q^2\beta_\eps(\mathfrak{u}_\eps)$ and the computation in \eqref{eqn:chiEpsBV} with a non-negative cutoff $\eta \geq {\bf 1}_{\mathcal{U}}$. 
Recall $\mathcal{B}_1$ from \eqref{eqn:chiEps}. For $\delta_0$ sufficiently small we have $2\mathcal{B}_1(\delta_0) < \delta$ and $1-\delta < 2\mathcal{B}_1(1-\delta_0)$, and inserting this into \eqref{eqn:chi1} gives 
$$ \mathcal{L}^{d+1}(\{\delta< \chi_\eps < 1-\delta\} \cap \mathcal{U}) \leq C(\beta,\delta_0,\mathcal{U}) \eps.$$
Sending $\eps\to 0$ in the above estimate concludes the claim. \qed

\subsection{Proof of \pref{supersolution-limit}}\label{s.proof-supersolution-limit}

    Suppose that $\varphi$ is a smooth test function touching $u_\infty$ from below at $x_0 \in U$. Suppose, towards a contradiction, that
    \[\Delta\varphi(x_0)>0 \ \hbox{ and if $\varphi(x_0) = 0$ then} \ |\grad \varphi(x_0)| > Q(x_0).\]
    Without loss of generality, we can assume that $\overline{B_r(x_0)} \subset U$, $\varphi < u_\infty$ on $\partial B_r(x_0)$ (recall Remark~\ref{r.strict-touching}), and $\varphi$ is a smooth strict subsolution of \eqref{e.bernoulli-basic} in $\overline{B_r(x_0)}$. Choose $a_1>0$ sufficiently small so that, for all $|a| \leq a_1$, $\varphi+a$ is still a smooth strict subsolution of \eqref{e.bernoulli-basic} in $\overline{B_r(x_0)}$ with $\varphi+a < u_\infty$ on $\partial B_r(x_0)$. Note also that $\varphi - a_1 < u_\infty$ on $\overline{B_r(x_0)}$.

    By the uniform convergence of $\mathfrak u(\cdot,t)$ to $u_\infty$ on $\overline{B_r(x_0)}$ as $t \to \infty$, there exists $T>0$ so that for all $t \geq T$
    \[\varphi + a_1 < \mathfrak u(\cdot,t) \ \hbox{ on } \ \partial B_r(x_0) \quad \hbox{and} \quad \varphi - a_1 < \mathfrak u(\cdot,t) \ \hbox{ on } \ \overline{B_r(x_0)}.\]
    With $\zeta$ as in the proof of \pref{subsolution-limit}, define for $\eta>0$
    \[ \psi(x,t) = \varphi(x) - a_1 \zeta(\eta(t-T)).\]
    Since $|\partial_t a_1\zeta(\eta(t-T))| \leq C\eta a_1$ we can choose $\eta>0$ sufficiently small so that $\psi$ is a smooth strict subsolution of \eqref{e.bernoulli-parabolic} in $\overline{B_r(x_0)} \times [T,\infty)$. By the two inequalities above and the definition of $\psi$, $\psi < \mathfrak u$ on $\partial_P(B_r(x_0)\times[T,\infty))$, in particular $\psi \prec \mathfrak u$ there. 
    
    By Definition \ref{d.parabolic-supersolution}, $\psi \prec \mathfrak u$ in $B_r(x_0) \times [T,\infty)$. But then
\[u_\infty(x_0) = \lim_{t \to \infty} \mathfrak{u}(x_0,t) \geq \lim_{t \to \infty} \psi(x_0,t) = \varphi(x_0)+a_1 > \varphi(x_0) = u_\infty(x_0),\]
which is a contradiction. \qed

\section{Details on the semilinear approximation}\label{s.semilinear}

This appendix concerns the semilinear approximation problem \eref{bernoulli-parabolic-semilinear}. We extend existing results in the literature from \cite{CaffarelliVazquez1995} and \cite{Petrosyan} to the case of $x$-dependent coefficients in a bounded domain. Continuity at the boundary and initial time both for the solution and the positive phases is extremely important in the analysis carried out in the main body of the paper. Since these results are not available in the literature, we give complete proofs here using standard arguments.

\subsection{Basic properties}\label{s.basic-properties} For fixed $\ep$, the nonlinearity $f(x,z) := -Q(x)^2\beta_\eps(z)$ is bounded and Lipschitz in $z$ uniformly in $x$, so for parabolic boundary data $g_\ep$ that is Lipschitz and time-independent, standard existence and uniqueness theory for semilinear parabolic equations (see, for instance, \cite{LadyzhenskayaSolonnikovUraltseva}*{Chapter V} and \cite{Lieberman_Book}*{Chapter IX}) yields a solution $\mathfrak{u}_\ep \in C\big(\overline{U}\times[0,\infty)\big) \cap C^{2,1}(U_\infty)$.  If $Q$ is smooth then the interior regularity can be upgraded to $\mathfrak{u}_\ep \in C^{\infty}(U_\infty)$.  This will be convenient for the Bernstein argument for the Lipschitz bound later. Furthermore, again using that $z \mapsto \beta_\ep(z)$ is Lipschitz, the parabolic comparison principle holds between continuous viscosity sub and supersolutions of \eqref{e.bernoulli-parabolic-semilinear}; see, for example, \cite{UsersGuide}*{Theorem 8.2}. To apply that result, one can make the equation proper by proving comparison for $e^{-\lambda t} u$ with $\lambda > \textup{Lip}(\beta_\ep)$.

\begin{lemma}\label{l.monotone}
Assume that $g_\ep$ is a continuous stationary subsolution (resp. supersolution) of \eqref{e.bernoulli-parabolic-semilinear}. Then the unique classical solution $\mathfrak{u}_\ep$ of \eqref{e.bernoulli-parabolic-semilinear} is monotone increasing (resp. monotone decreasing) in $t$.
\end{lemma}
\begin{proof}
By comparison principle, $\mathfrak{u}_\ep(x,h) \geq g_\ep(x)$ for all $h >0$. Consider $\mathfrak{u}_\ep(x,t+h)$, which also solves \eqref{e.bernoulli-parabolic-semilinear} with initial data $\mathfrak{u}_\ep(x,h) \geq g_\ep(x)$ and lateral boundary data $g_\ep(x)$. In other words, it is a supersolution of \eqref{e.bernoulli-parabolic-semilinear} so by comparison $\mathfrak{u}_\ep(x,t+h) \geq \mathfrak{u}_\ep(x,t)$ for all $t>0$. Since $h>0$ was arbitrary we derive the monotonicity property. The case of supersolution initial data is symmetrical.
\end{proof}

\subsection{Energy estimates}\label{s.energy-estimates} Multiplying the equation \eqref{e.bernoulli-parabolic-semilinear} by $\partial_t \mathfrak{u}_\ep$ and integrating on $U \times [0,T]$ gives the energy dissipation equality
\begin{equation}\label{eqn:dissipation}
\frac{1}{2}J(\mathfrak{u}_\ep(T),\chi_\ep(T);U) + \int_0^T\int_U (\partial_t \mathfrak{u}_\ep)^2 \, dx\, dt  = \frac{1}{2}J(\mathfrak{u}_\ep(0),\chi_\ep(0);U).
\end{equation}
This shows that $\partial_t\mathfrak{u}_\eps \in L^2(U_\infty)$ and $\grad \mathfrak{u}_\ep \in L^\infty_tL^2_x(U_\infty)$ with bounds uniform in $\ep>0$ as long as the initial data has uniformly bounded energy.

\subsection{Well-prepared data}\label{s.well-prepared-data-proofs}

In this section we give the proofs of \lref{profile-sub}, \lref{profile-super} and \lref{well-prepared}. The goal is to construct good approximations $g_\ep$ of the data $g$ that preserve the strict sub / supersolution property of $g$ with only a small perturbation. The idea is to compose $g$ with the one-dimensional transition layer for the semilinear ODE $Z'' = \beta(Z)$, although we need a bit of extra room since $g$ is not exactly linear.

\subsubsection{Proof of \lref{profile-sub}}\label{s.profile-sub-proof}
     Since $0 \leq 2\mathcal{B} \leq 1$, the function
\[G(z) := \int_1^z \frac{dw}{\sqrt{(1-\theta) + 2\theta  \mathcal{B}(w)}}\]
is well defined for $z \in \R$, is $C^1$ with $G' \in [1, (1-\theta)^{-1/2}]$, and is therefore a bijection of $\R$ onto $\R$. Define $\Phi(s) := G^{-1}(s - 1)$. Then $\Phi(1) = 1$ and
\begin{equation}\label{e.profile-foc}
\Phi'(s) = \sqrt{(1-\theta) + 2\theta  \mathcal{B}(\Phi(s))} \in [\sqrt{1-\theta},  1].
\end{equation}
Differentiating \eref{profile-foc} and using $\mathcal{B}' = \beta$ gives $\Phi'' = \theta \beta(\Phi)$. Since $\beta \in C^\infty$, a bootstrap of this identity gives $\Phi \in C^\infty(\R)$. This proves \ref{part.profile-sub-a}.

For $s \geq 1$ we have $\Phi(s) \geq 1$ (as $\Phi' > 0$), hence $\mathcal{B}(\Phi(s)) = \tfrac12$, whence $\Phi'(s) = 1$ for all $s \geq 1$. With $\Phi(1)=1$, this gives $\Phi(s) = s$, proving  \ref{part.profile-sub-b}.

The function $s \mapsto s - \Phi(s)$ is nondecreasing by \eref{profile-foc} and vanishes at $s = 1$, so $\Phi(s) \geq s$ for $s \leq 1$; also $\Phi(s) \leq \Phi(1) = 1$ for $s \leq 1$ since $\Phi' > 0$. In particular, $\Phi(0) \geq 0$, so the unique zero $s_0$ of the strictly increasing function $\Phi$ satisfies $s_0 \leq 0$. On $[0,1]$, $\Phi' \geq \sqrt{1-\theta}$ gives $\Phi(0) \leq \Phi(1) - \sqrt{1-\theta} = 1 - \sqrt{1-\theta}$, and then $\Phi' \geq \sqrt{1-\theta}$ on $[s_0, 0]$ gives $|s_0| \leq \Phi(0)/\sqrt{1-\theta} \leq c_\theta$. Finally, for $s \geq 1$ we have $(\Phi(s))_+ - s_+ = 0$; for $s_0 \leq s \leq 1$ both $(\Phi(s))_+$ and $s_+$ lie in $[0,1]$; and for $s \leq s_0$ both terms vanish since $s \leq s_0 \leq 0$. This proves  \ref{part.profile-sub-c}.

\qed

\subsubsection{Proof of \lref{profile-super}}\label{s.profile-super-proof}
 Set $\tilde\theta := \tfrac{1+\hat\theta}{2} \in (1,\hat\theta)$ and $\sigma := \hat\theta/\tilde\theta - 1 > 0$. Since $\beta > 0$ on $(0,1)$, $2\mathcal{B}$ is a strictly increasing bijection from $[0,1]$ onto $[0,1]$, so there is a unique $\kappa = \kappa_{\hat\theta} \in (0,1)$ with
\begin{equation}\label{e.kappa-def}
2\mathcal{B}(\kappa) = \frac{\tilde\theta - 1}{\tilde\theta} .
\end{equation}
Since $\tilde\theta$ increases with $\hat\theta$, the right side of \eqref{e.kappa-def} increases from $0$ to $1$ as $\hat\theta$ goes from $1$ to $\infty$; as $2\mathcal B$ is a strictly increasing continuous bijection of $[0,1]$, this proves \ref{part.profile-super-f}. Define $P : [\kappa,\infty) \to [0,1]$ by $P(z) := 2\tilde\theta\big(\mathcal{B}(z) - \mathcal{B}(\kappa)\big)$. Then $P$ is smooth and nondecreasing with $P' = 2\tilde\theta\beta \geq 0$, strictly increasing on $[\kappa,1]$, $P(\kappa) = 0$, $P'(\kappa) = 2\tilde\theta\beta(\kappa) > 0$, and
\[P(z) = 2\tilde\theta\Big(\tfrac12 - \tfrac{\tilde\theta-1}{2\tilde\theta}\Big) = 1 \  \hbox{for } z \geq 1 .\]

It would be natural to solve $\Psi' = \sqrt{P(\Psi)}$ with $\Psi(1) = 1$ since then $\Psi'' = \tilde\theta \beta(\Psi)$, but this solution attaches to the level $\kappa$ at some finite value and is only $C^{1,1}$ at the attachment point. To get a smooth solution we instead flatten the vector field quadratically at the bottom of its range, spending the slack $\sigma$. 

Fix a smooth reparametrization $\eta \in C^\infty([0,1];[0,1])$, depending only on $\sigma$, with
\begin{equation}\label{e.eta-properties}
\eta(p) = p^2 \ \hbox{ on } [0,p_0], \  \eta(p) = p \ \hbox{ on } [p_1,1], \  0 < \eta(p) \leq p \ \hbox{ on } (0,1], \  0 \leq \eta' \leq 1+\sigma,
\end{equation}
for suitable $0 < p_0 < p_1 < 1$.

Define $G : [\kappa,\infty) \to [0,1]$ by $G(z) := \sqrt{\eta(P(z))}$. The function $G$ is smooth: near $z = \kappa$ we have $P(z) \leq p_0$, so $G = P$ there, which is smooth; and where $P > 0$ the composition $\eta(P)$ is smooth and positive (by \eref{eta-properties}), so its square root is smooth. Moreover:
\begin{enumerate}[label = (\alph*)]
\item\label{it.g-ineq} $0 \leq (G^2)'(z) = \eta'(P(z)) P'(z) \leq (1+\sigma)\cdot 2\tilde\theta\beta(z) = 2\hat\theta \beta(z)$ for all $z \geq \kappa$;
\item\label{it.g-firstint} $G(z)^2 = \eta(P(z)) \leq P(z)$, so in particular $G \leq 1$, with $G(z) = 1$ for $z \geq 1$ and $G(z) < 1$ for $z < 1$ (as $P < 1$ there);
\item\label{it.g-linear} $G(\kappa) = 0$, $G > 0$ on $(\kappa,\infty)$, and, with $z_0 \in (\kappa,1)$ defined by $P(z_0) = p_0$,
\[c_1 (z-\kappa) \leq G(z) = P(z) \leq c_2 (z-\kappa) \  \hbox{for } z \in [\kappa, z_0],\]
where $c_1 := 2\tilde\theta\min_{[\kappa,z_0]}\beta > 0$ and $c_2 := 2\tilde\theta \max\beta$.
\end{enumerate}
Let $\Psi$ be the maximal solution of the autonomous ODE
\[\Psi' = G(\Psi), \  \Psi(1) = 1 .\]
Since $G$ is locally Lipschitz and $z \equiv \kappa$ is a stationary solution, $\Psi$ is defined on all of $\R$, takes values in $(\kappa,\infty)$, and is strictly increasing. Bootstrapping the ODE with $G$ smooth gives $\Psi \in C^\infty(\R)$. For $s \geq 1$, $G(\Psi) = 1$ once $\Psi \geq 1$, so $\Psi(s) = s$. This and $0 < \Psi' = G(\Psi) \leq 1$ prove  \ref{part.profile-super-b}. Differentiating,
\[\Psi'' = G'(\Psi)  G(\Psi) = \tfrac12 (G^2)'(\Psi) \in \big[0,\ \hat\theta \beta(\Psi)\big]\]
by \ref{it.g-ineq}, which proves \ref{part.profile-super-a}.

For  \ref{part.profile-super-c}: $\Psi > \kappa$ and $\Psi(s) < \Psi(1) = 1$ for $s < 1$ by strict monotonicity. The function $s - \Psi(s)$ has derivative $1 - G(\Psi(s)) > 0$ for $s < 1$ by \ref{it.g-firstint} and vanishes at $s = 1$, so $\Psi(s) > s$ for $s < 1$. As $s \to -\infty$, $\Psi$ decreases to a limit $\ell \geq \kappa$ which must satisfy $G(\ell) = 0$, so $\ell = \kappa$. The bound $\sup_\R|\Psi - s_+| \leq 1$ follows since for $s \leq 1$ both $\Psi(s)$ and $s_+$ lie in $[0,1]$, while for $s \geq 1$ they are equal.

For \ref{part.profile-super-e}: let $s_\# \in \R$ be the time with $\Psi(s_\#) = z_0$ (with $z_0$ from \ref{it.g-linear}). For $s \leq s_\#$ we have $\Psi(s) \in (\kappa, z_0]$, where \ref{it.g-linear} gives $\frac{d}{ds}\log(\Psi - \kappa) = \frac{G(\Psi)}{\Psi - \kappa} \geq c_1$. Integrating from $s$ to $s_\#$,
\[\Psi(s) - \kappa \leq (z_0 - \kappa)  e^{c_1 (s - s_\#)} \leq e^{c_1|s_\#|}  e^{c_1 s} \  \hbox{for } s \leq s_\#,\]
while for $s_\# \leq s \leq 0$ (if this interval is nonempty) the trivial bound $\Psi(s) - \kappa \leq 1 \leq e^{c_1|s_\#|}e^{c_1 s}$ holds. On the same regions, $\Psi' = G(\Psi) \leq c_2(\Psi - \kappa)$ for $s \leq s_\#$ and $\Psi' \leq 1$ otherwise, and $\Psi'' = \tfrac12\eta'(P)P'(\Psi) \leq P(\Psi)\cdot 2\tilde\theta\max\beta \leq c_2^2(\Psi-\kappa)$ for $s \leq s_\#$ (using $\eta' \leq 2P$ on $[0,p_0]$) and $\Psi'' \leq \hat\theta\max\beta$ otherwise. Since $s_\#$, $c_1$, $c_2$, $z_0$ depend only on $\hat\theta$ and $\beta$, all three bounds combine into \ref{part.profile-super-e} with $c_{\hat\theta} := c_1$ and a suitable $C_{\hat\theta} = C_{\hat\theta}(\hat\theta,\beta)$.
\qed

\subsubsection{Proof of Lemma~\ref{l.well-prepared}}\label{s.well-prepared-proof}

Set $L := \|\grad{g}\|_{L^\infty(U)}$ and $D := \|\Delta{g}\|_{L^\infty(U)}$. We first require $\ep_0 \leq \tfrac{a_0}{1 + c_\theta}$, so that in the increasing case the transition region of the profile is contained in $\{|{g}| \leq a_0\}$, as quantified below.  

\emph{Subsolution case.} Write $\Phi := \Phi_{\ep,\theta}$ and recall from \lref{profile-sub}\ref{part.profile-sub-c} the zero $\ep s_0 \in [-c_\theta\ep, 0]$ of $\Phi$. Item \ref{part.gep-1} holds since $0 \leq \Phi' \leq 1$ and ${g} \in C^{0,1}$, with $g_\ep = (\Phi({g}))_+ = \Phi({g}) = {g}$ on $\{{g} \geq \ep\}$. Item \ref{part.gep-2}: $\Phi(s) \geq s$ for all $s$ gives $(\Phi({g}))_+ \geq {g}_+$, and $\sup_\R|\Phi_+ - s_+| \leq \ep$ gives the upper bound. Item \ref{part.gep-5} is immediate from the strict monotonicity of $\Phi$ and $\{\Phi > 0\} = (\ep s_0,\infty)$. Item \ref{part.gep-4} follows from \ref{part.gep-1} and $0 \leq \chi^0_\eps \leq 1$:
\[J(g_\ep,\chi^0_\eps;U) \leq \int_U |\grad {g}|^2 {\bf 1}_{\{g_\ep>0\}}   dx + \int_U Q(x)^2   dx \leq \int_U |\grad{g}|^2   dx + Q_{\max}^2 |U|.\]

We prove \ref{part.gep-3}. Let $\psi \in C^2(U)$ touch $g_\ep$ from above at $x_* \in U$; we must show $\Delta \psi(x_*) \geq Q(x_*)^2 \beta_\eps(g_\ep(x_*))$.

First suppose $g_\ep(x_*) = 0$. Then $\psi \geq g_\ep \geq 0$ near $x_*$ and $\psi(x_*) = 0$, so $x_*$ is a local minimum of $\psi$ and $\Delta \psi(x_*) \geq 0 = Q(x_*)^2\beta_\eps(0)$, where $\beta_\eps(0) = 0$ since $\operatorname{supp}\beta = [0,1]$ and $\beta$ is continuous.

Now suppose $g_\ep(x_*) > 0$, i.e.\ ${g}(x_*) > \ep s_0$. Near $x_*$ we have $g_\ep = \Phi({g}) \in C^2$, so it suffices to verify the classical inequality at $x_*$:
\begin{equation}\label{e.chain-sub}
\Delta\big(\Phi({g})\big)(x_*) = \Phi''({g}) |\grad{g}|^2 + \Phi'({g}) \Delta{g} \geq Q(x_*)^2\beta_\eps\big(\Phi({g}(x_*))\big).
\end{equation}
Since ${g}(x_*) > \ep s_0 \geq -c_\theta\ep \geq -a_0$, the strict subsolution property gives $\Delta {g}(x_*) \geq 0$, and $\Phi' > 0$, so the second term in \eref{chain-sub} is non-negative. If $\beta_\eps(\Phi({g}(x_*))) = 0$ then \eref{chain-sub} follows since $\Phi'' = \theta\beta_\eps(\Phi) = 0$ there. If $\beta_\eps(\Phi({g}(x_*))) > 0$ then $0 < \Phi({g}(x_*)) < \ep$, which by \lref{profile-sub}\ref{part.profile-sub-c} forces $\ep s_0 < {g}(x_*) < \ep$, hence $|{g}(x_*)| \leq a_0$ by the choice of $\ep_0$. Then the strict subsolution property gives $|\grad {g}(x_*)|^2 \geq (1+\delta_0) Q(x_*)^2$, and, using $\Phi'' = \theta\beta_\eps(\Phi)$,
\begin{align*}
    \Phi''({g})|\grad{g}|^2 &\geq \theta (1+\delta_0) Q(x_*)^2 \beta_\eps(\Phi({g})) \geq Q(x_*)^2 \beta_\eps(\Phi({g}))
\end{align*}
at $x_*$, using $\theta(1+\delta_0) \geq 1$ for the second inequality. This proves \eref{chain-sub}, hence \ref{part.gep-3}.

\emph{Supersolution case.} Write $\Psi := \Psi_{\ep,\hat\theta}$ and $\kappa := \kappa_{\hat\theta} \in (0,1)$ from \lref{profile-super}. Item \ref{part.gep-1} holds since $0 < \Psi' \leq 1$, with $g_\ep = \Psi({g}) = {g}$ on $\{{g} \geq \ep\}$. Item \ref{part.gep-2}: $\Psi(s) \geq s$ and $\Psi > 0$ give $\Psi(s) \geq s_+$, hence $\Psi({g}) \geq {g}_+$, and $\sup_\R|\Psi - s_+| \leq \ep$ gives the upper bound. Item \ref{part.gep-5} is $\Psi > \ep\kappa$. Item \ref{part.gep-4} follows as before.

For \ref{part.gep-3}, since $\Psi \in C^\infty(\R)$ by \lref{profile-super} and ${g} \in C^2(\overline U)$, the composition $g_\ep = \Psi({g})$ belongs to $C^2(\overline U)$ and the supersolution property is the classical pointwise inequality
\begin{equation}\label{e.chain-super-goal}
\Delta\big(\Psi({g})\big) = \Psi''({g}) |\grad{g}|^2 + \Psi'({g}) \Delta{g} \leq Q(x)^2\beta_\eps\big(\Psi({g})\big) \  \hbox{in } U,
\end{equation}
which we verify in three regions.

Case ${g}(x) \geq \ep$. Then $\Psi({g}) = {g} \geq \ep$ at $x$, so $\beta_\eps(\Psi({g})) = 0$ and $\Psi''({g}) = 0$ (as $\Psi$ is affine on $[\ep,\infty)$), while $\Delta{g}(x) \leq 0$ since $g$ is a strict supersolution  as in \dref{strict-sub} (as ${g}(x) \geq \ep > -a_0$). Hence the left side of \eref{chain-super-goal} is non-positive.

Case $-a_0 < {g}(x) < \ep$. Then $|{g}(x)| \leq a_0$ (using $\ep \leq a_0$), so the strict supersolution property of $g$ \dref{strict-sub} gives $\Delta {g}(x) \leq 0$ and $|\grad{g}(x)|^2 \leq (1-\delta_0) Q(x)^2$. Then, using $0 \leq \Psi'' \leq \hat\theta\beta_\eps(\Psi)$ from \lref{profile-super}\ref{part.profile-super-a} and $\hat\theta \leq \frac{1}{1-\delta_0}$,
\begin{align*}
    \Psi''({g})|\grad{g}|^2 &\leq \hat\theta(1-\delta_0)Q(x)^2\beta_\eps(\Psi({g})) \leq Q(x)^2\beta_\eps(\Psi({g})).
\end{align*}
 Also $\Psi'\Delta{g} \leq 0$ since $\Delta g \leq 0$ on $g(x) \geq - a_0$. Thus, again, we have derived \eref{chain-super-goal}.

Case ${g}(x) \leq -a_0$. Here the strict supersolution property of $g$ gives no information, and we use the exponential decay of the profile instead. By \lref{profile-super}\ref{part.profile-super-e},
\[\Psi''({g})|\grad{g}|^2 + \Psi'({g})\Delta{g} \leq \frac{C_{\hat\theta}}{\ep} e^{-c_{\hat\theta}a_0/\ep} \big(L^2 + \ep D\big) \leq \frac{C_{\hat\theta}(L^2 + D)}{\ep} e^{-c_{\hat\theta} a_0/\ep} .\]
On the other hand $\Psi({g}(x)) = \ep\Psi_{1,\hat{\theta}}({g}(x)/\ep) \in (\ep\kappa,\ep\Psi_{1,\hat\theta}(-a_0/\ep)]$, and by \lref{profile-super}\ref{part.profile-super-c}, 
\[\Psi_{1,\hat{\theta}}(-a_0/\ep) \leq \kappa + C_{\hat\theta}e^{-c_{\hat\theta}a_0/\ep} \leq \tfrac{1+\kappa}{2} \ \hbox{ for $\ep$ small enough that } \ C_{\hat\theta}e^{-c_{\hat\theta}a_0/\ep} \leq \tfrac{1-\kappa}{2}.\]
 Hence, with $b_{\hat\theta} := \min_{[\kappa, \frac{1+\kappa}{2}]}\beta > 0$,
\[Q(x)^2\beta_\eps\big(\Psi({g}(x))\big) = \frac{Q(x)^2}{\ep}\beta\left(\frac{\Psi({g}(x))}{\ep}\right) \geq \frac{Q_{\min}^2 b_{\hat\theta}}{\ep},\]
and \eref{chain-super-goal} holds in this region as soon as $C_{\hat\theta}(L^2 + D) e^{-c_{\hat\theta}a_0/\ep} \leq Q_{\min}^2 b_{\hat\theta}$. Together with the condition of the previous display this fixes the threshold $\ep_0$, with the stated dependencies.
\qed

\subsection{Lipschitz estimate in space}\label{s.lipschitz-space}
 
We prove that solutions of the semilinear problem \eqref{e.bernoulli-parabolic-semilinear} are locally Lipschitz in space, uniformly in $\ep$. 

We prove two versions of the Lipschitz estimate. Both proofs are based on the Bernstein method, and are adaptations of the proofs of \cite{CaffarelliVazquez1995}*{Theorems 4.1 and 5.1}. The difference in our case is the spatial localization and the $x$-dependent coefficient field $Q(x)$. The inhomogeneous coefficients are the slightly trickier part, and we cannot avoid taking a derivative of the coefficient field here.

The first Lipschitz estimate is localized in both time and space.  

\begin{proposition}\label{p.localLip}
 Let $0 < r \leq 1$ and let $\mathfrak{u}_\ep \geq 0$ be a classical solution of 
\[(\partial_t - \Delta)\mathfrak{u}_\ep = - Q(x)^2\beta_\ep(\mathfrak{u}_{\ep}) \ \hbox{ in } \ \mathcal{Q}_{2r}\] and continuous on $\overline{\mathcal{Q}_{2r}}$. Then
	\begin{equation}\nonumber
	\|\nabla \mathfrak{u}_\ep\|_{L^\infty(\mathcal{Q}_r)} \leq C \left(\frac{\|\mathfrak{u}_\ep\|_{L^\infty(\mathcal{Q}_{2r})}}{r} + 1\right)
	\end{equation}
	with $C = C(d, Q_{\min}, Q_{\max}, \|\grad(Q^2)\|_{L^\infty}, \|\beta\|_{C^1})$.
\end{proposition}

The second is localized only in space and assumes the initial data satisfies a Lipschitz bound.

\begin{proposition}\label{p.localLip-initial-data}
 Let $0 < r \leq 1$ and let $\mathfrak{u}_\ep \geq 0$ be a classical solution of 
\[(\partial_t - \Delta)\mathfrak{u}_\ep = - Q(x)^2\beta_\ep(\mathfrak{u}_{\ep}) \ \hbox{ in } \ B_{2r} \times (0,T]\]   for some $0 < T \leq 4r^2$ and $u$ and $\grad u$ are continuous on $\overline{B_{2r}} \times [0,T]$. Then
	\begin{equation}\nonumber
	\|\nabla \mathfrak{u}_\ep\|_{L^\infty(B_r \times [0,T])} \leq C\left(\|\grad \mathfrak{u}_\ep(\cdot,0)\|_{L^\infty(B_{2r})}+\frac{\|\mathfrak{u}_\ep\|_{L^\infty(B_{2r}\times[0,T])}}{r}+1\right)
	\end{equation}
	with $C = C(d, Q_{\min}, Q_{\max}, \|\grad(Q^2)\|_{L^\infty}, \|\beta\|_{C^1})$.
\end{proposition}

As a consequence of these two we obtain the global-in-time interior Lipschitz estimate for \eqref{e.bernoulli-parabolic-semilinear}.

\begin{lemma}\label{l.lipschitz-initial}
Let $g_\ep$ non-negative and Lipschitz on $\overline{U}$. Let $\mathfrak{u}_\ep$ be the solution of \eqref{e.bernoulli-parabolic-semilinear}. For every $V \subset\subset U$ there is $C_V$ so that
\[\sup_{0 < \ep \leq 1}\ \sup_{t > 0}\ \|\grad \mathfrak{u}_\ep(\cdot,t)\|_{L^\infty(V)} \leq C_V(1+\|g_\ep\|_{L^\infty(U)}+\|\grad g_\ep\|_{L^\infty(U)}) .\]
Here $C_V $ depends only on $d(V,\R^d \setminus U)$, $N$, $C_\beta$, $d$, $Q_{\min}$ and $Q_{\max}$.
\end{lemma}

\begin{proof}
By a mollification argument using the uniqueness of continuous viscosity solutions of \eqref{e.bernoulli-parabolic-semilinear} it suffices to consider the case $g_\ep$ is smooth, so that $\grad \mathfrak{u}_\ep$ is continuous at $t=0$ on $V$.

Call $r = \min\{1, \frac12 d(V,\R^d \setminus U)\}$. Let $(x,t) \in V \times (0,\infty)$. If $t \geq 4r^2$ then $\mathcal{Q}_{2r}(x,t) \subset U_\infty$ and \pref{localLip} gives $|\grad \mathfrak{u}_\ep(x,t)| \leq C(\|g_\ep\|_\infty/r + 1)$. If $0 < t < 4r^2$ we apply \pref{localLip-initial-data} and get $|\grad \mathfrak{u}_\ep(x,t)| \leq C(\|\grad g_\ep\|_\infty +\|g_\ep\|_\infty/r+ 1)$.

\end{proof}

Now we return to prove \pref{localLip} and \pref{localLip-initial-data}. The two proofs are very similar so we do both at the same time.  

\begin{proof}[Proof of \pref{localLip} and \pref{localLip-initial-data}]

We can assume that $Q$ is smooth so that $\mathfrak{u}_\ep$ is smooth on the interior to justify all the derivative computations below, the end estimate will only depend on the Lipschitz norm of $Q$ so we can derive the general Lipschitz case by approximation. By parabolic rescaling it suffices to prove the estimate at unit scale. For $r \leq 1$ then $\hat u(y,s) := r^{-1}\mathfrak{u}_\ep(x_0 + ry, t_0 + r^2 s)$ solves the same equation on $\Gamma := B_2 \times (-4,0]$ with $\hat\ep := \ep/r$ (note $r\beta_\eps(r z) = \beta_{\ep/r}(z)$) and $\hat Q(y) := Q(x_0 + ry)$ in place of $\ep$ and $Q$. Then $\|\grad(\hat Q^2)\|_\infty = r\|\grad(Q^2)\|_\infty \leq N_Q$ since $r \leq 1$. Likewise the scaling $u \mapsto u/Q_{\max}$, $\ep \mapsto \ep/Q_{\max}$, $Q \mapsto Q/Q_{\max}$ maps solutions to solutions, so we may normalize $Q_{\max} = 1$.  Note that due to the former scaling argument we must handle all $ \ep >0$, for now we assume that $\ep \leq 1$, the case $\ep \geq 1$ is easier and done at the end.

Note that in the case of \pref{localLip-initial-data} the rescaled domain is $\Gamma := B_2 \times (-\tau_0,0]$ for  $\tau_0 = \frac{T}{r^2} \in (0,4]$.

Call $M:= \|u\|_{L^\infty(\mathcal{Q}_2)}$ and $L= \|\grad u(\cdot,-\tau_0)\|_{L^\infty(B_2)}$ (for \pref{localLip-initial-data}).

\emph{Step 1.} We introduce the change of variables
\[u = \phi(v)\]
for some increasing and smooth function $\phi : [0,\infty) \to [0,\infty)$ with $\phi(0) = 0$. Call $\ell(v) := \phi''(v)/\phi'(v)$. Since $\grad u = \phi'(v)\grad v$ and $\Delta u = \phi'(v)\Delta v + \phi''(v)|\grad v|^2$, after the change of variables the equation becomes
\[\partial_t v - \Delta v - \ell(v) |\grad v|^2 =- k(x,v) \ \hbox{ with } \  k(x,v) := \frac{Q(x)^2 \beta_\ep(\phi(v))}{\phi'(v)} \geq 0 .\]
Differentiating with respect to $x_i$, multiplying by $2\partial_{x_i}v$ and summing over $i$, we find that the quantity $w := |\grad v|^2$ satisfies
\begin{equation}\label{e.w-eq}
\partial_t w - \Delta w - 2\ell(v)\grad v\cdot\grad w = 2\ell'(v) w^2 - 2\partial_v k w - 2\grad_xk\cdot\grad v - 2|D^2v|^2 .
\end{equation}
The term $2\ell'(v)w^2$ has been created by the change of variables, and we will choose $\phi$ so that $\ell'(v)$ is negative everywhere and very negative across the transition layer. In particular we will have $2\ell'(v) w^2 = -2|\ell'(v)|w^2$. 

\emph{Step 2.} Put $v_{\max} := \ep + 4M$ and $a := \frac{1}{8v_{\max}^2}$. The following properties of $\phi$ will be used; the construction is postponed to Step 4.
\begin{enumerate}[label = (\alph*)]
    \item\label{phi.a} $1\leq \phi' \leq e$ on $[0,v_{\max}]$, and $\phi(v_{\max}) \geq 2M$.
    \item\label{phi.b} $\phi(\ep) \geq \ep$ so that $\{v < \ep\} = \{\phi^{-1}(u)<\ep\} = \{u < \phi(\ep)\} \supset \{u < \ep\}$.
    \item\label{phi.c} On $[0,\ep]$: $\ell' \leq -\ep^{-2}$, $0 \leq \ell \leq 2/\ep$, and $\phi' \geq 1$.
    \item\label{phi.d} On $[\ep,\infty)$: $\ell' \leq -a$.
    \item\label{phi.e} $\ell^2 \leq 2|\ell'|$ on $[0,v_{\max}]$.
\end{enumerate}
By \ref{phi.a}, $\phi$ maps $[0,v_{\max}]$ onto an interval containing $[0,2M]$, so $v = \phi^{-1}(u)$ is well defined, smooth in $\Gamma$, continuous on $\overline\Gamma$, and takes values in $[0,v_{\max}]$.

Let us record what these properties give for the coefficients in \eref{w-eq}. In $\{v < \ep\}$ we have
\[k_v = Q^2\Big(\beta_\ep'(u) - \frac{\beta_\ep(u) \ell(v)}{\phi'(v)}\Big), \  \grad_x k = \frac{\beta_\ep(u) \grad(Q^2)}{\phi'(v)},\]
and since $|\beta_\ep| \leq \|\beta\|_\infty\ep^{-1}$, $|\beta_\ep'| \leq \|\beta'\|_\infty\ep^{-2}$, $0 \leq \ell \leq 2/\ep$ and $\phi' \geq 1$ there,
\[|k_v| \leq C_1 \ep^{-2}, \  |\grad_xk| \leq N\|\beta\|_\infty \ep^{-1}, \  C_1 := \|\beta'\|_\infty + 2\|\beta\|_\infty .\]
On $\{v\geq \ep\}$ then $\phi(v) \geq \ep$ by \ref{phi.b} so $\beta_\ep(\phi(v)) \equiv 0$ and $k \equiv 0$.

\emph{Step 3.} We localize with a cutoff $\eta$, $0 \leq \eta \leq 1$, with $|\grad\eta| + |\partial_t\eta| + |D^2\eta| \leq C_d$. For \pref{localLip} we take $\eta(x,t) = \eta_1(x)\eta_2(t)$ with $\eta_1 \in C^\infty_c(B_2)$, $\eta_1 = 1$ on $B_1$, $\eta_2 = 1$ on $[-1,0]$ and $\eta_2 = 0$ on $(-\infty,-3]$. For \pref{localLip-initial-data} we just cut-off in space $\eta(x,t) = \eta_1(x)$. We consider the quantity
\[z := \eta^2 w .\]
Note that $z$ vanishes where $\eta = 0$. For \pref{localLip} this includes a neighborhood of $\partial_P\Gamma$ and so $z$ is continuous on $\overline \Gamma$. For \pref{localLip-initial-data}, $\eta =0$ on a neighborhood of the spatial boundary $\partial B_2 \times [-\tau_0,0]$, and $\grad v = \grad u/\phi'(v)$ extends continuously to $\overline\Gamma$, so still $z$ is continuous on $\overline \Gamma$.

  If $z \equiv 0$ there is nothing to prove. Otherwise $z$ attains a positive maximum at a point $(\hat x,\hat t)$ where $\eta > 0$ and $w > 0$. We will show that
\begin{equation}\label{e.zmax}
z(\hat x,\hat t) \leq \max\big(1, C + 2N\|\beta\|_\infty,\ C(1 + 4M)^2,L^2\big) .
\end{equation}
If $\hat t = -\tau_0$, which can only happen for \pref{localLip-initial-data}, then by \ref{phi.a} we have $z \leq w = |\grad u|^2/\phi'(v)^2 \leq L^2$ at $(\hat{x},\hat{t})$. Otherwise $(\hat x, \hat t)$ is an interior or final time maximum and the derivative tests hold
\[\partial_t z \geq 0, \grad z = 0, \hbox{ and } \Delta z \leq 0 .\]
Let us compute $\partial_t z - \Delta z - 2\ell\grad v\cdot\grad z$, which is non-negative at $(\hat x,\hat t)$ by the derivative tests. We have
\begin{align*}
\partial_t z - \Delta z - 2\ell\grad v\cdot\grad z &= \eta^2\big(\partial_t w - \Delta w - 2\ell\grad v\cdot\grad w\big) \\ & \  + w\big(\partial_t\eta^2 - \Delta\eta^2 - 2\ell\grad v\cdot\grad\eta^2\big) - 2\grad\eta^2\cdot\grad w .
\end{align*}
The first parenthesis is given by \eref{w-eq}. In the second parenthesis, $|\partial_t\eta^2 - \Delta\eta^2| \leq C$, while by Young's inequality and \ref{phi.e},
\[4|\ell| \eta|\grad\eta| w^{3/2} \leq |\ell'| \eta^2w^2 + \frac{4\ell^2}{|\ell'|}|\grad\eta|^2w \leq |\ell'| \eta^2w^2 + C w .\]
Finally $|\grad w| \leq 2|D^2v| w^{1/2}$, so that
\[-2\grad\eta^2\cdot\grad w \leq 8\eta|\grad\eta| |D^2v| w^{1/2} \leq \eta^2|D^2v|^2 + C w.\]
The positive Hessian term is absorbed by the term $-2\eta^2|D^2v|^2$ coming from \eref{w-eq}. Collecting terms, and using the bounds on $k_v$ and $\grad_xk$ from Step 2, we arrive at
\[0 \leq -|\ell'(v)| \eta^2w^2 + C w + \eta^2\big(2C_1\ep^{-2} w + 2N\|\beta\|_\infty\ep^{-1} w^{1/2}\big) {\bf 1}_{\{v < \ep\}} \  \hbox{at } (\hat x,\hat t) .\]
Dividing by $w > 0$ this reads
\begin{equation}\label{e.zmax-ineq}
|\ell'(v)| z \leq C + \eta^2\big(2C_1\ep^{-2} + 2N\|\beta\|_\infty \ep^{-1}w^{-1/2}\big) {\bf 1}_{\{v < \ep\}} \  \hbox{at } (\hat x,\hat t) .
\end{equation}
We now analyze \eref{zmax-ineq} separately in the two regions.

  \emph{In $v \geq \ep$.} Here the indicator vanishes and $|\ell'| \geq a$ by \ref{phi.d}, so
\[z \leq \frac{C}{a} = 8C v_{\max}^2 \leq C(1 + 4M)^2 .\]

\emph{In $v < \ep$.} Here $|\ell'| \geq \ep^{-2}$ by \ref{phi.c}, and multiplying \eref{zmax-ineq} by $\ep^2 \leq 1$ we get
\[z \leq C+C_1 + 2N\|\beta\|_\infty \ep w^{-1/2} .\]
If $w < 1$ then $z \leq \eta^2 w < 1$, and if $w \geq 1$ then $z \leq C+C_1 + 2N\|\beta\|_\infty$. 

This proves \eref{zmax}. Since $w = z$ where $\eta = 1$, which contains $B_1 \times (-1,0]$ for \pref{localLip} and $B_1 \times [-\tau_0,0]$ for \pref{localLip-initial-data}, and since $|\grad u|^2 = \phi'(v)^2 w \leq e^2 w$ by \ref{phi.a}, the lemma follows from \eref{zmax} and $M \geq 1$.

\emph{Step 4. Construction of $\phi$.} It is convenient to prescribe $\ell = \phi''/\phi'$ and then recover $\phi$ from $\phi'(v) = \exp\big(\int_0^v \ell\big)$, $\phi(v) = \int_0^v\phi'$. Let $\rho := \ep^2/(64 v_{\max})$, fix a nonincreasing $\psi \in C^\infty(\R)$ with $\psi = 1$ on $(-\infty,0]$ and $\psi = 0$ on $[1,\infty)$, and set
\[\ell(v) := \int_v^\ep \Big(\ep^{-2}\psi\Big(\frac{s-\ep}{\rho}\Big) + a\Big) ds, \  \hbox{so that} \  \ell'(v) = -\ep^{-2}\psi\Big(\frac{v-\ep}{\rho}\Big) - a .\]
Thus $\ell' = -\ep^{-2} - a$ on $[0,\ep]$, $\ell' = -a$ on $[\ep + \rho, \infty)$, and $\ell'$ is between these values in the short interval $[\ep, \ep+\rho]$. In particular \ref{phi.d} holds. On $[0,\ep]$ we have $\ell(v) = (\ep^{-2} + a)(\ep - v)$, so $0 \leq \ell \leq (\ep^{-2}+a)\ep \leq 2/\ep$ because $a = \frac{1}{8v_{\max}^2}\le \frac{1}{8\ep^2}$, as $v_{\max} \ge \ep$. Consequently $\phi'$ increases on $[0,\ep]$ from $\phi'(0) = 1$ to
\[\phi'(\ep) = \exp\big(\tfrac12(1 + a\ep^2)\big) \in [e^{1/2}, e],\]
which gives \ref{phi.c}, \ref{phi.b}, and the bounds of \ref{phi.a} on $[0,\ep]$.

For $v \geq \ep$, $\ell$ is negative, and since $0 \leq \psi \leq 1$ vanishes on $[1,\infty)$,
\begin{equation}\label{e.ell-bds-above-ep}
a(v-\ep) \leq -\ell(v) \leq \frac{\rho}{\ep^2} + a(v - \ep) .
\end{equation}
Hence, for $\ep \leq v \leq v_{\max}$, using $(\rho\ep^{-2})^2 = a/512$ and $a v_{\max}^2 = \tfrac18$,
\[\int_\ep^v|\ell| \leq \frac{\rho}{\ep^2} v_{\max} + \frac{a}{2}v_{\max}^2 \leq \frac1{16} + \frac1{16}, \ \hbox{so that} \ \phi'(v) = \phi'(\ep) e^{-\int_\ep^v|\ell|} \geq e^{\frac{1}{2}-\frac{1}{8}} \geq 1,\]
while $\phi'(v) \leq \phi'(\ep) \leq e$ by monotonicity. Then $\phi(v_{\max}) \geq \phi(\ep) + \frac{\phi'(\ep)}{2}\cdot 4M \geq 2M$, which completes \ref{phi.a}. Finally, for \ref{phi.e}: on $[0,\ep]$, $\ell^2/|\ell'| = (\ep^{-2} + a)(\ep-v)^2 \leq 1 + a\ep^2 \leq 2$, and on $[\ep,v_{\max}]$, using $|\ell'| \geq a$ and \eqref{e.ell-bds-above-ep},
\[\frac{\ell^2}{|\ell'|} \leq \frac{2\rho^2\ep^{-4}}{a} + 2a(v-\ep)^2 \leq \frac1{256} + 2a v_{\max}^2 \leq \frac12 .\]
This completes the construction and the proof in the case $\ep \leq 1$.

\emph{Step 5.}  Finally we mention how to handle the case $\ep \geq 1$. In this case no change of variables is needed. The nonlinearity satisfies $\|Q^2\beta_\ep\|_{C^1} \leq \|\beta\|_{C^1}$, and the classical Bernstein argument applies. Take $w := |\grad u|^2$ then $(\partial_t - \Delta)w \leq 2\|\beta'\|_\infty w + 2N\|\beta\|_\infty w^{1/2}$. Note also that $(\partial_t - \Delta)u^2 = -2Q^2\beta_\ep(u)u - 2w \leq -2w$. So that, with $\eta$ as in Step 3 and $\lambda$ large depending on $d$, $N$ and $\|\beta\|_{C^1}$, the function $z= \eta^2w + \lambda u^2 - \lambda t$ is a subsolution of the heat equation in $\Gamma$. Applying the maximum principle to $z$ gives the result. 
\end{proof}

\subsection{Attainment of the parabolic boundary data}\label{ss.attainment}

 In this subsection we prove that the semilinear approximations $\mathfrak{u}_\ep$ attain the boundary and initial data ${g}_+$ with a modulus which is uniform in $\ep$. This is key to applying the comparison principle, Theorem \ref{t.comparison}, in the limit.

\begin{proposition}[uniform attainment of the data]\label{p.attainment}
 Let $g_\ep$ be non-negative, $L$-Lipschitz on $\overline{U}$, and bounded above by $M$. Assume that $\mathfrak{u}_\ep \in C(\overline{U_\infty})$ is a classical solution of \eqref{e.bernoulli-parabolic-semilinear}. There is a modulus of continuity $\varpi$, independent of $\ep \in (0,1]$ and of $t$, such that for all $(x,t)$ and $(y,s)$ in $\overline{U_\infty}$
 \[ |\mathfrak{u}_\ep(x,t) - \mathfrak{u}_\ep(y,s)| \leq \varpi(|x-y|+|t-s|).\]
\end{proposition}

First we prove the spatial boundary regularity; this is where we need the domain regularity. Since $\partial U$ is $C^2$, the domain $U$ satisfies a uniform exterior ball condition: there is $\rho_0 = \rho_0(U) \in (0,1]$ such that for every $x_0 \in \partial U$ and every $\rho \in (0,\rho_0]$ there is $y_0 = y_0(x_0,\rho) \in \R^d$ with
\begin{equation}\label{e.exterior-ball}
|y_0 - x_0| = \rho, \  B_\rho(y_0) \cap U = \emptyset, \  \hbox{and} \  |x-y_0|^2 \geq \rho^2 + \tfrac12 |x-x_0|^2 \ \hbox{ for all } x \in \overline{U}.
\end{equation}
The last inequality follows, after possibly decreasing $\rho_0$, from the $C^2$ regularity of $\partial U$: writing $n$ for the exterior unit normal at $x_0$, every $x \in \overline{U}$ satisfies $(x-x_0)\cdot n \leq C_{\partial U} |x-x_0|^2$, so that
\[|x - y_0|^2 = \rho^2 + |x-x_0|^2 - 2\rho  (x-x_0)\cdot n \geq \rho^2 + (1 - 2\rho C_{\partial U})|x-x_0|^2,\]
and it suffices to take $\rho_0 \leq \tfrac{1}{4} C_{\partial U}^{-1}$.

\begin{lemma}\label{l.attainment-lateral}
Under the same hypotheses as \pref{attainment} there is a modulus of continuity $\varpi$, independent of $\ep \in (0,1]$ and of $t$, so that for all $x,y \in \overline{U}$ and any $t \geq 0$
\[ |\mathfrak{u}_\ep(x,t) - \mathfrak{u}_\ep(y,t)| \leq \varpi(|x-y|).\]
\end{lemma}

\begin{proof}[Proof of \lref{attainment-lateral}]
We aim to show that for every $x_0 \in \partial U$, $x \in \overline U$, and $t > 0$,
\[ |\mathfrak{u}_\ep(x,t) - g_\ep(x_0)| \leq \varpi(|x - x_0|).\]
Assuming we have this it is easy to derive the result by arguing as in \lref{lipschitz-initial} and combining the interior Lipschitz estimate with the boundary modulus.

We start with the lower bound. Essentially, we need to redo the standard barrier argument for harmonic functions, we will create a barrier which is a smooth strict subsolution of \eqref{e.bernoulli-parabolic} and then compose with $\Phi_{\ep,\theta}$ to create a subsolution of \eqref{e.bernoulli-parabolic-semilinear}. Let $x_0 \in \partial U$ and assume that $g(x_0)>0$, otherwise the lower bound is trivial. Let $\eta$ be a standard barrier for the Dirichlet problem at $x_0$. Specifically, say $x_0$ has an exterior ball of radius $\rho>0$ centered at $y_0$ and define $\eta(x) = 1-\rho^{d-1}|x - y_0|^{-(d-1)}$. Define $V := U \cap B_{3\rho}(x_0)$ and $\rho \leq |x - y_0| \leq 4\rho$ on $\overline V$. Then $\eta$ satisfies
\begin{enumerate}[label = (\roman*)]
\item\label{part.eta-superharmonic} $\Delta \eta  <0 $ in $\R^d \setminus \{y_0\}$,
\item\label{part.eta-gradient} $c \leq \rho|\grad \eta| \leq C$ on $\overline V$ and so $\eta(x)\leq C \rho^{-1}|x - x_0|$ on $\overline V$,
\item\label{part.quadratic-behavior} $\eta(x) \geq \tfrac12\big(1 - \tfrac{\rho^2}{|x-y_0|^2}\big) = \tfrac{|x-y_0|^2 - \rho^2}{2|x-y_0|^2} \geq \tfrac{|x-x_0|^2}{64 \rho^2}$ for $x \in \overline V$, using \eref{exterior-ball} and the elementary inequality $1 - r^{q} \geq \tfrac12(1-r)$ for $r \in [0,1]$, $q \geq \tfrac12$,
\item $\eta \geq 1-2^{-(d-1)} \geq \tfrac12$ on $U \cap \partial B_{3\rho}(x_0)$, since $|x - y_0| \geq |x-x_0| - \rho = 2\rho$ there.
\end{enumerate}
Let $0<\sigma < 1$ and define
\[ w(x) := g_\ep(x_0) - \sigma -A_\sigma \eta(x).\]
Our eventual goal is to show that $\mathfrak{u}_\ep(x,t) \geq w(x)$ via a comparison argument. If $\sigma \geq g_\ep(x_0)$ this is trivial, so we can assume that $\sigma < g_\ep(x_0)$. 

We will first show that $w$ is a smooth strict subsolution of \eqref{e.bernoulli-basic} in $\overline{U}$. Then we will compose with $w_\ep = \Phi_{\ep,\theta}(w)$ to get a strict subsolution of \eqref{e.bernoulli-parabolic-semilinear}. Then we will argue that $w_\ep \leq g_\ep$ on $\partial_P U_\infty$ in order to do a comparison argument and conclude.

From \ref{part.eta-superharmonic} $w$ is a smooth subharmonic function on $\overline{U}$. On $\overline{U}  \setminus B_{2\rho}(x_0)$
\begin{equation}\label{e.b2rho-outside}
w(x) \leq g_\ep(x_0) - \sigma - cA_\sigma \leq -c_{1/2} \leq 0 \ \hbox{ if } \ A_\sigma \geq C(1+M).
\end{equation}
Here $c_{1/2}$ is the constant from \lref{profile-sub}. In particular $w(x) \leq 0 \leq g_\ep(x)$ on $\overline{U}  \setminus B_{2\rho}(x_0)$. Next we claim that $w \leq g_\ep$ on $\overline{U}\cap B_{2\rho}(x_0)$.  Indeed by the Lipschitz bound ${g}_\ep(x) \geq {g}_\ep(x_0) - L|x-x_0|$, so it suffices that $\sigma + A_\sigma\eta(x) \geq L|x-x_0|$.  By \ref{part.quadratic-behavior}
\[\sigma + A_\sigma\eta(x) \geq \sigma + cA_\sigma \rho^{-2}|x-x_0|^2 \geq c\rho^{-1}\sqrt{\sigma A_\sigma} |x-x_0| \geq L|x-x_0|,\]
if we choose $A_\sigma \geq C L^2\rho^2/\sigma$.  Finally from \ref{part.eta-gradient} on $B_{3\rho}(x_0) \cap \overline{U}$
\[|\grad w| = A_\sigma|\grad \eta| \geq cA_\sigma\rho^{-1} \geq 2Q_{\max}\] 
 as long as we choose $A_\sigma \geq C\rho Q_{\max}$.

 Thus $w$ is a smooth strict subsolution of \eqref{e.bernoulli-basic} in $\overline{U}$ with $w \leq g_\ep$ on $\overline{U}$.  If $\ep \geq \sigma/4$ then, choosing also $A_\sigma \geq C\rho^2Q_{\max}^2\|\beta\|_\infty/\sigma$, $w$ is itself a classical subsolution of \eqref{e.bernoulli-parabolic-semilinear}. Then comparison directly gives $\mathfrak{u}_\ep \geq w$ in $U_\infty$. So assume from here forward that $\ep < \sigma/4$. 
 
  Take $\theta = 1/2$. Then $w_\ep(x) := (\Phi_{\ep,\theta}(w(x)))_+$ is a stationary viscosity subsolution of \eqref{e.bernoulli-parabolic-semilinear} for every $\ep \in (0,1]$, by the proof of \lref{well-prepared} (subsolution case). That proof uses only $\Delta w \geq 0$ and $\theta|\grad w|^2 \geq Q^2$ at points where $\ep s_0 < w < \ep$. By \lref{profile-sub}\ref{part.profile-sub-b} where $w(x) \geq \ep$ we get
 \[ w_\ep(x) = (\Phi_{\ep,\theta}(w(x)))_+ = w(x) \leq g_\ep(x).\]
 Where $w(x) \leq -c_\theta \ep$ then \lref{profile-sub}\ref{part.profile-sub-c} gives $w_\ep(x) = 0 \leq g_\ep(x)$.  Finally, suppose $-c_\theta\ep \leq w(x) \leq \ep$.
By \eqref{e.b2rho-outside} such a point lies in
$B_{2\rho}(x_0)$, and \ref{part.quadratic-behavior} gives
\[
\frac{A_\sigma}{64\rho^2}|x-x_0|^2
\leq A_\sigma\eta(x)
= g_\ep(x_0)-\sigma-w(x)
\leq g_\ep(x_0)+c_\theta\ep
\leq 2g_\ep(x_0).
\]
Since $\sigma<g_\ep(x_0)$, increasing the constant in
$A_\sigma\geq CL^2\rho^2/\sigma$ ensures that
$L|x-x_0|\leq g_\ep(x_0)/2$. Consequently, by the Lipschitz
continuity of $g_\ep$ and
\lref{profile-sub}\ref{part.profile-sub-c},
\[
w_\ep(x)\leq w(x)_+ + \ep
\leq 2\ep
\leq \frac{g_\ep(x_0)}{2}
\leq g_\ep(x_0)-L|x-x_0|
\leq g_\ep(x).
\]

 Thus we have shown that $w_\ep \leq g_\ep$ everywhere in $\overline{U}$, by comparison principle for \eqref{e.bernoulli-parabolic-semilinear} we conclude that $ \mathfrak{u}_\ep(x,t) \geq w_\ep(x,t) \geq w(x,t)$ on $U_\infty$.

 Then applying \ref{part.eta-gradient} we conclude, absorbing the extra $\rho^{-1}$ factor into the universal constant $C$,
 \[ \mathfrak{u}_\ep(x,t) \geq g_\ep(x_0) - \sigma - CA_\sigma |x-x_0| .\]
Taking $\varpi(r) := \inf_{\sigma \in (0,1)}\big(\sigma + CA_\sigma  r\big)$, which is a modulus of continuity, finishes the proof of the lower bound.

The upper bound is similar. Let $\eta$, $\sigma \in (0,1)$ and $A_\sigma$ be as in the lower bound above. By \ref{part.eta-superharmonic} and $\beta \geq 0$, $g_\ep(x_0)+\sigma+A_\sigma\eta$ is a stationary supersolution of \eqref{e.bernoulli-parabolic-semilinear}, and the estimates above which show $w \leq g_\ep$ on $\overline U$ show as well that $g_\ep(x_0)+\sigma+A_\sigma\eta \geq g_\ep$ on $\overline U$, hence 
\[g_\ep(x_0)+\sigma+A_\sigma\eta \geq \mathfrak{u}_\ep \ \hbox{ on } \ \partial_PU_\infty.\]
 By comparison and \ref{part.eta-gradient} this gives the upper bound inequality
\[\mathfrak{u}_\ep(x,t) \leq g_\ep(x_0)+\sigma+A_\sigma\eta(x) \leq g_\ep(x_0) + \sigma + CA_\sigma|x-x_0|. \]
As before take the same $\varpi(r) := \inf_{\sigma \in (0,1)}\big(\sigma + CA_\sigma  r\big)$ to conclude.
\end{proof}

The temporal regularity starts with $L^2$-in-time estimate that follows from the dissipation inequality.

\begin{lemma}\label{l.initial-L2}
Let $g_\ep$ be Lipschitz continuous and bounded on $\overline{U}$. For every $\ep \in (0,\ep_0]$ and $t ,s\geq 0$,
\[\|\mathfrak{u}_\ep(\cdot,t) - \mathfrak{u}_\ep(\cdot,s)\|_{L^2(U)} \leq \Big(\tfrac{1}{2}E_0  |t-s|\Big)^{1/2}.\]
\end{lemma}

\begin{proof}
Take $t>s$. By \eqref{eqn:dissipation} and \lref{well-prepared}(iv), $$\int_0^\infty\int_U (\partial_t \mathfrak{u}_\ep)^2  dx  dt \leq \tfrac12 J(g_\ep,\chi^0_\eps;U) \leq \tfrac12 E_0.$$ Since $\mathfrak{u}_\ep \in C([0,\infty);L^2(U))$ with $\partial_t \mathfrak{u}_\ep \in L^2(U\times(0,t))$, the fundamental theorem of calculus in $L^2$ and the Cauchy--Schwarz inequality give
\begin{multline}
\int_U |\mathfrak{u}_\ep(x,t) - \mathfrak{u}_\ep(x,s)|^2   dx = \int_U \Big|\int_s^t \partial_\tau \mathfrak{u}_\ep(x,\tau)  d\tau\Big|^2 dx \\
\leq (t-s) \int_s^t \int_U (\partial_\tau \mathfrak{u}_\ep)^2   dx  d\tau \leq \tfrac12 E_0  (t-s). 
\end{multline}
This includes the case $s=0$ where $\mathfrak{u}_\ep(\cdot,s)=g_\ep(\cdot)$.
\end{proof}

 Now we can easily prove the $L^\infty$ continuity as well.
 \begin{proof}[Proof of \pref{attainment}]
     Combining $L^2$ continuity in time \lref{initial-L2} with the uniform modulus of continuity in space from \lref{attainment-lateral} gives the continuity in $L^\infty$ norm with a uniform modulus independent of $t$.  The proof is elementary, if $\|\mathfrak{u}_\ep(\cdot,t) - \mathfrak{u}_\ep(\cdot,s)\|_{L^\infty(U)} =: \delta>0$ then $|\mathfrak{u}_\ep(\cdot,t) - \mathfrak{u}_\ep(\cdot,s)| \geq \frac{\delta}{2}$ in a $\varpi^{-1}(\delta/2)$ radius ball centered somewhere in $\overline{U}$.  Plugging into \lref{initial-L2} gives the bound $C|t-s|^{1/2} \geq c\varpi^{-1}(\delta/2)^{d/2}\delta$ and inverting this relation gives the modulus of continuity, up to updating $\varpi$ again.
 \end{proof}

\subsection{No outward jump of the positive phase at $t=0$}\label{s.no-outward-jump} \pref{attainment} controls the values of $\mathfrak{u}_\ep$ at the parabolic boundary. We also need continuity of the positivity set at the parabolic boundary for the comparison principle.  This is only an issue for the monotone increasing case since continuity in time of $\mathfrak{u}_\ep$ is sufficient to control the inwards propagation of the free boundary.

We need to show that the positivity set of the limit does not immediately jump outwards at the initial time.  If a ball $B_r(x_0)$ is contained in $\{g \leq 0\}$ then the continuity of $\mathfrak{u}_\ep$ at time $t=0$ gives that $\mathfrak{u}_\ep(x,t) \leq \ep+\varpi(t)$ on $B_r(x_0)$. Then a stationary supersolution barrier, built by composing a radial harmonic function of small slope with the profile $\Psi_{\ep,\hat\theta}$ of \lref{profile-super} bounds the propagation of the positive phase.

\begin{lemma}\label{l.positivity-trace}
Let $g_\ep$ be Lipschitz continuous and bounded on $\overline{U}$ and let $\mathfrak{u}_\ep$ be the solution of \eqref{e.bernoulli-parabolic-semilinear}. Suppose that $x_0 \in \overline{U}$, $r \leq 1$, and $B_{3r}(x_0) \cap \overline{U} \subset \{g_\ep \leq (1-\mu)\ep\}$ for some $ \mu \in (0,1]$.  Then there exists $\tau>0$ depending only on $r$, $\mu$, and universal parameters so that for every $\kappa \in (1-\mu,1]$ there is $\ep_0>0$, depending also on $\kappa$, with $\overline{B_{r}(x_0)\cap U} \subset \{\mathfrak{u}_\ep(\cdot,t) < \kappa\ep\}$ for all $ 0 \leq t \leq \tau$ and $0 < \ep \leq \ep_0$.
\end{lemma}

First we apply this to the $\ep \to 0$ limit in the monotone increasing case.

\begin{corollary}\label{c.positivity-no-jump}
Let $g$ be a smooth strict subsolution of \eqref{e.bernoulli-basic}, $g_\ep = (\Phi_{\ep,\theta}(g))_+$ as in \lref{well-prepared}, and $\mathfrak{u}_\ep$ be the corresponding solution of \eqref{e.bernoulli-parabolic-semilinear}. Suppose that $x_0 \in \overline{U}$, $r \leq 1$, and $B_{4r}(x_0) \cap \overline{U} \subset \{g < 0\}$.  Then there exists $\tau>0$ depending only on $r$ and universal parameters so that for every $\kappa \in (0,1]$, $\mathfrak{u}_\ep < \kappa\ep$ on $(\overline{B_r(x_0) \cap U}) \times [0,\tau]$ for all $\ep>0$ sufficiently small, depending on $\kappa$. In particular
\[\big((B_r(x_0)\cap\overline{U}) \times [0,\tau)\big) \cap \overline{\bigcup_{0<\kappa\le 1}{\limsup_{\ep \to 0}}^*\{\mathfrak{u}_\ep>\kappa\ep\}} = \emptyset.\]
\end{corollary}
\begin{proof}
Since $\overline{B_{3r}(x_0) \cap U}$ is compact and contained in $\{g<0\}$, there is $c_0>0$ with $g \leq -c_0$ there. By \lref{well-prepared}\ref{part.gep-5} $\{g_\ep>0\} \subset \{g > -c_\theta \ep\}$, so for $\ep>0$ sufficiently small $B_{3r}(x_0) \cap \overline{U} \subset \{g_\ep = 0\}$. Then \lref{positivity-trace} applies with $\mu = 1$ and there exists $\tau>0$ depending on $r>0$ and universal parameters so that, for every $\kappa \in (0,1]$, $(\overline{B_r(x_0) \cap U}) \times [0,\tau] \subset \{\mathfrak{u}_\ep < \kappa\ep\}$ for $\ep$ sufficiently small. Hence each ${\limsup}^*_{\ep\to0}\{\mathfrak{u}_\ep > \kappa\ep\}$ misses the relatively open set $(B_r(x_0)\cap\overline{U}) \times [0,\tau)$, and so does the closure of their union. Covering a compact $K \subset\subset \overline{U} \cap \{g<0\}$ by finitely many such balls proves Proposition \ref{p.semilinear-wellposed}\ref{part.trace}.
\end{proof}

Now we return to the proof of \lref{positivity-trace}.

\begin{proof}[Proof of \lref{positivity-trace}]
By the definition of $\kappa_{\hat \theta}$ in \eqref{e.kappa-def}, and \lref{profile-super}\ref{part.profile-super-f}, $\hat\theta \mapsto \kappa_{\hat\theta}$ is a continuous increasing bijection from $(1,\infty)$ onto $(0,1)$. Let $\Theta>1$, depending only on $\mu$, be such that $\kappa_{\Theta} = 1-\frac{1}{2}\mu$. Given $\kappa \in (1-\mu,1]$ we choose $\hat\theta \in (1,\Theta]$ so that $1-\mu < \kappa_{\hat \theta} < \kappa$.

Let $\eta$ be a radial strictly superharmonic function in $B_3 \setminus B_{3/2}$ which is zero on $\partial B_{2}$ and $1$ on $\partial B_3$. Extend $\eta$ in a smooth monotone radial way to be equal to $-1$ on $B_{1}$. Call $c_0(d)$ to be supremum of $|\grad \eta|$ on $B_{3} \setminus B_{3/2}$. Let $\gamma>0$ sufficiently small so that $\gamma c_0 \leq \Theta^{-1/2}Q_{\min}$. Note that $\gamma$ depends on $\mu$ but not on $\kappa$. Then call $\hat\eta(x) = \gamma r\eta(\frac{x-x_0}{r})$, this is a smooth strict supersolution of \eqref{e.bernoulli-basic} in $B_{3r}$ with parameters $\delta_0 = 1-\frac{1}{\hat\theta}$ and $a_0 := -\gamma r\,\eta|_{\partial B_{3/2}} > 0$. Then $\hat\theta = \frac{1}{1-\delta_0}$ so that the hypothesis of \lref{well-prepared} is satisfied.

Then define $\hat \eta_\ep = \Psi_{\ep,\hat\theta}(\hat \eta)$. By \lref{well-prepared} we conclude that $\hat \eta_\ep$ is a supersolution of \eqref{e.bernoulli-parabolic-semilinear} in $B_{3r}(x_0)$ for $\ep$ sufficiently small, depending on $\hat\theta$. By choice of $\hat\theta$ and \lref{profile-super} we have 
\[\hat \eta_\ep > \kappa_{\hat \theta}\ep > (1-\mu)\ep \geq g_\ep \ \hbox{ on } \ B_{3r}(x_0) \cap \overline{U}\]
 and on $\partial U$ this gives $\hat\eta_\ep > g_\ep = \mathfrak{u}_\ep$.  On $\partial B_{3r}(x_0)$ we have $\hat\eta_\ep \geq \hat\eta = \gamma r$ since $\Psi_{\ep,\hat\theta}(s) \geq s$. By \pref{attainment} 
\[\mathfrak{u}_\ep(x,t) \leq g_\ep(x) + \varpi(t) \leq \ep + \varpi(\tau) < \gamma r \ \hbox{ on } \ \partial B_{3r}(x_0) \times [0,\tau]\]
 for $\ep < \frac{1}{2}\gamma r$, and $\tau$ sufficiently small so that $\varpi(\tau) \leq \frac{1}{2}\gamma r$.  Since $\gamma>0$ depends only on $\mu$, $\tau$ depends only on $r$ and $\mu$, and not on $\kappa$.  
 
 Thus comparison principle for \eqref{e.bernoulli-parabolic-semilinear} on $(B_{3r}(x_0) \cap U) \times (0,\tau]$ yields $\mathfrak{u}_\ep(x,t) \leq \hat \eta_\ep(x)$ in $(B_{3r}(x_0)\cap U) \times (0,\tau]$. Since $\hat \eta = - \gamma r$ in $B_{r}(x_0)$ and using \lref{profile-super}\ref{part.profile-super-e} we get $\hat \eta_\ep  \leq \kappa_{\hat\theta} \ep+C_{\hat\theta}\ep e^{-c_{\hat\theta}\gamma r/\ep} < \kappa\ep$ for $\ep>0$ sufficiently small depending on $\gamma$, $r$, and $\kappa$ (through $\hat\theta$).
\end{proof}

\bibliography{articles.bib}
\end{document}